\documentclass[11pt]{amsart}
\usepackage[margin=1in]{geometry}
\usepackage{latexsym}

\usepackage[utf8]{inputenc}
\usepackage{amsfonts}
\usepackage{amsmath}
\usepackage{amssymb}
\usepackage{amsthm}
\usepackage{commath}
\usepackage{subfig}
\usepackage{float}
\usepackage{mathrsfs}
\usepackage[colorlinks]{hyperref}
\usepackage[margin=0.5cm]{caption}
\hypersetup{
	linkcolor=blue,
	citecolor=green,
	filecolor=Red,
	urlcolor=NavyBlue,
	menucolor=Red,
	runcolor=Red}
\usepackage{thmtools}
\usepackage{thm-restate}
\usepackage{cleveref}
\numberwithin{equation}{section}
\allowdisplaybreaks
\usepackage{caption}
\usepackage[dvipsnames]{xcolor}

\usepackage[style=alphabetic, giveninits, maxbibnames=20]{biblatex}
\renewbibmacro{in:}{}

\usepackage{todonotes}
\usepackage{amsmath}

\newcommand{\Acyc}{\textsf{\textup{Acyc}}}
\newcommand{\FS}{\mathsf{FS}}

\newcommand{\Path}{\textsf{\textup{Path}}}
\newcommand{\Cycle}{\textsf{\textup{Cycle}}}

\newcommand{\Star}{\textsf{\textup{Star}}}
\newcommand{\diam}{\textup{diam}}
\newcommand{\inv}{\textup{inv}}

\newcommand{\bd}{\textup{bd}}

\newtheorem{theorem}{Theorem}[section]
\newtheorem{lemma}[theorem]{Lemma}
\newtheorem{proposition}[theorem]{Proposition}
\newtheorem{corollary}[theorem]{Corollary}

\newtheorem{conjecture}[theorem]{Conjecture}
\newtheorem{question}[theorem]{Question}
\newtheorem{problem}[theorem]{Problem}

\DeclareFieldFormat[article]{title}{#1}

\theoremstyle{definition}
\newtheorem{definition}[theorem]{Definition}
\newtheorem{remarkx}[theorem]{Remark}
\newenvironment{remark}
  {\pushQED{\qed}\remarkx}
  {\popQED\endremarkx}

\begin{document}

\title{On the Diameters of Friends-and-Strangers Graphs}

\author{Ryan Jeong}
\address{Department of Pure Mathematics and Mathematical Statistics, Wilberforce Road, Cambridge, CB3 0WA, UK}
\email{rj450@cam.ac.uk}

\maketitle

\begin{abstract}
    Given simple graphs $X$ and $Y$ on the same number of vertices, the friends-and-strangers graph $\FS(X, Y)$ has as its vertices all bijections from $V(X)$ to $V(Y)$, where two bijections are adjacent if and only if they differ on two adjacent elements of $V(X)$ with images adjacent in $Y$. We study the diameters of connected components of friends-and-strangers graphs: the diameter of a component of $\FS(X,Y)$ corresponds to the largest number of swaps necessary to go from one configuration in the component to another. We show that any component of $\FS(\Path_n, Y)$ has $O(n^2)$ diameter and that any component of $\FS(\Cycle_n, Y)$ has $O(n^4)$ diameter, improvable to $O(n^3)$ whenever $\FS(\Cycle_n, Y)$ is connected. These results address an open problem posed by Defant and Kravitz. Using an explicit construction, we show that there exist $n$-vertex graphs $X$ and $Y$ such that $\FS(X,Y)$ has a component with $e^{\Omega(n)}$ diameter. This answers a question raised by Alon, Defant, and Kravitz in the negative. As a corollary, we observe that for such $X$ and $Y$, the lazy random walk on this component of $\FS(X,Y)$ has $e^{\Omega(n)}$ mixing time. This result deviates from related classical theorems regarding rapidly mixing Markov chains and makes progress on another open problem of Alon, Defant, and Kravitz. We conclude with several suggestions for future research.
\end{abstract}

\section{Introduction}

\subsection{Background and Motivation}

Let $X$ and $Y$ be $n$-vertex simple graphs. Interpret the vertices of $X$ as positions, and the vertices of $Y$ as people: say two people in the vertex set of $Y$ are friends if they are adjacent and strangers if they are not. Each person picks a position to stand on, yielding a starting configuration. From here, at any point in time, two friends standing on adjacent positions may switch places: we call this operation a friendly swap. From the initial configuration, say the $n$ people have a final configuration in mind, and they know it can be reached from the initial configuration by some sequence of friendly swaps. What is the worst-case (over pairs of starting and final configurations) number of friendly swaps that is necessary in order for the $n$ people to achieve the final configuration from the starting configuration? 

We may formalize the problem using the following definition.

\begin{definition}[\cite{defant2021friends}] \label{defn:fs_def}
Let $X$ and $Y$ be simple graphs on $n$ vertices. The \textit{\textcolor{red}{friends-and-strangers graph}} of $X$ and $Y$, denoted $\FS(X, Y)$, is a graph with vertices consisting of all bijections from $V(X)$ to $V(Y)$, with bijections $\sigma, \tau \in \FS(X,Y)$ adjacent if and only if there exists an edge $\{a, b\}$ in $X$ such that
\begin{enumerate}
    \item $\{\sigma(a), \sigma(b)\} \in E(Y)$,
    \item $\sigma(a) = \tau(b), \ \sigma(b) = \tau(a)$,
    \item $\sigma(c) = \tau(c)$ for all $c \in V(X) \setminus \{a, b\}$.
\end{enumerate}
In other words, $\sigma$ and $\tau$ differ precisely on two adjacent vertices of $X$ whose images under $\sigma$ (and $\tau$) are adjacent in $Y$. For any such bijections $\sigma, \tau$, we say that $\tau$ is achieved from $\sigma$ by an \textit{\textcolor{red}{$(X, Y)$-friendly swap}}.
\end{definition}

\begin{figure}[ht]
\begin{minipage}{.5\linewidth}
\centering
\subfloat[The graph $X$.]{\label{fig:X_gph}\includegraphics[width=0.48\textwidth]{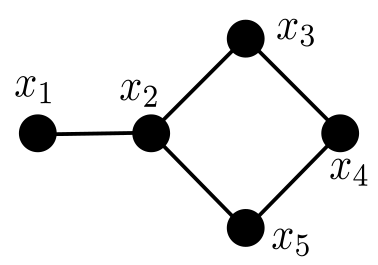}}
\end{minipage}%
\hfill
\begin{minipage}{.5\linewidth}
\centering
\subfloat[The graph $Y$.]{\label{fig:Y_gph}\includegraphics[width=0.48\textwidth]{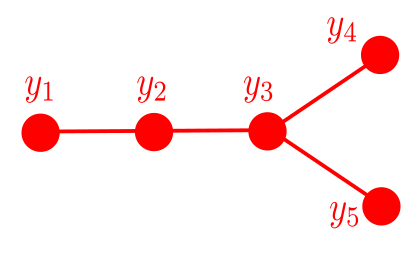}}
\end{minipage}\par\medskip
\centering
\subfloat[A sequence of $(X, Y)$-friendly swaps. The transpositions between adjacent configurations denote the two vertices in $X$ over which the $(X, Y)$-friendly swap takes place. Red text corresponds to vertices in $Y$ placed upon vertices of $X$, in black text: using colored text for vertices in $Y$ to distinguish them from vertices in $X$ in black text will be a convention throughout the work. The leftmost configuration corresponds to the bijection $\sigma$ in the vertex set of $\FS(X, Y)$ that maps $\sigma(x_1) = y_1$, $\sigma(x_2) = y_5$, $\sigma(x_3) = y_3$, $\sigma(x_4) = y_4$, and $\sigma(x_5) = y_2$. The other configurations correspond analogously to vertices in $\FS(X, Y)$.]{\label{fig:friendly_swaps}\includegraphics[width=.99\textwidth]{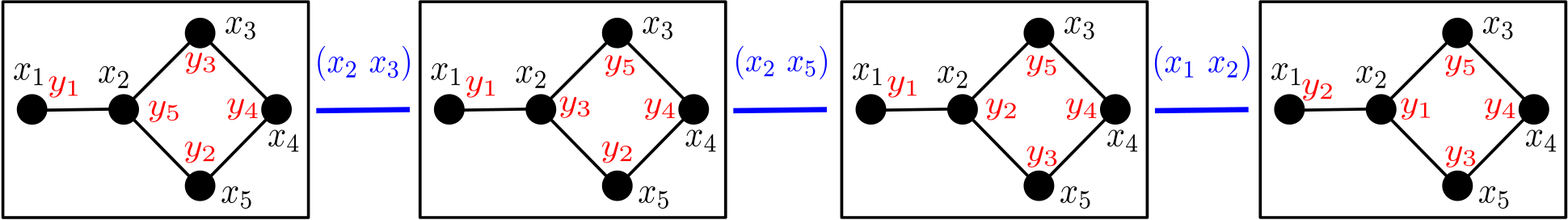}}

\caption{A sequence of $(X, Y)$-friendly swaps in $\FS(X, Y)$ for the $5$-vertex graphs $X$ and $Y$. Configurations in the bottom row correspond to vertices in $V(\FS(X, Y))$. Two consecutive configurations differ by an $(X, Y)$-friendly swap, so the corresponding vertices are adjacent in $\FS(X, Y)$.}
\label{fig:defn}
\end{figure}
See Figure \ref{fig:defn} for an illustration of Definition \ref{defn:fs_def} on five-vertex graphs. Defant and Kravitz \cite{defant2021friends}, in addition to introducing the framework of friends-and-strangers graphs, described the connected components of $\FS(\Path_n, Y)$ and $\FS(\Cycle_n, Y)$ in terms of the acyclic orientations of $\overline{Y}$ (the complement of $Y$), and determined both necessary conditions and sufficient conditions for $\FS(X, Y)$ to be connected. In a different paper \cite{jeong2022structural}, we extend their results: \cite[Corollary 4.14]{defant2021friends} states that $\FS(\Cycle_n, Y)$ is connected if and only if $\overline{Y}$ is a forest with trees of jointly coprime sizes, and we establish that if $X$ is biconnected (i.e., connected and with no cut vertex) and $Y$ is a graph for which $\FS(\Cycle_n, Y)$ is connected, then $\FS(X, Y)$ is connected, settling \cite[Conjecture 7.1]{defant2021friends}. In \cite{jeong2022structural}, we also initiate the study of the girth of friends-and-strangers graphs. Motivated by \cite{kornhauser1984pebble} and connections to molecular programming as seen in \cite{brailovskaya2019reversible}, the framework of friends-and-strangers was later generalized by \cite{milojevic2023connectivity} to permit for multiplicities onto vertices, in which many of the main results of \cite{defant2021friends, wilson1974graph} were also generalized accordingly.

A central objective in the study of friends-and-strangers graphs is to determine necessary and sufficient conditions for their connectivity. Indeed, $\FS(X,Y)$ being connected corresponds exactly to the property that one can go between any two configurations in $\FS(X,Y)$ via some sequence of $(X,Y)$-friendly swaps. Of course, the conditions one may derive will depend upon the assumptions on $X$ and $Y$ under which one works. If one elects to proceed under a regime in which $\FS(X,Y)$ cannot be connected (such as when $X$ and $Y$ are both bipartite; see the discussion around \cite[Proposition 2.7]{defant2021friends} and \cite[Subsection 2.3]{alon2022typical} for a parity obstruction which demonstrates why this is the case), one may instead study how small the number of connected components may be under this regime, and the natural question here is to ask for further conditions on $X$ and $Y$ ensuring that $\FS(X,Y)$ achieves the smallest possible number of connected components. As pursued in \cite[Sections 3 and 4]{defant2021friends} for (respectively) paths and cycles, one direction of inquiry is to fix (without loss of generality, as we will see in Proposition \ref{prop:basic_properties}(1)) $X$ to be some particular graph, and study the structure of $\FS(X, Y)$ for arbitrary $Y$: see \cite{defant2022connectedness, lee2022connectedness, wang2023connectivity, wilson1974graph, zhu2023evacuating}. It is also very natural to ask extremal and probabilistic questions concerning the connectivity of friends-and-strangers graphs, such as minimum degree conditions on $X$ and $Y$ which ensure that $\FS(X,Y)$ is connected or for threshold probabilities on Erd\H{o}s-R\'enyi random graphs $X,Y$ regarding the connectivity of $\FS(X,Y)$: see \cite{alon2022typical, bangachev2022asymmetric, jeong2023bipartite, milojevic2023connectivity, wang2022connectedness}.

The setup proposed by Definition \ref{defn:fs_def} is quite general. Indeed, friends-and-strangers graphs serve both as a common natural generalization of many classical combinatorial objects and as a framework which embodies many important problems in discrete mathematics and theoretical computer science. We illustrate this claim with a non-exhaustive listing of relevant examples. The graph $\FS(X, K_n)$ is the Cayley graph of the symmetric group on the vertex set of $X$ generated by the transpositions corresponding to the edges of $X$; we refer the reader to \cite{defant2021friends} and the references therein for a comprehensive discussion regarding the relevance of friends-and-strangers graphs within algebraic combinatorics. Letting $X$ be the $4$-by-$4$ grid and $Y$ a star graph, studying $\FS(X, Y)$ is equivalent to studying the configurations and moves that can be performed on the famous $15$-puzzle (with the central vertex of the star graph corresponding to the empty tile); see \cite{brunck2023books, demaine2018simple, parberry2015solving, yang2011sliding} for similar inquiries of a recreational flavor. The works \cite{naatz2000graph, stanley2012equivalence} both study the structure of the graph $\FS(\Path_n, Y)$ under certain restrictions on $Y$, while the works \cite{barrett1999elements, reidys1998acyclic} utilize $\FS(\Path_n, Y)$ to investigate the acyclic orientations of $\overline{Y}$. Asking if $X$ and $Y$ pack \cite{bollobas1978packings, kuhn2009minimum, sauer1978edge, yap1988packing, yuster2007combinatorial} in the graph packing literature is equivalent to asking if there exists an isolated vertex in $\FS(X,Y)$. Studying the token swapping problem \cite{aichholzer2021hardness, biniaz2023token, bonnet2018complexity, miltzow_et_al:LIPIcs:2016:6408, yamanaka2015swapping} on the graph $X$ is equivalent to studying distances between configurations in $\FS(X, K_n)$. Finally, as we will briefly touch upon in Subsection \ref{subsec:proof_of_lower_bound}, the interchange process on the graph $X$ \cite{aldous2002reversible, alon2013probability, angel2003random, berestycki2006phase, caputo2010proof, elboim2023infinite, hammond2015sharp, hermon2021interchange, schramm2005compositions} can be phrased in terms of (continuous-time) random walks on $\FS(X, K_n)$.

\subsection{Main Results and Organization} Unlike the existing body of work that studies the connectivity of friends-and-strangers graphs, the present paper initiates the study of their diameters, corresponding to the length of the ``longest shortest path," with lengths of shortest paths evaluated over all pairs of vertices. Indeed, the diameter of a connected component of $\FS(X, Y)$ corresponds to the largest number of $(X,Y)$-friendly swaps necessary to achieve one configuration in the component from another. In a more recreational tone, if we think of $\FS(X,Y)$ as a generalized $15$-puzzle, we are asking for the longest solution length for any solvable puzzle involving ``board $X$ and rules $Y$." The works \cite{alon2022typical, defant2021friends} both posed the following question, which asks whether the distance between any two configurations in $\FS(X,Y)$ is polynomial in the size of $X$ and $Y$.

\begin{question} [\cite{alon2022typical, defant2021friends}] \label{ques:poly_bdd}
Does there exist an absolute constant $C > 0$ such that for all $n$-vertex graphs $X$ and $Y$, every connected component of $\FS(X, Y)$ has diameter at most $n^C$?
\end{question}

In Section \ref{sec:background}, we introduce some background that we shall need later in the work. Before tackling the more global Question \ref{ques:poly_bdd}, in Section \ref{sec:fixed_one_graph}, we fix (without loss of generality) $X$ to be a complete, path, or cycle graph, and derive upper bounds on the maximum diameter of a component of $\FS(X,Y)$ in each setting. Our results on paths and cycles address an open problem posed in \cite[Subsection 7.3]{defant2021friends}. Furthermore, the discussion therein suggests that one must restrict their attention to rather contrived choices of graphs $X$ and $Y$ in order for $\FS(X,Y)$ to have a component with diameter that is superpolynomial in the size of $X$ and $Y$, suggesting that Question \ref{ques:poly_bdd} may be challenging to settle via constructive means if it holds in the negative.

In Section \ref{sec:large_diameter}, we establish the main result of this article, Theorem \ref{thm:main_diam_result}, which answers Question \ref{ques:poly_bdd} in the negative. We prove this theorem by constructing, for all integers $L \geq 1$, graphs $X_L$ and $Y_L$ on the same number of vertices: see Figure \ref{fig:demos} for a schematic diagram of the construction for $L=3$. The construction is such that the number of vertices of $X_L$ and $Y_L$ is $\Theta(L)$, and there exist two configurations $\sigma_s, \sigma_f \in V(\FS(X_L, Y_L))$ which lie in the same connected component $\mathscr{C}$ of $\FS(X_L, Y_L)$ and for which the distance between $\sigma_s$ and $\sigma_f$ is $e^{\Omega(n)}$. 

\begin{restatable}{theorem}{mainDiamResult} \label{thm:main_diam_result}
    For all $n \geq 2$, there exist $n$-vertex graphs $X$ and $Y$ such that $\FS(X,Y)$ has a connected component with diameter $e^{\Omega(n)}$.
\end{restatable}

\begin{figure}[ht]
  \centering
  \subfloat[The graph $X_3$.]{\includegraphics[width=0.495\textwidth]{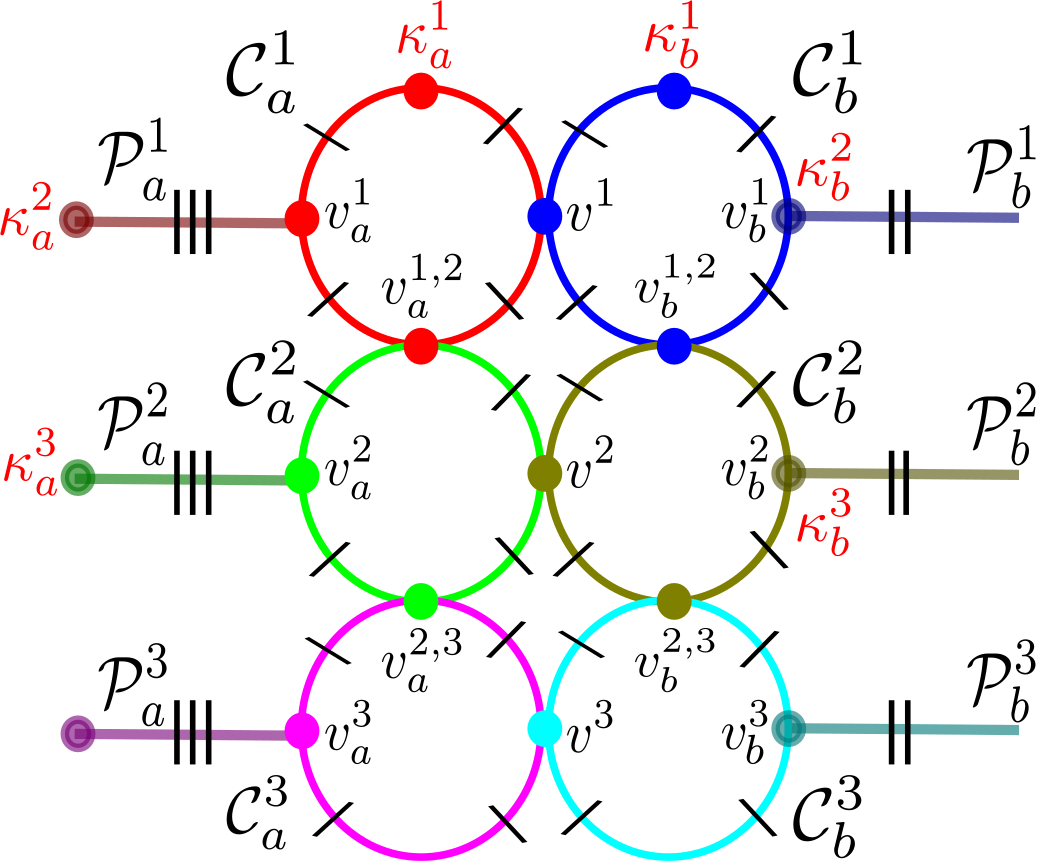}\label{fig:demo_X3}}
  \hfill
  \subfloat[The graph $Y_3$.]{\includegraphics[width=0.495\textwidth]{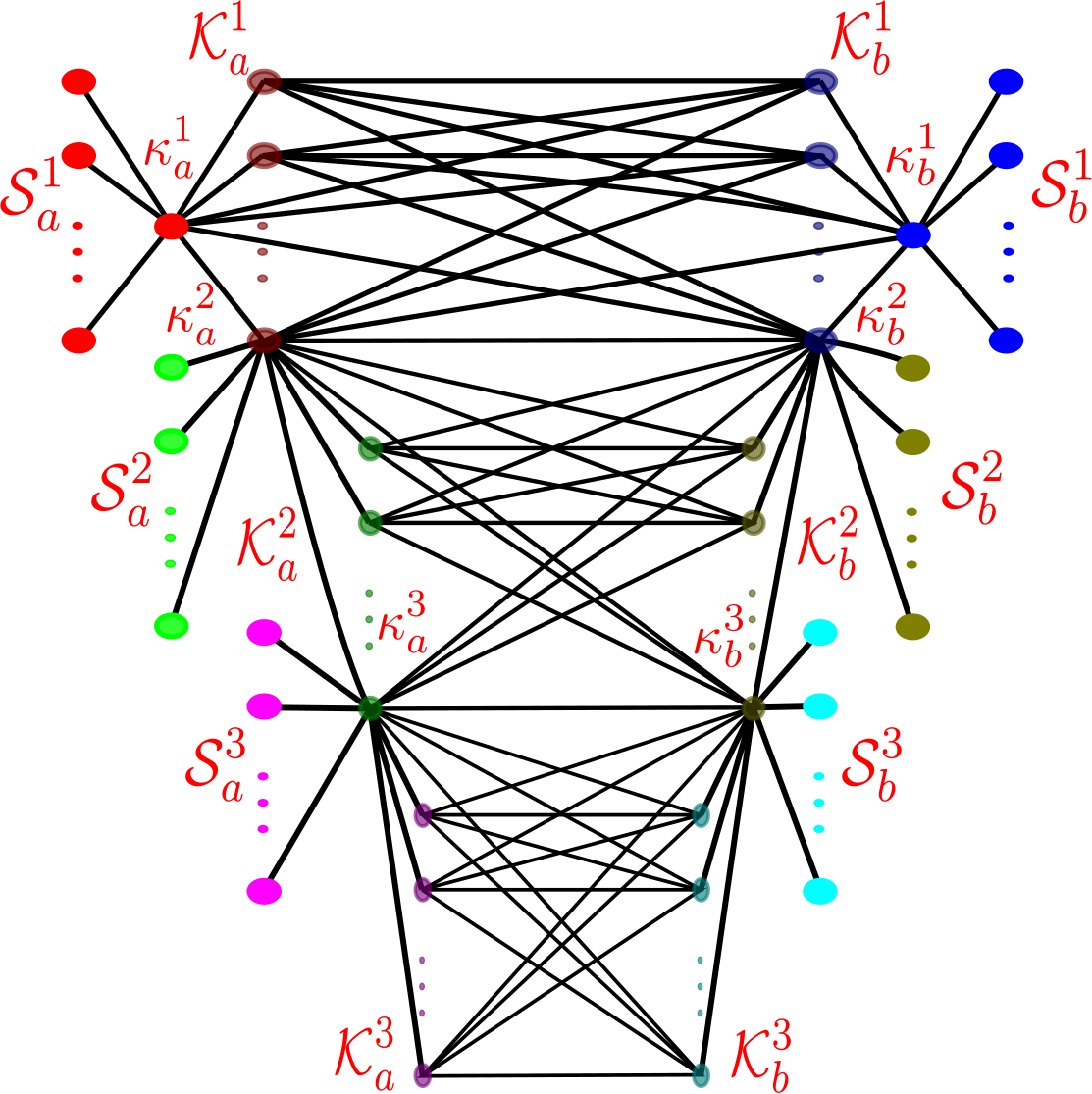}\label{fig:demo_Y3}}
  \caption{The graphs $X_3$ and $Y_3$.}
  \label{fig:demos}
\end{figure}

At the end of Section \ref{sec:large_diameter}, we briefly discuss implications of Theorem \ref{thm:main_diam_result} to the study of random walks on friends-and-strangers graphs, and deduce what might be thought of as the natural stochastic analogue of Theorem \ref{thm:main_diam_result}. Random walks on friends-and-strangers graphs model a variant of the interchange process where we may prohibit certain pairs of particles from swapping positions. This result contrasts many classical theorems regarding rapidly mixing Markov chains, all of which may be readily rewritten using the language of friends-and-strangers graphs, and makes progress on another open problem posed in \cite[Section 7]{alon2022typical}. We conclude the work with Section \ref{sec:future_directions}, which suggests several open problems and directions for future research.

\subsection{Notation} \label{subsec:notation}

In this article, unless stated otherwise, we assume that all graphs are simple. We employ standard asymptotic notation in this paper. Unless stated otherwise, all asymptotic notation in this paper will be with respect to $n$. Purely for the sake of completeness, we state the following standard notation.
\begin{itemize}
    \item The vertex and edge sets of a graph $G$ are denoted by $V(G)$ and $E(G)$, respectively.
    \item The complement of the graph $G$ is denoted $\overline{G}$.
    \item The statement that $G$ and $H$ are isomorphic is written as $G \cong H$.
    \item For a subset $S \subset V(G)$, we let $G|_{S}$ denote the induced subgraph of $G$ with vertex set $S$.
    \item The open neighborhood of $v \in V(G)$, which is the collection of all neighbors of $v$, is denoted by $N_G(v)$. The closed neighborhood of $v \in V(G)$ is denoted $N_G[v] = N_G(v) \cup \{v\}$. For a subset of vertices $S \subseteq V(G)$, we let
    \begin{align*}
        & N_G(S) = \bigcup_{v \in S} N_G(v), 
        & N_G[S] = \bigcup_{v \in S} N_G[v].
    \end{align*}
    \item The disjoint of a collection of graphs $\{G_i\}_{i \in I}$, notated $\bigoplus_{i \in I} G_i$, is the graph with vertex set $\bigsqcup_{i \in I} V(G_i)$ and edge set $\bigsqcup_{i \in I} E(G_i)$. This readily extends to expressing a graph as the disjoint union of its connected components.
    \item The distance $d(v,w)$ between $v, w \in V(G)$ is the length of the shortest path from $v$ to $w$. The diameter of a component $\mathscr{C}$ of $G$ is $\max_{v,w \in V(\mathscr{C})} d(v,w)$.
\end{itemize}
Graphs with vertex set $[n] := \{1, \dots, n\}$ that will be relevant later are
\begin{itemize}
    \item the complete graph $K_n$, with $E(K_n) := \{\{i, j\} : \{i, j \in [n], \ i \neq j\}\}$;
    \item the complete bipartite graph $K_{i,j}$, with $E(K_{i,j}) := \{\{v_1, v_2\}: v_1 \in [i], v_2 \in \{i+1, \dots, i+j\}\}$, which naturally partitions $V(K_{i,j})$ into two sets (henceforth called partite sets);
    \item the path graph $\Path_n$, with $E(\Path_n) := \{\{i, i+1\} : i \in [n-1]\}$;
    \item the cycle graph $\Cycle_n$, with $E(\Cycle_n) := \{\{i, i+1\} : i \in [n-1]\} \cup \{\{n, 1\}\}$;
    \item the star graph $\Star_n := K_{1,n-1}$.
\end{itemize}

\section{Background} \label{sec:background}

In this section, we introduce some background and summarize results from prior work that will be relevant later in the present paper, particularly in Section \ref{sec:fixed_one_graph}. Throughout this section and Section \ref{sec:fixed_one_graph}, we will assume that the vertex set of all graphs is $[n]$, with edge sets as in Subsection \ref{subsec:notation}. Note that if both $V(X)$ and $V(Y)$ are $[n]$, then the vertices of $\FS(X,Y)$ are the elements of $\mathfrak{S}_n$, the symmetric group of degree $n$.

\subsection{Acyclic Orientations} \label{subsec:acyclic_orientations}

An \textit{\textcolor{red}{orientation}} of a graph $G$ is an assignment of a direction to every edge of $G$, and an \textit{\textcolor{red}{acyclic orientation}} of $G$ is an orientation with no directed cycles. Denote the set of all acyclic orientations of $G$ by $\Acyc(G)$. We will be interested in operations on acyclic orientations of $G$ called \textit{\textcolor{red}{flips}} and \textit{\textcolor{red}{double-flips}}, as defined in \cite{defant2021friends}. Notably, it was shown in \cite[Theorem 4.7]{defant2021friends} that double-flips on acyclic orientations in $\Acyc(\overline{Y})$ are paramount in describing the connected components of $\FS(\Cycle_n, Y)$.

Letting $\alpha \in \Acyc(G)$, converting a source of $\alpha$ into a sink or a sink of $\alpha$ into a source by reversing the directions of all its incident edges results in another acyclic orientation $\alpha'$ of $G$. We call such an operation a \textit{\textcolor{red}{flip}}, and we say that $\alpha$ and $\alpha'$ are \textit{\textcolor{red}{flip equivalent}}, denoted $\alpha \sim \alpha'$. In the literature, the equivalence classes in $\Acyc(G)/\!\!\sim$ are called \textit{\textcolor{red}{toric acyclic orientations}}; we refer the interested reader to \cite{chen2010orientations, develin2016toric, macauley2011posets, pretzel1986reorienting, speyer2009powers} for related reading. We will further say that we perform an \textit{\textcolor{red}{inflip}} on $\alpha$ if we convert a source into a sink (the direction of all incident edges ``go into" the new sink), and an \textit{\textcolor{red}{outflip}} if we convert a sink into a source.\footnote{In particular, we may apply these operations to isolated vertices.} 

Similarly, flipping a nonadjacent source and sink of $\alpha$ into (respectively) a sink and a source results in another acyclic orientation $\alpha''$ of $G$: we call such an operation a \textit{\textcolor{red}{double-flip}}, and we say $\alpha$ and $\alpha''$ are \textit{\textcolor{red}{double-flip equivalent}}, denoted $\alpha \approx \alpha''$. It is easy to show that $\sim$ and $\approx$ are equivalence relations on $\Acyc(G)$. We denote the set of double-flip equivalence classes of $\Acyc(G)$ by $\Acyc(G)/\!\!\approx$, and denote the \textit{\textcolor{red}{double-flip equivalence class}} for which $\alpha$ is a representative by $[\alpha]_\approx$.

Assume $V(G) = [n]$, and take $\alpha \in \Acyc(G)$. Associated to the acyclic orientation $\alpha$ is a poset $P_\alpha = ([n], \leq_\alpha)$, where $i \leq_\alpha j$ if and only if there exists a directed path from $i$ to $j$ in $\alpha$. We define a \textit{\textcolor{red}{linear extension of $P_\alpha$}} to be any permutation $\sigma \in \mathfrak{S}_n$ such that $\sigma^{-1}(i) \leq \sigma^{-1}(j)$ whenever $i \leq_\alpha j$. We let $\mathcal L(\alpha)$ denote the collection of linear extensions of $P_\alpha$. For any $\sigma \in \mathfrak{S}_n$, it is not hard to see that there exists a unique acyclic orientation $\alpha_{G}(\sigma) \in \Acyc(G)$ for which $\sigma \in \mathcal L(\alpha_{G}(\sigma))$, and that this acyclic orientation is the result of directing each edge $\{i, j\} \in E(G)$ from $i$ to $j$ if and only if $\sigma^{-1}(i) < \sigma^{-1}(j)$. It is also not hard to see that the poset $P_\alpha$ associated to $\alpha \in \Acyc(G)$ has a linear extension (e.g., for $i \in [n]$, we can construct a linear extension $\sigma$ by setting $\sigma^{-1}(i)$ to be a source of $\alpha$, then removing the source and all incident edges from $\alpha$; in an abuse of notation,\footnote{We will commit similar abuses of notation in Section \ref{sec:fixed_one_graph}. They should not raise any confusion when invoked.} we understand $\alpha$ here as being mutated over the course of this greedy algorithm). We write 
\begin{align*}
    \mathcal L([\alpha]_\approx) = \bigsqcup_{\hat{\alpha} \in [\alpha]_\approx} \mathcal L(\hat{\alpha}).
\end{align*}
We refer the reader to \cite[Section 4]{defant2021friends} for a more comprehensive discussion regarding why these notions are of importance in the study of friends-and-strangers graphs (though this is illuminated in passing in Subsection \ref{subsec:background} and in the arguments of Section \ref{sec:fixed_one_graph}). 

For a graph $G$ and acyclic orientation $\alpha \in \Acyc(G)$, we can partition the directed edges of any cycle subgraph $\mathcal C$ of $G$ into $\mathcal C_\alpha^-$ and $\mathcal C_\alpha^+$, corresponding to edges directed in one of two possible directions under $\alpha$ in $\mathcal C$. The article \cite{pretzel1986reorienting} studied precisely when an acyclic orientation could be reached from another by a sequence of inflips or outflips, while \cite{propp2021lattice} extends this result by providing an upper bound on the number of inflips or outflips necessary to reach $\alpha$ from $\alpha'$ whenever $\alpha \sim \alpha'$. 

\begin{lemma}[\cite{pretzel1986reorienting, propp2021lattice}] \label{lem:inflips_outflips}
For $\alpha, \alpha' \in \Acyc(G)$, $\alpha'$ can be reached from $\alpha$ by a sequence of inflips if and only if for every cycle subgraph $\mathcal C$ of $G$, $|\mathcal C_\alpha^-| = |\mathcal C_{\alpha'}^-|$. Furthermore, whenever this is the case, $\alpha'$ can be reached from $\alpha$ by a sequence of at most $\binom{n}{2}$ inflips. Similarly, $\alpha'$ can be reached from $\alpha$ by a sequence of outflips if and only if for every cycle subgraph $\mathcal C$ of $G$, $|\mathcal C_\alpha^-| = |\mathcal C_{\alpha'}^-|$. Furthermore, whenever this is the case, $\alpha'$ can be reached from $\alpha$ by a sequence of at most $\binom{n}{2}$ outflips.
\end{lemma}
We build on Lemma \ref{lem:inflips_outflips}. The following proposition establishes that we could have defined flip equivalence strictly with respect to inflips or outflips, as this would have resulted in the same notion.

\begin{proposition} \label{prop:flip_equiv}
Acyclic orientations $\alpha, \alpha' \in \Acyc(G)$ are flip equivalent if and only if $\alpha'$ can be reached from $\alpha$ by a sequence of inflips. Similarly, $\alpha \sim \alpha'$ if and only if $\alpha'$ can be reached from $\alpha$ by a sequence of outflips.
\end{proposition}

\begin{proof}
The statement that $\alpha'$ is reachable from $\alpha$ via a sequence of inflips (or outflips) implying $\alpha \sim \alpha'$ is immediate. To prove the converse, notice that for any cycle subgraph $\mathcal C$ of $G$ and acyclic orientations $\alpha, \alpha' \in \Acyc(G)$ for which $\alpha'$ can be reached from $\alpha$ by a flip, $|\mathcal C_\alpha^-| = |\mathcal C_{\alpha'}^-|$. Thus, if $\alpha \sim \alpha'$, then $|\mathcal C_\alpha^-| = |\mathcal C_{\alpha'}^-|$, so $\alpha'$ can be reached from $\alpha$ via a sequence of inflips (or outflips).
\end{proof}

\subsection{Background on Friends-and-Strangers Graphs} \label{subsec:background}

We mention those general properties of friends-and-strangers graphs that we will need later in the article. We refer the reader to \cite[Section 2]{defant2021friends} for a thorough treatment of the general properties of friends-and-strangers graphs. 

\begin{proposition}[{\cite[Proposition 2.6]{defant2021friends}}] \label{prop:basic_properties}
The following properties hold.
\begin{enumerate}
    \item Definition \ref{defn:fs_def} is symmetric with respect to $X$ and $Y$: we have that $\FS(X,Y) \cong \FS(Y,X)$.
    \item The graph $\FS(X,Y)$ is bipartite.
    \item If $X$ or $Y$ is disconnected, or if $X$ and $Y$ are connected graphs on $n \geq 3$ vertices and each have a cut vertex, then $\FS(X, Y)$ is disconnected.
\end{enumerate}

\end{proposition}

The definitions concerning acyclic orientations that were introduced in Subsection \ref{subsec:acyclic_orientations} were observed in \cite{defant2021friends} to be central in describing the structure of the connected components of $\FS(\Path_n, Y)$ and $\FS(\Cycle_n, Y)$, which are the graphs we will be interested in during Section \ref{sec:fixed_one_graph}. Specifically, we have the following theorems.

\begin{theorem}[{\cite[Theorem 3.1]{defant2021friends}}] \label{thm:path_comp}
Let $\alpha \in \Acyc(\overline{Y})$. Take any linear extension $\sigma \in \mathcal L(\alpha)$, and let $H_\alpha$ denote the connected component of $\FS(\Path_n, Y)$ which contains $\sigma$. Then 
\begin{align*}
    \FS(\Path_n, Y) = \bigoplus_{\alpha \in \Acyc(\overline{Y})} H_\alpha
\end{align*}
and $V(H_\alpha) = \mathcal L(\alpha)$. In particular, $H_\alpha$ is independent of the choice of $\sigma$.
\end{theorem}

\begin{theorem}[{\cite[Theorem 4.7]{defant2021friends}}] \label{thm:cycle_comp}
Let $\alpha \in \Acyc(\overline{Y})$. Take any linear extension $\sigma \in \mathcal L([\alpha]_\approx)$, and let $H_{[\alpha]_\approx}$ denote the connected component of $\FS(\Cycle_n, Y)$ which contains $\sigma$. Then 
\begin{align*}
    \FS(\Cycle_n, Y) = \bigoplus_{[\alpha]_\approx \in \Acyc(\overline{Y})/\approx} H_{[\alpha]_\approx}
\end{align*}
and $V(H_{[\alpha]_\approx}) = \mathcal L([\alpha]_\approx)$. In particular, $H_{[\alpha]_\approx}$ is independent of the choice of $\sigma$.
\end{theorem}

Defant and Kravitz \cite{defant2021friends} also determined precisely when $\FS(\Cycle_n, Y)$ is connected. The coprimality condition on the sizes of the components of $\overline{Y}$ in Theorem \ref{thm:cycle_connected} may seem surprising at first glance. We refer the reader to the discussion around \cite[Corollary 4.12]{defant2021friends} and \cite[Corollary 4.14]{defant2021friends} to see where this condition emerges and why it is a natural one.

\begin{theorem}[{\cite[Corollary 4.14]{defant2021friends}}] \label{thm:cycle_connected}
Let $Y$ be a graph on $n \geq 3$ vertices. Then $\FS(\Cycle_n, Y)$ is connected if and only if $\overline{Y}$ is a forest with trees $\mathcal T_1, \dots, \mathcal T_r$ such that $\gcd(|V(\mathcal T_1)|, \dots, |V(\mathcal T_r)|) = 1$.
\end{theorem}

\section{Diameters of \texorpdfstring{$\FS(X,Y)$}{FS(X,Y)} with One Graph Fixed} \label{sec:fixed_one_graph}

Before investigating (and settling) the more global question of whether or not the diameters of connected components of friends-and-strangers graphs are polynomially bounded (in the sense posed by Question \ref{ques:poly_bdd}), we begin by restricting our study by choosing one of the two graphs $X$ and $Y$ to come from a natural family of graphs, and then establish bounds on the diameter of any connected component of $\FS(X,Y)$.

\subsection{Complete Graphs} \label{subsec:complete_graphs}

We begin by setting $Y = K_n$. Take any two configurations $\sigma, \tau \in V(\FS(X,K_n))$ that lie in the same connected component. Consider the following iterative algorithm, applied starting from $\sigma$ and proceeding sequentially on $i \in [n]$. In an abuse of notation, $\sigma$ is understood to be mutated over the course of this algorithm as we perform $(X,K_n)$-friendly swaps to modify its mappings.
\begin{enumerate}
    \item If $\sigma(i) = \tau(i)$, do nothing.
    \item If $\sigma(i) \neq \tau(i)$, swap $\tau(i)$ onto $i$ along a simple path, then swap $\sigma(i)$ back along the simple path that $\tau(i)$ traversed. 
\end{enumerate}
It is straightforward to prove via induction that at the beginning of any iteration $i \in [n]$, $\sigma(i)$ and $\tau(i)$ lie upon the same connected component of $X$ (so that the algorithm may always proceed), and that $\sigma(j) = \tau(j)$ for all $j < i$. Thus, $\sigma = \tau$ when the algorithm terminates after $n-1$ iterations (it must be that $\sigma(n) = \tau(n)$ at the beginning of the $n$\textsuperscript{th} iteration). For any iteration $i \in [n]$, step (2) requires at most $n-1$ $(X, K_n)$-friendly swaps to move $\tau(i)$ onto $i$, and at most $n-2$ $(X, K_n)$-friendly swaps to move $\sigma(i)$ back. This establishes that the diameter of any component of $\FS(X, K_n)$ is therefore at most $(n-1)((n-1)+(n-2)) = 2n^2 - 5n + 3 = O(n^2)$.

Finding the exact distance between two configurations in $\FS(X, K_n)$ is known as the \textit{\textcolor{red}{token swapping problem on $X$}} in the theoretical computer science literature. The $O(n^2)$ bound on the diameter of any component of $\FS(X, K_n)$ is well known, and we also have a bound of $\Omega(n^2)$ on the diameter of any component of $\FS(X, K_n)$ for particular choices of $X$ (e.g., see Remark \ref{rmk:path_lower_bound}). In general, computing exact distances between two configurations in $\FS(X,K_n)$, as well as the diameters of its connected components, is challenging, even when imposing additional assumptions on $X$ (e.g., see \cite{biniaz2023token, yamanaka2015swapping}). There do exist, however, exact polynomial-time algorithms which solve the token swapping problem for a number of choices of $X$, including cliques \cite{cayley1849lxxvii}, paths \cite{jerrum1985complexity}, stars \cite{portier1990whitney}, cycles \cite{kawahara2016time, van2014upper}, and complete bipartite graphs \cite{yamanaka2015swapping}. See Subsection \ref{subsec:complexity} for additional discussion regarding matters of hardness and approximation.

\subsection{Path Graphs} \label{subsec:path_graphs}

In this subsection, we fix $X = \Path_n$. We begin by introducing a notion which will serve as a monovariant in the proof of Proposition \ref{prop:path_diam}.

\begin{definition}
For $\sigma, \tau \in \mathfrak{S}_n$, call the ordered pair $(i,j)$ ($i,j \in [n]$, $i < j$) a \textit{\textcolor{red}{$(\sigma, \tau)$-inversion}} if either 
\begin{enumerate}
	\item $\sigma^{-1}(i) < \sigma^{-1}(j)$ and $\tau^{-1}(j) < \tau^{-1}(i)$, 
	\item $\sigma^{-1}(j) < \sigma^{-1}(i)$ and $\tau^{-1}(i) < \tau^{-1}(j)$.
\end{enumerate}
Denote the number of $(\sigma,\tau)$-inversions by $\inv(\sigma, \tau)$.
\end{definition}

In other words, the ordered pair $(i,j)$ is a $(\sigma, \tau)$-inversion if the relative ordering of the inverse images of $i, j$ under $\sigma$ is opposite that of $\tau$. If (without loss of generality) $\tau$ is the identity permutation, then $\inv(\sigma, \tau) = \inv(\sigma)$, the number of inversions of $\sigma$. It also follows immediately that $\inv(\sigma, \tau) = 0$ if and only if $\sigma = \tau$. 

\begin{proposition} \label{prop:path_diam}
Take $\alpha \in \Acyc(\overline{Y})$, and let $H_\alpha$ denote the corresponding connected component of $\FS(\Path_n, Y)$. Let $P_\alpha = ([n], \leq_\alpha)$ be the poset on $[n]$ for which $i \leq_\alpha j$ if and only if there exists a directed path from $i$ to $j$ in $\overline{Y}$ under $\alpha$. Then $\diam(H_\alpha) \leq \binom{n}{2} - p_\alpha$, where $p_\alpha$ denotes the number of comparable ordered pairs $(i,j)$ with $i, j \in [n]$, $i < j$ in $P_\alpha$.
\end{proposition}

\begin{proof}
We will show for any $\sigma, \tau \in V(H_\alpha)$ that $d(\sigma, \tau) = \inv(\sigma, \tau)$. Any $(\Path_n, Y)$-friendly swap reduces the number of $(\sigma,\tau)$-inversions by at most one, so $d(\sigma, \tau) \geq \inv(\sigma, \tau)$. Now consider the following variant of the bubble sort algorithm, which we perform beginning from $\sigma = \sigma(1) \sigma(2) \cdots \sigma(n)$. Say $\sigma(i_1) = \tau(1)$, and swap $\sigma(i_1)$ down to position $1$, yielding $\sigma_1$ with $\sigma_1(1) = \tau(1)$. Now, say $\sigma(i_2) = \tau(2)$ (with $i_2 \geq 2$), and swap $\sigma(i_2)$ down to position $2$, yielding $\sigma_2$ with $\sigma_2(j) = \tau(j)$ for $j \in [2]$; continue until we achieve $\sigma_n = \tau$. It is immediate that the execution of any swap performed during this algorithm would decrement $\inv(\sigma, \tau)$ by $1$. Furthermore, any proposed swap in this algorithm can be executed, i.e., involves two elements which comprise an edge in $Y$. Indeed, Theorem \ref{thm:path_comp} yields $\sigma, \tau \in V(H_\alpha) = \mathcal L(\alpha)$, but the existence of a swap in this algorithm that cannot be executed would yield $\alpha_{\overline{Y}}(\sigma) \neq \alpha_{\overline{Y}}(\tau)$ (if the proposed swap fails to be an edge in $Y$, it is an edge in $\overline{Y}$, and would be directed in opposite directions under $\alpha_{\overline{Y}}(\sigma)$ and $\alpha_{\overline{Y}}(\tau)$ because the two elements comprising the swap constitute a $(\sigma,\tau)$-inversion), which is a contradiction. Thus, $d(\sigma, \tau) = \inv(\sigma, \tau)$. If $(i,j) \in P_\alpha$, it follows from $\sigma, \tau \in \mathcal L(\alpha)$ that $\sigma^{-1}(i) < \sigma^{-1}(j)$ and $\tau^{-1}(i) < \tau^{-1}(j)$, so $(i,j)$ is not a $(\sigma, \tau)$-inversion. Thus, $d(\sigma, \tau) = \inv(\sigma, \tau) \leq \binom{n}{2} - p_\alpha$, and therefore $\diam(H_\alpha) \leq \binom{n}{2} - p_\alpha$. 
\end{proof}

Certainly, the two vertices incident to an edge of $\overline{Y}$ are comparable in the poset $P_\alpha = ([n], \leq_\alpha)$ for any $\alpha \in \Acyc(\overline{Y})$. This yields the following statement, as $\binom{n}{2} - p_\alpha \leq \binom{n}{2} - |E(\overline{Y})| = |E(Y)|$. For simplicity, we appeal to Theorem \ref{thm:path_bd}, rather than Proposition \ref{prop:path_diam}, in forthcoming arguments.

\begin{theorem} \label{thm:path_bd}
    The diameter of any connected component of $\FS(\Path_n, Y)$ is at most $|E(Y)|$.
\end{theorem}

\begin{remark} \label{rmk:path_lower_bound}
    It is not hard to see that $\FS(\Path_n, K_n)$ is connected (e.g., for any two configurations $\sigma, \tau \in V(\FS(\Path_n, K_n))$, the algorithm from Subsection \ref{subsec:complete_graphs} yields a path between $\sigma$ and $\tau$). From the proof of Proposition \ref{prop:path_diam}, we have for any $\sigma, \tau \in V(\FS(\Path_n, K_n))$ that $d(\sigma, \tau) = \inv(\sigma, \tau) \leq \binom{n}{2}$, and $\inv(\sigma, \tau) = \binom{n}{2}$ when $\tau$ is the ``reverse" of $\sigma$ (i.e., $\tau(i) = \sigma(n-i+1)$ for all $i \in [n]$). So $\diam(\FS(\Path_n, K_n)) = \binom{n}{2}$. Combined with Proposition \ref{prop:path_diam}, this establishes that the maximum diameter of a component of $\FS(\Path_n, K_n)$ is $\Omega(n^2)$, and thus $\Theta(n^2)$. Thus, there exist families of $n$-vertex graphs $Y$ for which the maximum diameter of a component of $\FS(\Path_n, Y)$ has diameter $\Theta(n^2)$. The same can be said for $\FS(K_n, Y)$.
\end{remark}

\begin{remark}
The upper bound of $\binom{n}{2} - p_\alpha$ on $\diam(H_\alpha)$ in Theorem \ref{prop:path_diam} corresponds to the number of ordered pairs $(i,j)$ ($i < j$; $i, j \in [n]$) that are incomparable in the poset $P_\alpha = ([n], \leq_\alpha)$. It follows from the proof of Proposition \ref{prop:path_diam} that for arbitrary $\sigma, \tau \in V(H_\alpha) = \mathcal L(\alpha)$, any $(\sigma,\tau)$-inversion must be a pair of incomparable elements in $P_\alpha$, and $d(\sigma, \tau) = \inv(\sigma, \tau)$. We now apply these observations to show that the upper bound on $\diam(H_\alpha)$ fails to be sharp: Figure \ref{fig:path_diam_countereg} provides an illustration of our construction. For $n = 6$, consider the graph shown in Figure \ref{fig:path_diam_countereg_a}, whose complement is shown in Figure \ref{fig:path_diam_countereg_b}. We will take $\alpha \in \Acyc(\overline{Y})$ to be an acyclic orientation for which the edges in this connected component are oriented as in Figure \ref{fig:path_diam_countereg_c}.

\begin{figure}[ht]
    \centering
    \subfloat[The graph $Y$.]{\includegraphics[width=0.33\textwidth]{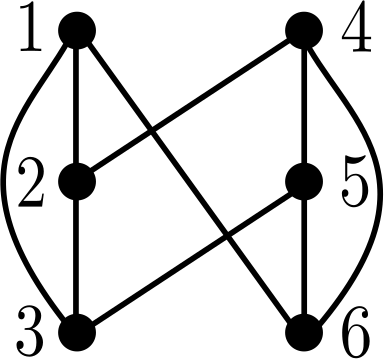}\label{fig:path_diam_countereg_a}}
    \hfill
    \subfloat[The graph $\overline{Y}$.]{\includegraphics[width=0.33\textwidth]{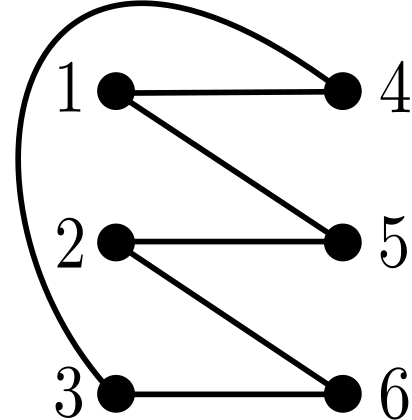}\label{fig:path_diam_countereg_b}}
    \hfill
    \subfloat[Direction of the edges in this component under $\alpha$.]{\includegraphics[width=0.33\textwidth]{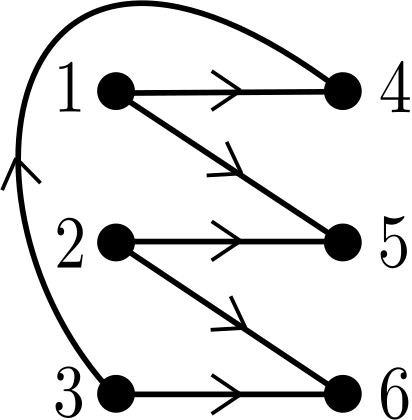}\label{fig:path_diam_countereg_c}}
    \caption{The construction we used to show that the bound given in Proposition \ref{prop:path_diam} fails to be sharp in general.}
    \label{fig:path_diam_countereg}
\end{figure}

Assume (towards a contradiction) that there exist $\sigma, \tau \in V(H_\alpha) = \mathcal L(\alpha)$ for which $d(\sigma, \tau) = \inv(\sigma, \tau) = \binom{n}{2} - p_\alpha$, so that all pairs of incomparable elements in $P_\alpha$ are $(\sigma,\tau)$-inversions. Any two elements in $\{1, 2, 3\}$ are incomparable in $P_\alpha$, so the relative ordering of $\{1, 2, 3\}$ in $\sigma$ must be the relative ordering of $\{1,2,3\}$ in $\tau$ reversed. Without loss of generality, assume $\sigma$ has relative ordering $1 \to 2\to 3$, so $\tau$ has $3 \to 2\to 1$. Since $\sigma, \tau \in \mathcal L(\alpha)$, the element $4$ follows vertex $2$ in both $\sigma$ and $\tau$, so $(2,4)$ is not a $(\sigma,\tau)$-inversion. But $(2,4)$ is incomparable in $P_\alpha$, a contradiction.
\end{remark}

\subsection{Cycle Graphs} \label{subsec:cycle_graphs}

In this subsection, we fix $X = \Cycle_n$. The setting $Y = K_n$ has been studied in the context of circular permutations \cite{kim2016sorting, van2014upper}. In particular, \cite[Procedure 3.6]{kim2016sorting} provides an algorithm that achieves the minimal number of $(\Cycle_n, K_n)$-friendly swaps between any two permutations in $\mathfrak{S}_n$. Extracting these results yields that the diameter of $\FS(\Cycle_n, K_n)$ is $\lfloor n^2/4 \rfloor$. In the spirit of Remark \ref{rmk:path_lower_bound}, it follows that there exist families of $n$-vertex graphs $Y$ for which $\FS(\Cycle_n, Y)$ has diameter $\Theta(n^2)$, and it is worth asking what conditions on $Y$ yield that the maximum diameter of a connected component of $\FS(\Cycle_n, Y)$ is at most quadratic in $n$. In this direction, we have the following proposition.

\begin{proposition} \label{prop:span_star}
If $Y$ has an isolated vertex or $|E(Y)| \leq n-2$, then the diameter of any connected component of $\FS(\Cycle_n, Y)$ is at most $|E(Y)|$.
\end{proposition}

\begin{proof}
Consider any $\sigma, \tau \in V(\FS(\Cycle_n, Y))$ which lie in the same component. If $Y$ has an isolated vertex $v$, then it must be that $\sigma^{-1}(v)$ remains fixed over any sequence of $(\Cycle_n, Y)$-friendly swaps from $\sigma$ to $\tau$. Thus, it must be that any path from $\sigma$ to $\tau$ in $\FS(\Cycle_n, Y)$ is a path in 
\begin{align*}
    \FS\left(\Cycle_n |_{V(\Cycle_n) \setminus \{\sigma^{-1}(v)\}}, Y_{V(Y) \setminus \{v\}}\right),
\end{align*}
from which the result follows from Theorem \ref{thm:path_bd}. For the setting $|E(Y)| \leq n-2$, we will show that any $\sigma, \tau \in V(\FS(\Cycle_n, Y))$ in the same connected component will remain in the same component after removing some edge from $\Cycle_n$, from which the desired result again follows immediately from Theorem \ref{thm:path_bd}. Assume (towards a contradiction) that every path from $\sigma$ to $\tau$ in $\FS(\Cycle_n, Y)$ involves a swap over every edge in $E(\Cycle_n)$. Consider a shortest path $\Sigma = \{\sigma_i\}_{i=0}^\lambda$ from $\sigma$ or $\tau$, which has that $\sigma_0 = \sigma$ and $\sigma_\lambda = \tau$, and $\lambda \geq n$ by the assumption. Consider the subsequence $\{\sigma_i\}_{i=0}^{n-1}$ consisting of the first $n-1$ $(\Cycle_n, Y)$-friendly swaps of $\Sigma$. This must be a shortest path from $\sigma$ to $\sigma_{n-1}$ in $\FS(\Cycle_n, Y)$, and swaps upon at most $n-1$ edges of $\Cycle_n$: say $e \in E(\Cycle_n)$ is an edge upon which a swap does not occur, and let $\Cycle_n^{-e}$ be $\Cycle_n$ with this edge $e$ removed. Then $\{\sigma_i\}_{i=0}^{n-1}$ is a shortest path from $\sigma$ to $\sigma_{n-1}$ in $\FS(\Cycle_n^{-e}, Y)$ with length $n-1$. This contradicts Theorem \ref{thm:path_bd}, which yields $d(\sigma, \sigma_{n-1}) \leq |E(Y)| \leq n-2$.
\end{proof}

We were unable to extend the $O(n^2)$ bound from Proposition \ref{prop:span_star} to general $Y$, although we suspect that this is the truth (see Subsection \ref{subsec:improvements}). However, the existence of a universal constant $C>0$ such that the maximum diameter of a component of $\FS(\Cycle_n, Y)$ is $O(n^C)$ remains highly desirable. In conjunction with Theorem \ref{thm:cycle_connected}, the following theorem yields such a result whenever $\FS(\Cycle_n, Y)$ is connected.

\begin{theorem} \label{thm:cycle_diam_1}
Let $Y$ be a graph on $n \geq 3$ vertices, and let $n_1, \dots, n_r$ denote the sizes of the components of $\overline{Y}$. If $\gcd(n_1, \dots, n_r) = 1$, then any component of $\FS(\Cycle_n, Y)$ has diameter at most $4n^3 + |E(Y)|$.
\end{theorem}

\begin{proof}
Certainly, $r \geq 2$. Without loss of generality, we assume that $n_1 \leq \dots \leq n_r$, and we denote the corresponding components of $\overline{Y}$ by $\overline{Y_1}, \dots, \overline{Y_r}$, respectively. For $\alpha \in \Acyc(\overline{Y})$, we let $\alpha_i$ denote the acyclic orientation induced by $\alpha$ on $\overline{Y_i}$. We now fix $\alpha, \alpha'' \in \Acyc(\overline{Y})$ such that $\alpha \approx \alpha''$. Before studying distances in $\FS(\Cycle_n, Y)$, we will first bound the number of double-flips necessary to reach $\alpha''$ from $\alpha$. Certainly, $\alpha_i \sim \alpha''_i$ for all $i \in [r]$, and by Proposition \ref{prop:flip_equiv}, we can reach $\alpha''_i$ in no more than $\binom{n_i}{2}$ inflips or outflips from $\alpha_i$. Observe that for any $\alpha_i''$, we may return to $\alpha_i''$ by applying a different sequence of $n_i$ inflips; see Figure \ref{fig:inflips} for an illustration. Indeed, take a linear extension $\sigma \in \mathcal L(\alpha_i'')$, labeled $\sigma = \sigma(1)\sigma(2)\cdots \sigma(n_i)$, and perform an inflip on $\alpha''_i$ by converting the source $\sigma(1)$ into a sink, so that $\sigma(2)\dots \sigma(n_i)\sigma(1)$ is a linear extension of the poset associated to the resulting acyclic orientation in $\Acyc(\overline{Y_i})$. Performing $n_i$ inflips on $\alpha''_i$ in this manner returns $\sigma$ as a linear extension of the poset associated to the resulting acyclic orientation: since there exists a unique acyclic orientation $\alpha_{\overline{Y}}(\sigma) \in \Acyc(\overline{Y_i})$ for which $\sigma \in \mathcal L(\alpha_{\overline{Y}}(\sigma))$, this acyclic orientation must be $\alpha'_i$. Similarly, we can return to $\alpha_i''$ by applying a sequence of $n_i$ outflips.

\begin{figure}[ht]
    \centering
    \includegraphics[width=\textwidth]{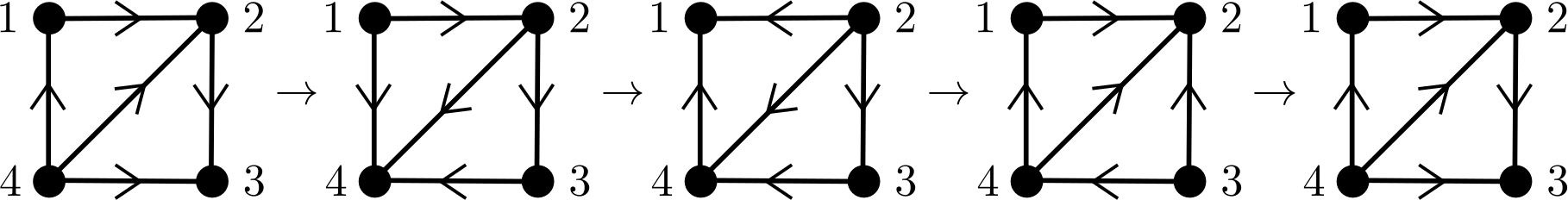}
    \caption{An example of a sequence of $n$ inflips which takes an acyclic orientation $\alpha$ of an $n$-vertex graph back to itself. We demonstrate on a $4$-vertex graph. The permutations $\bigl(\begin{smallmatrix}
    1 & 2 & 3 & 4 \\
    4 & 1 & 2 & 3
    \end{smallmatrix}\bigl)$, $\bigl(\begin{smallmatrix}
    1 & 2 & 3 & 4 \\
    1 & 2 & 3 & 4
    \end{smallmatrix}\bigl)$, $\bigl(\begin{smallmatrix}
    1 & 2 & 3 & 4 \\
    2 & 3 & 4 & 1
    \end{smallmatrix}\bigl)$, $\bigl(\begin{smallmatrix}
    1 & 2 & 3 & 4 \\
    3 & 4 & 1 & 2
    \end{smallmatrix}\bigl)$ are linear extensions of the posets associated with the first four acyclic orientations shown, respectively. The first and fifth acyclic orientations are the same.}
    \label{fig:inflips}
\end{figure}

Recalling that a double-flip applied to an acyclic orientation involves flipping a nonadjacent source and sink into (respectively) a sink and source, we thus proceed as follows. Starting from the acyclic orientation $\alpha$, perform a sequence of double-flips that act as inflips on sources in $\alpha_r$ and outflips on sinks in $\alpha_1, \dots, \alpha_{r-1}$ until we have reached $\alpha_1'', \dots, \alpha_r''$ at least once. Specifically, begin by performing inflips on $\alpha_r$ and outflips on $\alpha_1$ until we either reach $\alpha_1''$ (at which point we begin performing outflips on sinks in $\alpha_2$) or $\alpha_r''$ (at which point we begin performing inflips on sources in $\alpha_r''$ as described previously to return to $\alpha_r''$ every $n_r$ inflips). If we reach $\alpha_1'', \dots, \alpha_{r-1}''$ prior to $\alpha_r''$, then perform outflips on sinks in $\alpha_1''$ (returning to $\alpha_1''$ every $n_1$ outflips) until $\alpha_r''$ is reached: from here, pair these outflips on sinks with inflips on sources in $\alpha_r''$ until we retain $\alpha_1''$. Otherwise, we reach $\alpha_r''$ prior to $\alpha_1'', \dots, \alpha_{r-1}''$, for which $\alpha_r''$ will be ``offset" once we have $\alpha_1'', \dots, \alpha_{r-1}''$, since we are performing inflips on sources which return to $\alpha_r''$ every $n_r$ inflips. In either case, call the resulting acyclic orientation $\Tilde{\alpha}$, which satisfies $\Tilde{\alpha_i} = \alpha_i''$ for all $i \in [r-1]$ while $\Tilde{\alpha}_r$ differs from $\alpha_r''$ by some offset $0 \leq c < n_r$. By tracing the preceding description and recalling Proposition \ref{prop:flip_equiv}, it follows that the number of double-flips we perform to reach $\Tilde{\alpha}$ from $\alpha$ is bounded above by
\begin{align*}
    \max\left\{\binom{n_r}{2} + n_1, \ \sum_{i=1}^{r-1} \binom{n_i}{2}\right\} \leq \sum_{i=1}^r n_i^2 \leq \left(\sum_{i=1}^r n_i \right)^2 = n^2.
\end{align*}
By B\'ezout's Lemma (recall that $\gcd(n_1, \dots, n_r) = 1$), there exist integers $0 \leq d_1, \dots, d_{r-1} < n_r$ such that 
\begin{align*}
    d_1n_1 + \cdots + d_{r-1}n_{r-1} \equiv n_r - c \pmod{n_r}.
\end{align*}
Thus, from $\Tilde{\alpha}$, we can reach $\alpha''$ by performing $d_in_i$ outflips on $\Tilde{\alpha}_i = \alpha_i''$ for $i \in [r-1]$ (returning to $\Tilde{\alpha}_i = \alpha_i''$ every $n_i$ outflips), while performing inflips on $\Tilde{\alpha}_r$ as discussed to reach $\alpha_r''$. The number of double-flips we perform to reach $\alpha''$ from $\Tilde{\alpha}$ is therefore bounded above by 
\begin{align*}
    \sum_{i=1}^{r-1} d_in_i \leq \max\left\{d_1, \dots, d_{r-1}\right\}\left(\sum_{i=1}^{r-1} n_i\right) \leq n_rn \leq n^2,
\end{align*}
so at most $2n^2$ double-flips are necessary to reach $\alpha''$ from $\alpha$.

\medskip

We now turn to bounding $d(\sigma, \tau)$ for configurations $\sigma, \tau \in V(\FS(\Cycle_n, Y))$ in the same connected component. By Theorem \ref{thm:cycle_comp}, we have that $\sigma, \tau \in \mathcal L([\alpha]_\approx)$ for some $[\alpha]_\approx \in \Acyc(\overline{Y}) / \!\! \approx$. Denote $\alpha = \alpha_{\overline{Y}}(\sigma)$ and $\alpha'' = \alpha_{\overline{Y}}(\tau)$. By the preceding discussion, we can reach $\alpha''$ from $\alpha$ in $\lambda \leq 2n^2$ double-flips, yielding a sequence of acyclic orientations $\Sigma = \{\alpha_i\}_{i=0}^\lambda$ in the equivalence class $[\alpha]_\approx$ with $\alpha_0 = \alpha$ and $\alpha_\lambda = \alpha''$. From $\Sigma$, we will now construct a sequence of $(\Cycle_n, Y)$-friendly swaps which we can apply on $\sigma$; see Figure \ref{fig:cycle_flips} for an illustration. If the double-flip we performed to reach $\alpha_1$ from $\alpha$ inflips the source $v$ and outflips the sink $w$ in $\alpha$, it follows from $\sigma \in \mathcal L(\alpha)$ that for any $i < \sigma^{-1}(v)$, $\{\sigma(i), v\} \in E(Y)$. Indeed, if we had that $\{\sigma(i), v\} \in E(\overline{Y})$, $v$ being a source in $\alpha$ would imply that this edge is directed from $v$ to $\sigma(i)$ in $\alpha$, contradicting $\sigma \in \mathcal L(\alpha)$. Similarly, for any $j > \sigma^{-1}(w)$, $\{\sigma(j), w\} \in E(Y)$. Thus, we can swap $v$ to $1$ and $w$ to $n$ in no more than $2n-3$ $(\Cycle_n, Y)$-friendly swaps: it is easy to check that the resulting configuration remains in $\mathcal L(\alpha)$. Then we perform a $(\Cycle_n, Y)$-friendly swap which swaps $v$ and $w$ along the edge $\{1, n\}$ ($\{v,w\} \notin E(\overline{Y})$ by the definition of a double-flip, so $\{v,w\} \in E(Y)$). It is also straightforward to check that the configuration $\sigma_1$ resulting from this interchange is now in $\mathcal L(\alpha_1)$. 

\begin{figure}[ht]
    \centering
    \subfloat[Acyclic orientations $\alpha$, $\alpha_1 \in \Acyc(\overline{Y})$.]{\includegraphics[width=0.4\textwidth]{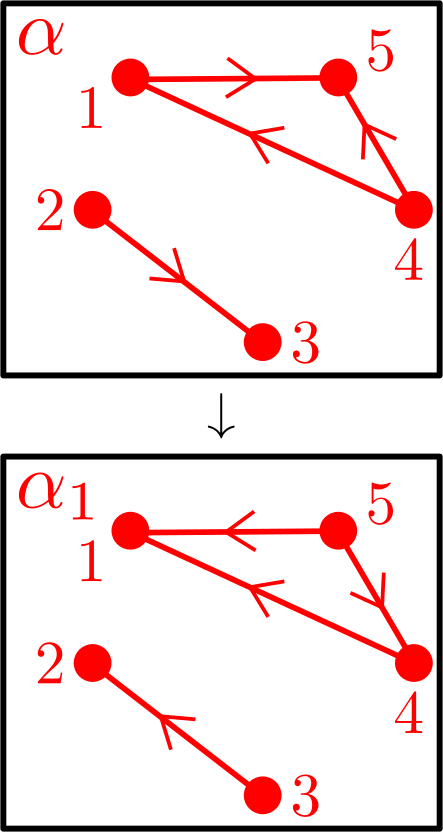}\label{fig:cycle_diam_1b_1}}
    \hfill
    \subfloat[The corresponding sequence of $(\Cycle_n, Y)$-friendly swaps we construct.]{\includegraphics[width=0.4\textwidth]{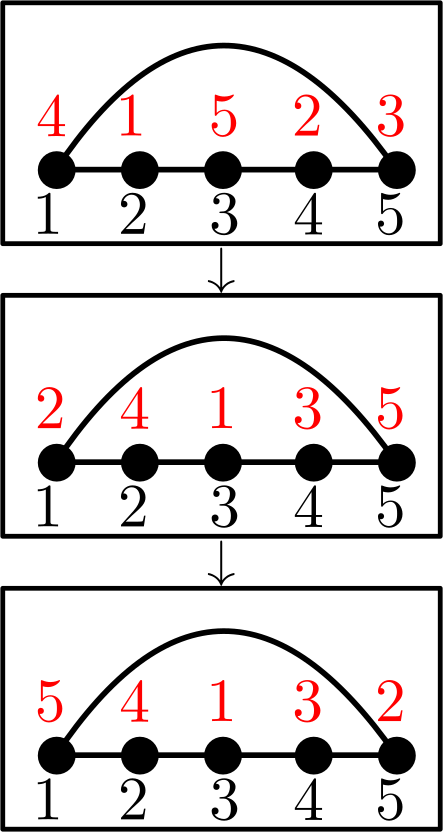}\label{fig:cycle_diam_1b_2}}
    \caption{The sequence of $(\Cycle_n, Y)$-friendly swaps that we construct corresponding to $\alpha, \alpha_1 \in \Acyc(\overline{Y})$ that are double-flip equivalent. We demonstrate on $5$-vertex graphs. The topmost bijection is in $\mathcal L(\alpha)$. We inflip $v = 2$ and outflip $w = 5$ to reach $\alpha''$ from $\alpha$: swapping $v$ left to $1$, then $w$ right to $5$, then swapping $v$ and $w$ along $\{1,5\}$ yields a permutation in $\mathcal L(\alpha_1)$.}
    \label{fig:cycle_flips}
\end{figure}

Proceed similarly through all $\lambda$ double-flips, and call the resulting configuration $\Tilde{\sigma}$: this configuration satisfies $\Tilde{\sigma} \in \mathcal L(\alpha'')$. Since $\Tilde{\sigma}, \tau \in \mathcal L(\alpha'')$, it follows from Theorem \ref{thm:path_comp} that $\Tilde{\sigma}, \tau$ lie in the same component of $\FS(\Path_n, Y)$ (specifically, the copy of $\Path_n$ in $\Cycle_n$ resulting from excluding the edge $\{1, n\}$). By Theorem \ref{thm:path_bd}, we can now reach $\tau$ from $\Tilde{\sigma}$ by performing no more than $|E(Y)|$ $(\Cycle_n, Y)$-friendly swaps. Altogether, we have that
\begin{align*}
    d(\sigma, \tau) \leq 2n^2 \cdot 2n + |E(Y)| = 4n^3 + |E(Y)|,
\end{align*}
so at most $4n^3 + |E(Y)|$ $(\Cycle_n, Y)$-friendly swaps are necessary to reach $\tau$ from $\sigma$.
\end{proof}

\begin{corollary}
For $n \geq 3$, if $\FS(\Cycle_n, Y)$ is connected, $\diam(\FS(\Cycle_n, Y)) \leq 4n^3 + |E(Y)|$.
\end{corollary}

Theorem \ref{thm:cycle_diam_1} can now be invoked to establish the following general bound on the diameter of any connected component of $\FS(\Cycle_n, Y)$, where $Y$ is arbitrary. This proves that, in the sense of Question \ref{ques:poly_bdd}, the diameter of $\FS(\Cycle_n, Y)$ is polynomially bounded.

\begin{theorem} \label{thm:cycle_diam_2}
    The diameter of any component of $\FS(\Cycle_n, Y)$ is at most $8n^4(1+o(1))$.
\end{theorem}

\begin{proof}
    Consider two configurations $\sigma, \tau \in V(\FS(\Cycle_n, Y))$ in the same connected component. We construct an $(n+1)$-vertex graph $Y'$ by adding a vertex $v$ to $Y$ that is adjacent to all vertices in $V(Y)$, so $Y'$ has a spanning star subgraph with central vertex $v$. Define bijections $\sigma', \tau' \in V(\FS(\Cycle_{n+1}, Y'))$ by 
    \begin{align*}
        & \sigma'(i) = \begin{cases}
            \sigma(i) & i \in [n], \\
            v & i = n+1,
        \end{cases}
        & \tau'(i) = \begin{cases}
            \tau(i) & i \in [n], \\
            v & i = n+1.
        \end{cases}
    \end{align*}
    The configurations $\sigma', \tau'$ are in the same component of $\FS(\Cycle_{n+1}, Y')$. Indeed, from a sequence of $(\Cycle_n, Y)$-friendly swaps $\Sigma_1$ from $\sigma$ to $\tau$ of shortest length, we can construct a sequence $\Sigma_1'$ of $(\Cycle_{n+1}, Y')$-friendly swaps from $\sigma'$ to $\tau'$ by replacing every swap in $\Sigma_1$ which occurs along $\{1,n\} \in E(\Cycle_n)$ by a sequence of three swaps along the following edges in $E(\Cycle_n)$:
    \begin{align*}
        \{n, n+1\}, \{1, n+1\}, \{n, n+1\}.
    \end{align*}
    It is straightforward to confirm that $\Sigma_1'$ is a path from $\sigma'$ to $\tau'$, constructed from $\Sigma_1$ by ``crossing" the vertex $v$ as needed. Since $Y'$ has a spanning star subgraph, $\overline{Y'}$ has an isolated vertex, so it follows immediately that the components of $\overline{Y'}$ have jointly coprime size. So by Theorem \ref{thm:cycle_diam_1},
    \begin{align*}
        d(\sigma', \tau') \leq 4(n+1)^3 + |E(Y')| \leq 4(n+1)^3 + \binom{n+1}{2} = 4n^3(1+f(n)),
    \end{align*}
    where $f(n) = o(1)$. Let $\Sigma_2'$ be a sequence of swaps from $\sigma'$ to $\tau'$ of length at most $4n^3(1+f(n))$. We construct $\Sigma_2$ from $\Sigma_2'$ by removing all $(\Cycle_{n+1}, Y')$-friendly swaps involving $v$: it is straightforward to notice that $\Sigma_2$ yields a path of length at most $4n^3(1+f(n))$ from $\sigma$ to some cyclic rotation $\tau_*$ of $\tau$, i.e., $d(\sigma, \tau_*) \leq 4n^3(1+f(n))$. Towards a contradiction, assume $d(\tau, \tau_*) \geq 8n^4(1+2f(n)) + n$, and let $v_1, \dots, v_{n+1}$ be vertices along a shortest path from $\tau$ to $\tau_*$ satisfying $d(v_i, v_{i+1}) > 8n^3(1+2f(n))$ for all $i \in [n]$. Such vertices $v_i$ exist due to our assumption on $d(\tau, \tau_*)$. By appealing to the same argument as above, we deduce that there exist cyclic rotations $\sigma_1, \dots, \sigma_{n+1}$ of $\sigma$ such that $d(v_i, \sigma_i) \leq 4n^3(1+f(n))$ for all $i \in [n+1]$. Since there exist $n$ distinct rotations of $\sigma$, the pigeonhole principle yields the existence of $i \neq j$ for which
    \begin{align*}
        d(v_i, v_j) \leq d(v_i, \sigma') + d(\sigma', v_j) \leq 4n^3(1+f(n)) + 4n^3(1+f(n)) \leq 8n^3(1+2f(n))
    \end{align*}
    for some rotation $\sigma'$ of $\sigma$, which is a contradiction. Therefore, we conclude that
    \begin{align*}
        d(\sigma, \tau) \leq d(\sigma, \tau_*) + d(\tau_*, \tau) \leq 4n^3(1+f(n)) + 8n^3(1+2f(n))(n+1) = 8n^4(1+o(1)).
    \end{align*}
    The desired result now follows immediately.
\end{proof}

From Theorem \ref{thm:cycle_diam_2}, we can also extract the following analogue of Lemma \ref{lem:inflips_outflips} for double-flips.

\begin{corollary} \label{cor:double_flip_dist}
    If $\alpha, \alpha'' \in \Acyc(G)$ satisfy $\alpha \approx \alpha''$, then we can reach $\alpha''$ from $\alpha$ in no more than $4n^4(1+o(1))$ double-flips.
\end{corollary}

\begin{proof}
    Given an $n$-vertex graph $G$ and $\alpha, \alpha'' \in \Acyc(G)$ satisfying $\alpha \approx \alpha''$, extract linear extensions $\sigma \in \mathcal L(\alpha)$, $\tau \in \mathcal L(\alpha'')$, and consider $\sigma$ and $\tau$ as vertices of $\FS(\Cycle_n, G)$. By Theorem \ref{thm:cycle_diam_2}, $d(\sigma, \tau) \leq 8n^4(1+o(1))$, so let $\Sigma = \{\sigma_i\}_{i=0}^\lambda$ be a shortest sequence of swaps from $\sigma$ to $\tau$, so $\lambda \leq 8n^4(1+o(1))$. Let $\Sigma_0 = \{\sigma_{i_j}\}_{j=0}^{\lambda'}$ be the subsequence of $\Sigma$ consisting of all indices $i_j$ for which $\sigma_{i_j+1}$ is reached from $\sigma_{i_j}$ by a $(\Cycle_n, G)$-friendly swap across the edge $\{1,n\}$. Since $\lambda$ is smallest possible, two consecutive swaps of $\Sigma$ cannot both be across the edge $\{1,n\}$, so $\lambda' \leq 4n^4(1+o(1))$. We will now describe how to use $\Sigma_0$ to construct a sequence $\Sigma' = \{\alpha_j\}_{j=0}^{\lambda'+1}$ of acyclic orientations, with $\alpha_0 = \alpha$ and $\alpha_{\lambda'+1} = \alpha''$, for which $\alpha_j$ is reachable from $\alpha_{j-1}$ by a double-flip for all $j \in [\lambda'+1]$. The desired result will then follow immediately. 
    
    Since we reached $\sigma_{i_0}$ from $\sigma$ by swapping along the graph $\FS(\Path_n, G)$ (specifically, the copy of $\Path_n$ in $\Cycle_n$ resulting from excluding the edge $\{1, n\}$), it follows from Theorem \ref{thm:path_comp} that $\sigma_{i_0} \in \mathcal L(\alpha)$. Let $\alpha_1$ be the result of taking $\alpha$ and performing a double-flip which involves an inflip on the source $\sigma_{i_0}(1)$ and an outflip on the sink $\sigma_{i_0}(n)$. Note that $\{\sigma_{i_0}(1), \sigma_{i_0}(n)\} \in E(Y)$ (we swapped these two vertices to reach $\sigma_{i_0+1}$ from $\sigma_{i_0}$), so $\{\sigma_{i_0}(1), \sigma_{i_0}(n)\} \notin E(\overline{Y})$, from which it follows that this is a valid double-flip.\footnote{This correspondence between double-flips and paths in $\FS(\Cycle_n, G)$ is the same as that which was observed in the first paragraph of the proof of \cite[Theorem 4.1]{defant2021friends}.} It is easy to check that $\sigma_{i_0+1} \in \mathcal L(\alpha_1)$, and by appealing to Theorem \ref{thm:path_comp} as before, $\sigma_{i_1} \in \mathcal L(\alpha_1)$. Continuing like this sequentially on $j \in [\lambda'+1]$ (the preceding discussion being the $j=1$ case) yields the desired sequence $\Sigma'$: for the case $j=\lambda'+1$, it follows as before from Theorem \ref{thm:path_comp} that $\sigma_{i_{\lambda'}+1}$ and $\tau$ are linear extensions of the poset (i.e., associated to the same acyclic orientation of $G$), so the final acyclic orientation in $\Sigma'$ is $\alpha_G(\tau) = \alpha''$.
\end{proof}

\section{Proof of Main Result} \label{sec:large_diameter}

We devote this section to answering Question \ref{ques:poly_bdd} in the negative, establishing Theorem \ref{thm:main_diam_result}.

\subsection{The Graphs \texorpdfstring{$X_L$}{XL} and \texorpdfstring{$Y_L$}{YL}} \label{subsec:construction}

We begin with the following observation. One can understand this as the central vertex of $\Star_n$ acting as a ``knob" rotating around $\Cycle_n$, and all other vertices of $V(\Star_n)$ moving cyclically around it: $n(n-1)$ such swaps in the same direction are needed for all vertices of $\Star_n$ to return to their original positions in the starting configuration. This interpretation will help motivate our construction.

\begin{lemma} \label{lem:star_cycle}
Every connected component of $\FS(\Cycle_n, \Star_n)$ is isomorphic to $\Cycle_{n(n-1)}$.
\end{lemma}
\begin{proof}
Consider a component $\mathcal C$ of $\FS(\Cycle_n, \Star_n)$ with permutation $\sigma = \sigma(1)\cdots \sigma(n)$ such that $\sigma(1)$ is the central vertex of $\Star_n$. With $V(\Cycle_{n(n-1)}) = [n(n-1)]$, construct $\varphi: V(\Cycle_{n(n-1)}) \to V(\mathcal C)$ by defining $\varphi(i)$ to be the permutation achieved by starting from $\sigma$ and swapping $\sigma(1)$ rightward $i$ times (e.g., $\varphi(1) = \sigma(2)\sigma(1)\cdots \sigma(n)$). It follows that $\varphi$ is a graph isomorphism.
\end{proof}

We will now construct the graphs $X_L$ and $Y_L$, for every integer $L \geq 1$, that we study to prove Theorem \ref{thm:main_diam_result}. In the following description, assume we have fixed some arbitrary integer $L \geq 1$. 

\subsubsection*{\textcolor{red}{The Graph $X_L$.}}
The graph $X_L$ contains an $L \times 2$ array of cycle subgraphs, with adjacent cycles intersecting in exactly one vertex. Say $X_L$ has $L$ layers, indexed by $\ell \in [L]$; we will subscript subgraphs and vertices corresponding to the ``left column" of $X_L$ by $a$, and those in the right by $b$. As such, we denote the left and right cycle subgraphs in layer $\ell$ by $\mathcal C_a^\ell$ and $\mathcal C_b^\ell$, respectively. Corresponding to each $\mathcal C_a^\ell$ and $\mathcal C_b^\ell$ is a path subgraph of $X_L$ extending out of it; that corresponding to $\mathcal C_a^\ell$ is denoted $\mathcal P_a^\ell$, and similarly $\mathcal P_b^\ell$ for $\mathcal C_b^\ell$. Denote the subgraph of $X_L$ consisting of the $\ell$th layer by $X^\ell$. The subgraph consisting of $\mathcal P_a^\ell$ and $\mathcal C_a^\ell$ is denoted $X_a^\ell$, and similarly $X_b^\ell$ for $P_b^\ell$ and $\mathcal C_b^\ell$. Denote, whenever they are defined for $\ell \in [L]$,
\begin{gather*}
    v_a^\ell = V(\mathcal P_a^\ell) \cap V(\mathcal C_a^\ell), \ v_b^\ell = V(\mathcal P_b^\ell) \cap V(\mathcal C_b^\ell), \ v^\ell = V(\mathcal C_a^\ell) \cap V(\mathcal C_b^\ell), \\
    v_a^{\ell, \ell+1} = V(\mathcal C_a^\ell) \cap V(\mathcal C_a^{\ell+1}), \ v_b^{\ell, \ell+1} = V(\mathcal C_b^\ell) \cap V(\mathcal C_b^{\ell+1}).
\end{gather*}
For each of the following sets, we place three inner vertices in the path in $\mathcal C_a^\ell$ between the two vertices in the set:
\begin{align*}
    \{v_a^\ell, v_a^{\ell, \ell+1}\}, \{v_a^{\ell, \ell+1}, v^\ell\}, \{v^\ell, v_a^{\ell-1, \ell}\}, \{v_a^{\ell-1, \ell}, v_a^\ell\}.
\end{align*}
The analogous statement for $\mathcal C_b^\ell$ holds. The exceptions are layers $1$ and $L$: we place seven inner vertices in the upper path from $v_a^1$ to $v^1$ in $\mathcal C_a^1$ and the upper path from $v_b^1$ to $v^1$ in $\mathcal C_b^1$, and seven inner vertices in the lower path from $v_a^L$ to $v^L$ in $\mathcal C_a^L$ and from $v_b^L$ to $v^L$ in $\mathcal C_b^L$. It follows from our construction that for every $\ell \in [L]$,
\begin{align*}
    |V(\mathcal C_a^\ell)| = |V(\mathcal C_b^\ell)| = 16.
\end{align*}
We will also set,\footnote{It will be important that, for every $\ell \in [L]$, $\mathcal P_a^\ell$ has exactly one more vertex than $\mathcal P_b^\ell$. The choice of the lengths of these paths, as well as the number of inner vertices in the segments of the cycle subgraphs, is not terribly important as long as they are not too small. The values we chose here suffice.} for every $\ell \in [L]$,
\begin{align*}
    |V(\mathcal P_a^\ell)| = 16, |V(\mathcal P_b^\ell)| = 15,
\end{align*}
so that the graph $X_L$ has 
\begin{align*}
    n = 60 + 58(L-1) = 58L+2
\end{align*}
vertices. (Indeed, it can be checked that layer $1$ has $60$ vertices, and for each subsequent layer, we add $58$ new vertices to the graph $Y_L$.) Figure \ref{fig:X_3} illustrates this construction for $L=3$.

\begin{figure}[ht]
\begin{center}
\includegraphics[width=0.5\textwidth]{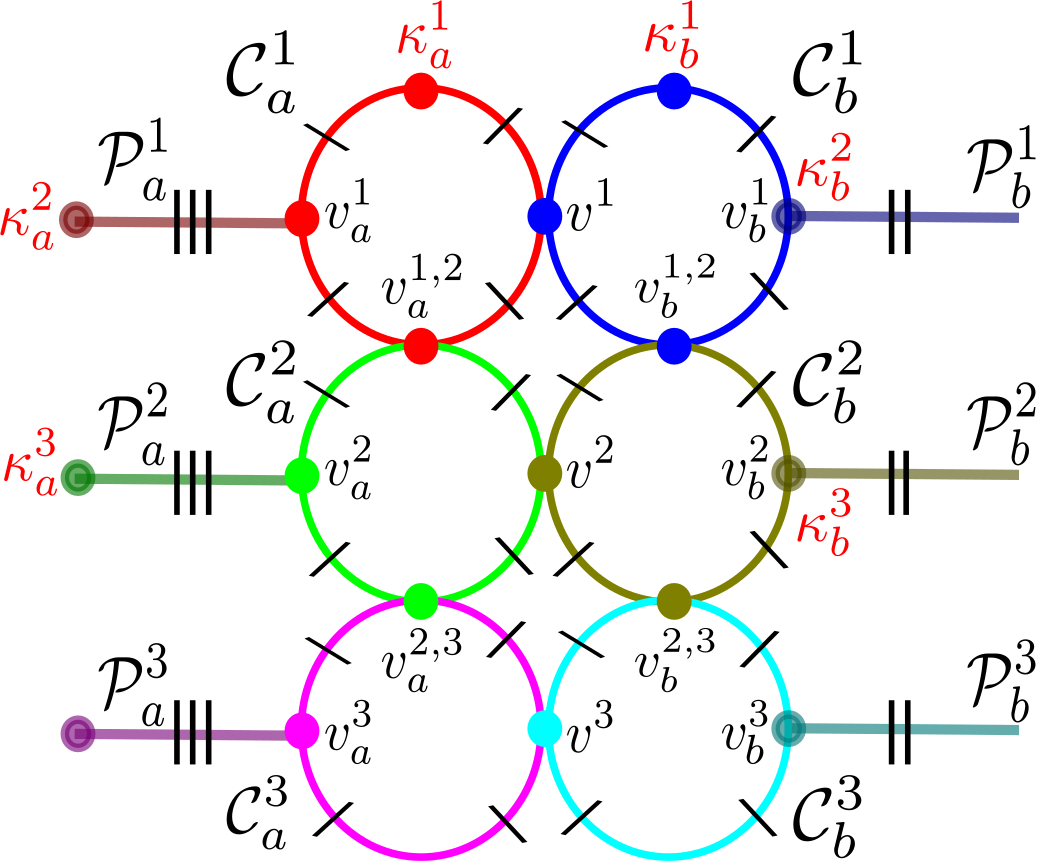} 
\caption{Labeled schematic diagram of the construction for $X_3$. Subgraphs of $X_3$ marked a specific color correspond to the $\sigma_s$-preimages of the vertices of the same color in Figure \ref{fig:Y_3}. We take care in appropriately coloring the vertices between two adjacent cycle subgraphs and between adjacent path and cycle subgraphs. Paths marked with one hatch mark have three inner vertices. The paths $\mathcal P_b^i$ with two hatch marks have $15$ vertices, while paths $\mathcal P_a^i$ with three hatch marks have $16$ vertices.}
\label{fig:X_3}
\end{center}
\end{figure}

\subsubsection*{\textcolor{red}{The Graph $Y_L$.}} 

We construct a complementary graph $Y_L$ for each $X_L$: we assign to each cycle subgraph $\mathcal C_a^\ell$ and $\mathcal C_b^\ell$ of $X_L$ a corresponding ``knob vertex" in $V(Y_L)$, denoted $\kappa_a^\ell$ and $\kappa_b^\ell$, respectively; we set a collection of vertices of $V(Y_L)$ to swap only with each knob. The construction of $Y_L$ proceeds sequentially according to $\ell \in [L]$. Take two disjoint copies of $\Star_{15}$, denoted $\mathcal S_a^1$ and $\mathcal S_b^1$, with central vertices $\kappa_a^1$ and $\kappa_b^1$, respectively, and a complete bipartite graph $\mathcal K^1$ with $15$ vertices in each of its partite sets $\mathcal K_a^1$ and $\mathcal K_b^1$. Set $\kappa_a^1$ and $\kappa_b^1$ adjacent to all the vertices in $V(\mathcal K^1)$. If $L=1$, this completes the construction of $Y_L$. If $L > 1$, take one vertex each in $\mathcal K_a^1$ and $\mathcal K_b^1$, which shall correspond to $\kappa_a^2$ and $\kappa_b^2$, central vertices of star subgraphs (again both isomorphic to $\Star_{15}$) $\mathcal S_a^2$ and $\mathcal S_b^2$, respectively, and also construct complete bipartite graph $\mathcal K^2$ with $15$ vertices in each of its partite sets $\mathcal K_a^2$ and $\mathcal K_b^2$. Set $\kappa_a^2$ and $\kappa_b^2$ adjacent to all the vertices in $V(\mathcal K^2)$. Proceed similarly: for $2 \leq \ell \leq L$, take two vertices of $\mathcal K^{\ell-1}$ in opposite partite sets and construct $\mathcal S_a^\ell$, $\mathcal S_b^\ell$, and $\mathcal K^\ell$, related as before, until all $n = 58L+2$ vertices are exhausted. We shall often refer to vertices $\kappa_a^\ell$ and $\kappa_b^\ell$ as \textit{\textcolor{red}{knob vertices}} of $Y_L$. Figure \ref{fig:Y_3} illustrates this construction for $L = 3$, while Figure \ref{fig:Y_3_text} provides a ``collapsed" view of our construction. 

\begin{figure}[ht]
\begin{center}
\includegraphics[width=0.5\textwidth]{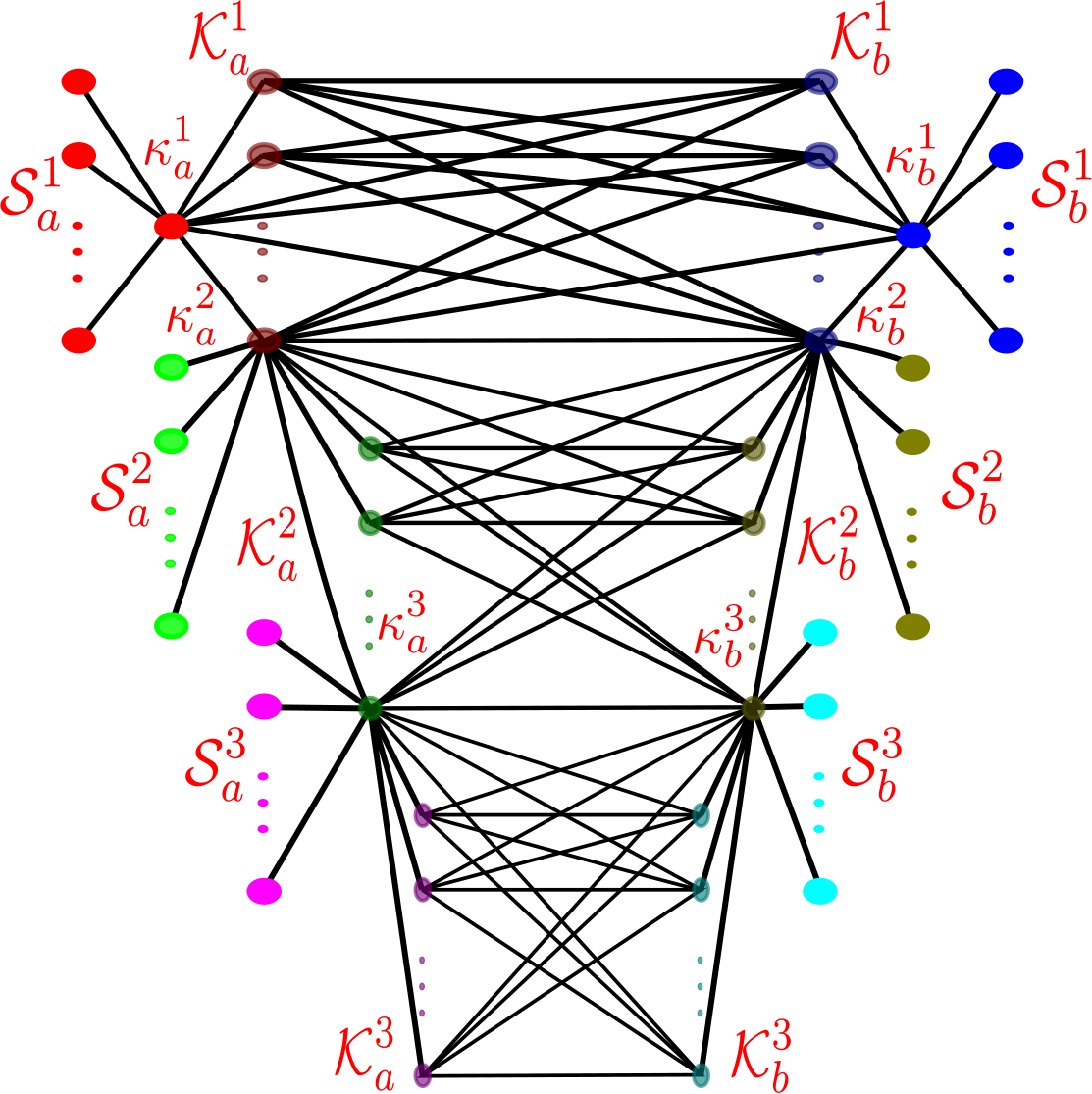} 
\caption{Labeled schematic diagram of the construction for $Y_3$. The vertices of $Y_3$ marked with a particular color correspond to the $\sigma_s$-images of the vertices of the same color in Figure \ref{fig:X_3}.}
\label{fig:Y_3}
\end{center}
\end{figure}

\subsubsection*{\textcolor{red}{The Starting Configuration $\sigma_s$ and its Connected Component $\mathscr{C}$.}} 

Take an arbitrary $L \geq 1$ and graphs $X_L$, $Y_L$. We are now going to describe a specific starting configuration $\sigma_s(X_L, Y_L) \in V(\FS(X_L,Y_L))$ which lies in the connected component $\mathscr{C}(X_L, Y_L)$ of $\FS(X_L,Y_L)$; we will later show that there exists a different configuration in $\mathscr{C}(X_L, Y_L)$ whose distance from $\sigma_s(X_L, Y_L)$ is $e^{\Omega(n)}$. Henceforth, we abbreviate $\sigma_s(X_L, Y_L)$ and $\mathscr{C}(X_L, Y_L)$ to $\sigma_s$ and $\mathscr{C}$. In forthcoming discussions, $X_L$ and $Y_L$ will be understood to be arbitrary such graphs on the same number of vertices.

Take all $15$ vertices in $V(\mathcal K_a^1)$ and place them onto $V(\mathcal P_a^1) \setminus \{v_a^1\}$, and the $15$ vertices in $V(\mathcal K_b^1)$ onto $V(\mathcal P_b^1)$; if $L > 1$, we place $\kappa_a^2$ onto the leftmost vertex of $V(\mathcal P_a^1)$ and $\kappa_b^2$ onto $v_b^1$. Now take subgraph $\mathcal S_a^1$ of $Y_L$: place $\kappa_a^1$ onto the middle vertex of the upper path between $v_a^1$ and $v^1$ (which has seven vertices), and place all $14$ leaves of $\mathcal S_a^1$ onto the remaining $14$ vertices of $V(\mathcal C_a^1) \setminus \{v^1\}$ in some way. Similarly, take $\mathcal S_b^1$: place $\kappa_b^1$ onto the middle vertex of the upper path between $v^1$ and $v_b^1$, and place all $14$ leaves of $\mathcal S_b^1$ onto the remaining $14$ vertices of $V(\mathcal C_b^1)$. This has filled all mappings on the subgraph $V(X^1)$ of $X_L$ by vertices in $V(\mathcal K^1)$, $V(\mathcal S_a^1)$, and $V(\mathcal S_b^1)$, and thus yields $\sigma_s$ if $L=1$.

Proceed sequentially according to the layer $\ell \in [L]$: say we placed all vertices of $V(\mathcal K^i)$, $V(\mathcal S_a^i)$, and $V(\mathcal S_b^i)$ for $i < \ell$ onto the corresponding $V(X^i)$ of $X_L$. Place all $15$ vertices in $V(\mathcal K_a^\ell)$ onto $V(\mathcal P_a^\ell) \setminus \{v_a^\ell\}$, and the $15$ vertices in $V(\mathcal K_b^\ell)$ onto $V(\mathcal P_b^\ell)$; if $L > \ell$, place $\kappa_a^{\ell+1}$ onto the leftmost vertex of $V(\mathcal P_a^\ell)$ and $\kappa_b^{\ell+1}$ onto $v_b^\ell$. Now take $\mathcal S_a^\ell$, and place its $14$ leaves onto the remaining $14$ vertices in $V(\mathcal C_a^\ell) \setminus \{v^\ell\}$. Similarly take $\mathcal S_b^\ell$, and place its $14$ leaves onto the $14$ remaining vertices in $V(\mathcal C_b^\ell)$. An illustration of this starting configuration is given in Figures \ref{fig:X_3} and \ref{fig:Y_3}: the vertices of a particular color in Figure \ref{fig:Y_3} are placed upon the correspondingly colored subgraph in Figure \ref{fig:X_3} to achieve $\sigma_s \in V(\FS(X_L, Y_L))$.

\begin{figure}
    \centering
    \includegraphics[width=0.7\textwidth]{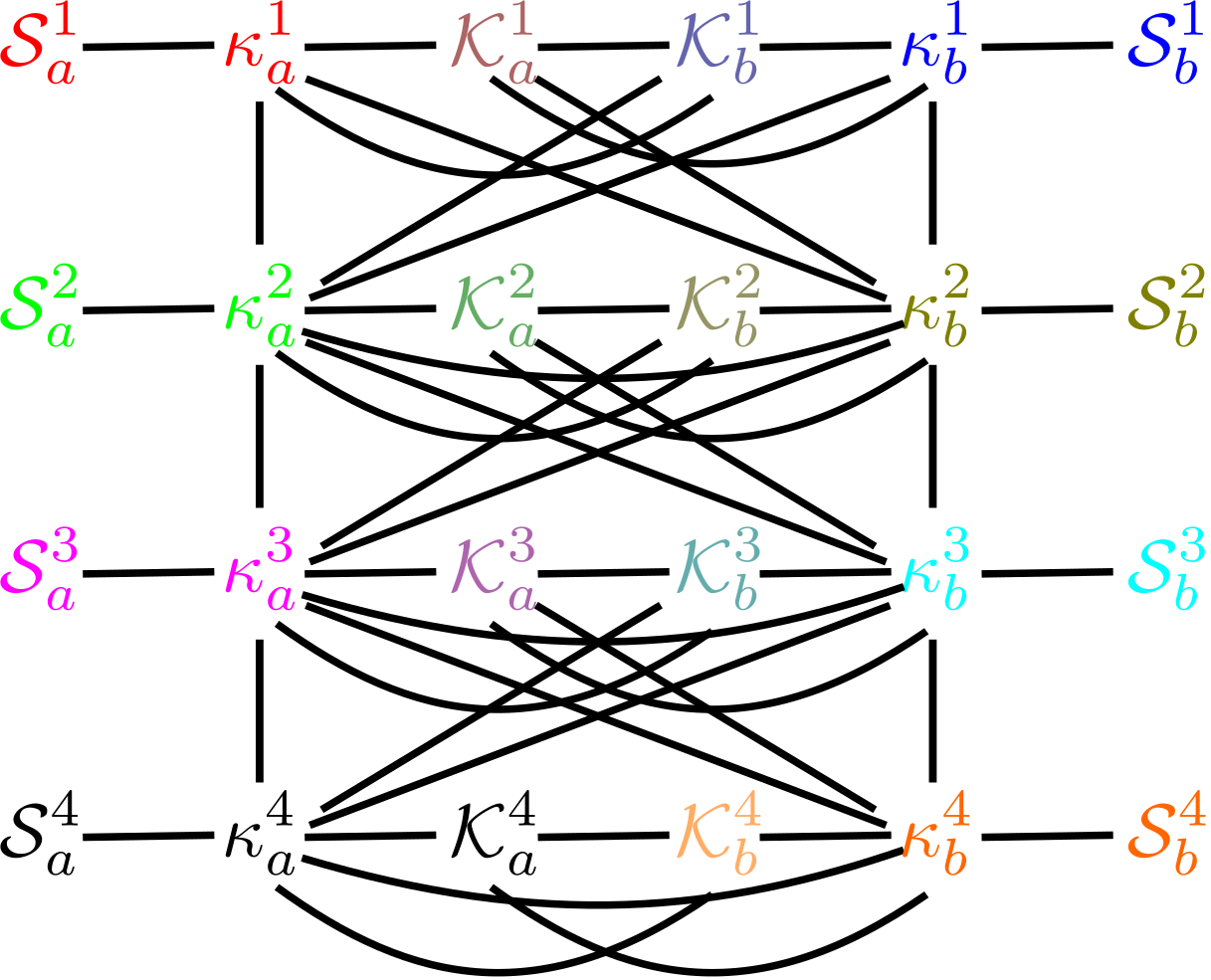}
    \caption{A simplified schematic diagram of $Y_4$ to illustrate the neighborhoods of different kinds of vertices. Here, all subgraphs $\mathcal S_a^\ell$, $\mathcal S_b^\ell$, $\mathcal K_a^\ell$, and $\mathcal K_b^\ell$ are to be understood as excluding any knob vertices.}
    \label{fig:Y_3_text}
\end{figure}

\begin{remark} \label{rule_of_two_remark}
By the construction of $\sigma_s \in V(\FS(X_L, Y_L))$, for any $\ell \in [L]$, 
\begin{align*}
    & V(\mathcal S_a^\ell) \setminus \{\kappa_a^\ell\} \subset \sigma_s(V(\mathcal C_a^\ell)),
    & V(\mathcal S_b^\ell) \setminus \{\kappa_b^\ell\} \subset \sigma_s(V(\mathcal C_b^\ell)).
\end{align*}
As such, all leaves of a star subgraph $\mathcal S_a^\ell$ or $\mathcal S_b^\ell$ of $Y_L$ are placed onto a corresponding cycle subgraph $\mathcal C_a^\ell$ or $\mathcal C_b^\ell$ of $X_L$, respectively. This yields that, for any $\ell \in [L]$,
\begin{align*}
    |\sigma_s(V(\mathcal C_a^\ell)) \setminus (V(\mathcal S_a^\ell) \setminus \{\kappa_a^\ell\})| = |\sigma_s(V(\mathcal C_b^\ell)) \setminus (V(\mathcal S_b^\ell) \setminus \{\kappa_b^\ell\})| = 2.
\end{align*}
In other words, the number of vertices upon any cycle subgraph $\mathcal C_a^\ell$ or $\mathcal C_b^\ell$ of $X_L$ which are not leaves of the corresponding star subgraph of $Y_L$, under $\sigma_s$, is exactly two.
\end{remark}

We introduce the following definition for notational convenience in forthcoming arguments.

\begin{definition} \label{defn:boundary}
    Fix $\ell \in [L]$.
    \begin{itemize}
        \item The \textit{\textcolor{red}{boundary}} $\bd(\mathcal C_a^\ell)$ of $\mathcal C_a^\ell$ is the subset of $\{v_a^\ell, v_a^{\ell-1,\ell}, v_a^{\ell, \ell+1}, v^\ell\}$ defined for $\ell$.
        \item The \textit{\textcolor{red}{boundary}} $\bd(\mathcal C_b^\ell)$ of $\mathcal C_b^\ell$ is the subset of $\{v_b^\ell, v_b^{\ell-1,\ell}, v_b^{\ell, \ell+1}, v^\ell\}$ defined for $\ell$.
    \end{itemize}
\end{definition}

In Subsections \ref{subsec:configurations_in_c} and \ref{subsec:extractions}, unless otherwise stated, we fix an arbitrary integer $L \geq 1$ and refer to the graphs $X_L$ and $Y_L$, with $\sigma_s$ denoting the corresponding starting configuration. We elect to refer to paths in $\FS(X_L, Y_L)$ as \textit{\textcolor{red}{swap sequences}}, which are denoted by the vertices and edges in $\FS(X_L, Y_L)$ that constitute the path. More specifically, a swap sequence of length $\lambda$ is a sequence of vertices $\Sigma = \{\sigma_i\}_{i=0}^\lambda \subseteq V(\FS(X_L, Y_L))$ for which $\{\sigma_{i-1}, \sigma_{i}\} \in E(\FS(X_L, Y_L))$ for all $i \in [\lambda]$.

\subsection{Configurations in \texorpdfstring{$\mathscr{C}$}{C}} \label{subsec:configurations_in_c}

In this subsection, we derive properties satisfied by all vertices in $\mathscr{C}$. Intuitively, our aim in this subsection is to uncover many conditions satisfied by all of the vertices in $\mathscr{C}$, which has the effect of producing strong rigidities on the corresponding swapping problem. These rigidities will allow us to argue in Subsection \ref{subsec:extractions} that in order to move certain vertices in $Y_L$ down and across the graph $X_L$, we necessarily must perform very specific sequences of swaps.

Remark \ref{rule_of_two_remark} observes that in the starting configuration $\sigma_s$, the leaves of any star graph $\mathcal S_a^\ell$ or $\mathcal S_b^\ell$ lie upon the vertices of $\mathcal C_a^\ell$ and $\mathcal C_b^\ell$, respectively. In particular, for any cycle subgraph $\mathcal C_a^\ell$ in $X_L$, exactly two vertices that are not leaves of $\mathcal S_a^\ell$ lie upon them; an analogous statement holds for cycle subgraphs of the form $\mathcal C_b^\ell$. We begin our study of $\mathscr{C}$ by establishing that this property is maintained after any sequence of swaps in $\FS(X_L, Y_L)$ beginning at $\sigma_s$, i.e., that all vertices in $\mathscr{C}$ satisfy this property: we prove this in Proposition \ref{prop:rule_of_two}.

\begin{proposition} \label{prop:rule_of_two}
Any $\sigma \in V(\mathscr{C})$ satisfies, for all $\ell \in [L]$,
\begin{align*}
    V(\mathcal S_a^\ell) \setminus \{\kappa_a^\ell\} \subset \sigma(V(\mathcal C_a^\ell)) \text{ and } V(\mathcal S_b^\ell) \setminus \{\kappa_b^\ell\} \subset \sigma(V(\mathcal C_b^\ell)).
\end{align*}
\end{proposition}

As in Remark \ref{rule_of_two_remark}, this means that for any cycle subgraph $\mathcal C_a^\ell$ or $\mathcal C_b^\ell$ in $X_L$ and $\sigma \in V(\mathscr{C})$,
\begin{align*}
    |\sigma(V(\mathcal C_a^\ell)) \setminus (V(\mathcal S_a^\ell) \setminus \{\kappa_a^\ell\})| = |\sigma(V(\mathcal C_b^\ell)) \setminus (V(\mathcal S_b^\ell) \setminus \{\kappa_b^\ell\})| = 2,
\end{align*}
since $|V(\mathcal S_a^\ell) \setminus \{\kappa_a^\ell\}| = |V(\mathcal S_b^\ell) \setminus \{\kappa_b^\ell\}| = 14$, and $|V(\mathcal C_a^\ell)| = |V(\mathcal C_b^\ell)| = 16$ for all $\ell \in [L]$. 

\begin{remark}
    Although Proposition \ref{prop:rule_of_two} describes a global property maintained by all configurations in $\mathscr{C}$, we frequently appeal to it (for sake of brevity) as a local property satisfied by specific configurations in $\mathscr{C}$ during the proof of Proposition \ref{prop:rule_of_two}.\footnote{Making this clarification is important, as the proof proceeds by assuming (towards a contradiction) that Proposition \ref{prop:rule_of_two} is satisfied by particular configurations in $\mathscr{C}$ and is violated by another.} This practice of localizing a more global statement to a particular configuration will also be utilized for other results in later proofs in this section, and it should not raise any ambiguity whenever it is invoked.
\end{remark}

\begin{proof}[Proof of Proposition \ref{prop:rule_of_two}]
Assume (towards a contradiction) that the proposition is false, so there exists a swap sequence $\Sigma = \{\sigma_i\}_{i=0}^\lambda$ with $\sigma_0 = \sigma_s$ in $\mathscr{C}$ of shortest length $\lambda$ containing a vertex violating Proposition \ref{prop:rule_of_two}: $\sigma_\lambda$ violates Proposition \ref{prop:rule_of_two}, while all $\sigma_i$ for $i < \lambda$ satisfy it, and $\lambda \geq 1$. Thus, there exists a star subgraph $\mathcal S$ (of form $\mathcal S_a^\ell$ or $\mathcal S_b^\ell$) of $Y_L$ and a leaf $\mu \in V(\mathcal S)$ such that $\sigma_{\lambda-1}^{-1}(\mu)$ is in the appropriate cycle subgraph, but $\sigma_\lambda^{-1}(\mu)$ is not. Say $\mathcal S = \mathcal S_a^\ell$ for $\ell \in [L]$: raising a contradiction when $\mathcal S = \mathcal S_b^\ell$ is entirely analogous. Here, $\mu \in V(\mathcal S_a^\ell) \setminus \{\kappa_a^\ell\}$ has $N_{Y_L}(\mu) = \{\kappa_a^\ell\}$ and $\sigma_{\lambda-1}^{-1}(\mu) \in V(\mathcal C_a^\ell)$, $\sigma_\lambda^{-1}(\mu) \notin V(\mathcal C_a^\ell)$, so $\sigma_{\lambda-1}^{-1}(\mu) \in \bd(\mathcal C_a^\ell)$ and $\sigma_\lambda$ is reached from $\sigma_{\lambda-1}$ by swapping $\mu$ and $\kappa_a^\ell$. Figure \ref{fig:rule_of_two} depicts the configurations described in the following two cases.

\subsubsection*{\textbf{Case 1: $\sigma_{\lambda-1}^{-1}(\mu) = v_a^\ell$.}} Here, $\sigma_{\lambda-1}^{-1}(\kappa_a^\ell) \in N_{X_L}(v_a^\ell) \cap V(\mathcal P_a^\ell)$. Let $\xi < \lambda-1$ be the final such index with $\sigma_\xi^{-1}(\kappa_a^\ell) \notin V(\mathcal P_a^\ell) \setminus \{v_a^\ell\}$; $\xi$ is well-defined since 
\begin{align*}
    \sigma_s^{-1}(\kappa_a^\ell) \notin V(\mathcal P_a^\ell) \setminus \{v_a^\ell\},
\end{align*}
which implies $\sigma_s \neq \sigma_{\lambda-1}$, so $\lambda \geq 2$. By the definition of $\xi$ and $\sigma_{\lambda-1}^{-1}(\kappa_a^\ell) \in N_{X_L}(v_a^\ell) \cap V(\mathcal P_a^\ell)$,
\begin{align} \label{eq:rule_of_two_eq_1}
    \sigma_j^{-1}(\kappa_a^\ell) \in V(\mathcal P_a^\ell) \setminus \{v_a^\ell\} \text{ for } \xi + 1 \leq j \leq \lambda-1.
\end{align}
Necessarily, $\sigma_\xi^{-1}(\kappa_a^\ell) = v_a^\ell$ and $\sigma_{\xi+1}^{-1}(\kappa_a^\ell) \in N_{X_L}(v_a^\ell) \cap V(\mathcal P_a^\ell)$, so 
\begin{align*}
    \sigma_\xi^{-1}(\mu) = \sigma_{\xi+1}^{-1}(\mu) \in V(\mathcal C_a^\ell) \setminus \{v_a^\ell\};
\end{align*}
note that $\sigma_\xi$ satisfies Proposition \ref{prop:rule_of_two}. Since $N_{Y_L}(\mu) = \{\kappa_a^\ell\}$ and there are no edges between $V(\mathcal C_a^\ell) \setminus \{v_a^\ell\}$ and $V(\mathcal P_a^\ell) \setminus \{v_a^\ell\}$, it follows from \eqref{eq:rule_of_two_eq_1} that $\sigma_j^{-1}(\mu)$ is fixed for $\xi \leq j \leq \lambda-1$, so
\begin{align*}
    \sigma_\xi^{-1}(\mu) = \sigma_{\lambda-1}^{-1}(\mu) \in V(\mathcal C_a^\ell) \setminus \{v_a^\ell\},
\end{align*}
contradicting $\sigma_{\lambda-1}^{-1}(\mu) = v_a^\ell$.

\subsubsection*{\textbf{Case 2: $\sigma_{\lambda-1}^{-1}(\mu) \neq v_a^\ell$.}} Here, $\sigma_{\lambda-1}^{-1}(\mu) \in \bd(\mathcal C_a^\ell) \setminus \{v_a^\ell\}$, and
\begin{align*}
    \sigma_{\lambda-1}^{-1}(\kappa_a^\ell) = \sigma_\lambda^{-1}(\mu) \in N_{X_L}(\sigma_{\lambda-1}^{-1}(\mu)) \setminus V(\mathcal C_a^\ell).
\end{align*}
Proceeding backwards in $\Sigma$, it must be that either 
\begin{align*}
    \sigma_{\lambda-2}^{-1}(\mu) \neq \sigma_{\lambda-1}^{-1}(\mu) \text{ or } \sigma_{\lambda-2}^{-1}(\kappa_a^\ell) \neq \sigma_{\lambda-1}^{-1}(\kappa_a^\ell);
\end{align*}
note that $\lambda \geq 2$, since $\sigma_s^{-1}(\kappa_a^\ell) \neq \sigma_{\lambda-1}^{-1}(\kappa_a^\ell)$. Indeed, if not, swapping $\mu$ and $\kappa_a^\ell$ directly from $\sigma_{\lambda-2}$ raises a contradiction on $\lambda$ being minimal. Now, $N_{Y_L}(\mu) = \{\kappa_a^\ell\}$ implies 
\begin{align*}
    \sigma_{\lambda-2}^{-1}(\kappa_a^\ell) \neq \sigma_{\lambda-1}^{-1}(\kappa_a^\ell) \text{ and } \sigma_{\lambda-2}^{-1}(\mu) = \sigma_{\lambda-1}^{-1}(\mu),
\end{align*}
since if both preimages differ, $\sigma_{\lambda-2} = \sigma_\lambda$. Thus, $\sigma_{\lambda-2}^{-1}(\kappa_a^\ell) \notin V(\mathcal C_a^\ell)$ and 
\begin{align*}
    \sigma_{\lambda-2}(\sigma_{\lambda-1}^{-1}(\kappa_a^\ell)) \notin V(\mathcal S_a^\ell) \setminus \{\kappa_a^\ell\}
\end{align*}
by Proposition \ref{prop:rule_of_two} (on $\sigma_{\lambda-2}$), so $\sigma_{\lambda-2}(\sigma_{\lambda-1}^{-1}(\kappa_a^\ell))$ is not a leaf (see $N_{Y_L}(\kappa_a^\ell)$). We further assume that
\begin{align*}
    \sigma_{\lambda-2}^{-1}(\mu) = \sigma_{\lambda-1}^{-1}(\mu) = v^\ell;
\end{align*}
raising a contradiction for the cases $v_a^{\ell-1, \ell}$ and $v_a^{\ell, \ell+1}$ can be done analogously. So 
\begin{align*}
    \sigma_{\lambda-2}^{-1}(\{\kappa_a^\ell,\mu\}) \subset V(\mathcal C_b^\ell) \text{ and } \sigma_{\lambda-2}(\sigma_{\lambda-1}^{-1}(\kappa_a^\ell)) \in V(\mathcal C_b^\ell).
\end{align*}
Altogether, we have that
\begin{align*}
    |\sigma_{\lambda-2}(V(\mathcal C_b^\ell)) \setminus (V(\mathcal S_b^\ell) \setminus \{\kappa_b^\ell\})| \geq 3, 
\end{align*}
and since $|V(\mathcal C_b^\ell)| = 16$ and $|V(\mathcal S_b^\ell) \setminus \{\kappa_b^\ell\}| = 14$, $V(\mathcal S_b^\ell) \setminus \{\kappa_b^\ell\} \not\subset \sigma_{\lambda-2}(V(\mathcal C_b^\ell))$. Thus, $\sigma_{\lambda-2}$ violates Proposition \ref{prop:rule_of_two}, contradicting $\lambda$ being minimal.
\begin{figure}[ht]
  \centering
  \subfloat[Case $1$. After $\sigma_\xi$, $\kappa_a^\ell$ does not exit $V(\mathcal P_a^\ell) \setminus \{v_a^\ell\}$, leading to a contradiction on the placement of $\mu$ in $\sigma_{\lambda-1}$.]{\includegraphics[width=0.42\textwidth]{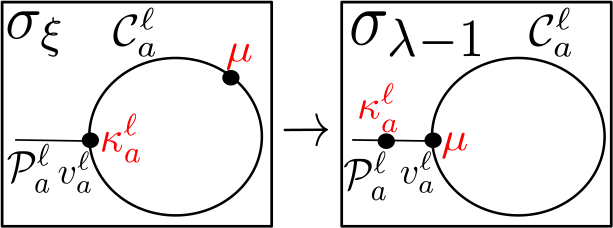}\label{fig:rule_of_two_1}}
  \hfill
  \subfloat[Case $2$. Here, $\sigma_{\lambda-1}$ results by swapping $\kappa_a^\ell$ along $\mathcal C_b^\ell$, so that $\sigma_{\lambda-2}$ violates Proposition \ref{prop:rule_of_two} on $\mathcal C_b^\ell$.]{\includegraphics[width=0.53\textwidth]{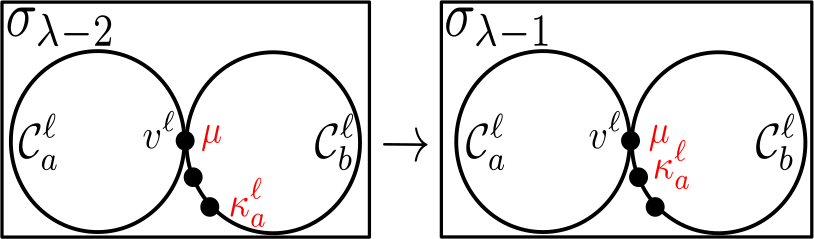}\label{fig:rule_of_two_2}}
  \caption{Configurations in $\Sigma$ raising a contradiction for both cases in the proof of Proposition \ref{prop:rule_of_two}.}
  \label{fig:rule_of_two}
\end{figure}
\end{proof}

Proposition \ref{prop:rule_of_two} restricts the preimages of the leaves of $\mathcal S_a^\ell$ and $\mathcal S_b^\ell$ under any $\sigma \in V(\mathscr{C})$. We now derive a restriction on the preimages of all other vertices in $V(Y_L)$ under any $\sigma \in V(\mathscr{C})$. As Proposition \ref{prop:layer_independence} formalizes, for such $\sigma$, any vertex in $V(Y_L)$ is close to its preimage in $\sigma_s$. 

\begin{restatable}{proposition} {layerIndependence}\label{prop:layer_independence}
Any configuration $\sigma \in V(\mathscr{C})$ must satisfy the following four properties.
\begin{enumerate}
    \item The layer $1$ knob vertices lie upon the corresponding subgraph of $X^1$, i.e.,
    \begin{align*}
        \sigma^{-1}(\kappa_a^1) \in V(X_a^1) \text{ and } \sigma^{-1}(\kappa_b^1) \in V(X_b^1).
    \end{align*}
    \item For $2 \leq \ell \leq L$, the layer $\ell$ knob vertices lie upon the subgraphs $X^{\ell-1}$ or $X^\ell$, i.e.,
    \begin{align*}
        \{\sigma^{-1}(\kappa_a^\ell), \sigma^{-1}(\kappa_b^\ell)\} \subset V(X^{\ell-1}) \cup V(X^\ell).
    \end{align*}
    \item For $\ell \in [L-1]$, any vertex in $V(\mathcal K^\ell)$ that is not a layer $\ell+1$ knob lies upon $X^\ell$, i.e.,
    \begin{align*}
        \sigma^{-1}(V(\mathcal K^\ell) \setminus \{\kappa_a^{\ell+1}, \kappa_b^{\ell+1}\}) \subset V(X^{\ell}),
    \end{align*}
    and every vertex in $V(\mathcal K^L)$ lies upon $X^L$, i.e., 
    \begin{align*}
        \sigma^{-1}(V(\mathcal K^L)) \subset V(X^L).
    \end{align*}
    \item For $\ell \in [L]$, there is at most one $\mu \in V(\mathcal K^\ell)$ not in $V(\mathcal P_a^\ell) \cup V(\mathcal P_b^\ell)$, i.e.,
    \begin{align*}
        |\sigma^{-1}(V(\mathcal K^\ell)) \setminus (V(\mathcal P_a^\ell) \cup V(\mathcal P_b^\ell))| \leq 1.
    \end{align*}
\end{enumerate}
\end{restatable}
Confirming that the starting configuration $\sigma_s$ satisfies these four properties is straightforward. Case 4 of the proof of Proposition \ref{prop:layer_independence} relies on the following Lemma \ref{lem:two_statements}, which is illustrated in Figure \ref{fig:two_statements}. In the statement of the lemma, we elect to index the final term of the swap sequence by $\lambda-1$ as this is where the result applies in the proof of Proposition \ref{prop:layer_independence}.
\begin{lemma} \label{lem:two_statements}
    Let $\Sigma = \{\sigma_i\}_{i=0}^{\lambda-1}$ with $\sigma_0 = \sigma_s$, $\lambda \geq 1$ be a swap sequence in $\mathscr{C}$ such that for all $1 \leq i \leq \lambda-1$, $\sigma_i$ satisfies the four properties of Proposition \ref{prop:layer_independence}. Then for all $0 \leq i \leq \lambda-1$ and $\ell \in [L]$, the following two statements hold.
\begin{enumerate}
    \item If $\sigma_i^{-1}(V(\mathcal K^\ell)) \subset V(\mathcal P_a^\ell) \cup V(\mathcal P_b^\ell)$, then
    \begin{align*}
        \sigma_i(\{v_a^\ell, v_b^\ell\}) \subset V(\mathcal K^\ell) \implies |\sigma_i^{-1}(\{\kappa_a^\ell, \kappa_b^\ell\}) \cap ((V(\mathcal P_a^\ell) \setminus \{v_a^\ell\}) \cup (V(\mathcal P_b^\ell) \setminus \{v_b^\ell\}))| = 1.
    \end{align*}
    \item If $\sigma_i^{-1}(V(\mathcal K^\ell)) \not\subset V(\mathcal P_a^\ell) \cup V(\mathcal P_b^\ell)$, then 
    \begin{align*}
        & \sigma_i(v_a^\ell) \in V(\mathcal K^\ell) \implies \sigma_i^{-1}(\{\kappa_a^\ell, \kappa_b^\ell\}) \cap (V(\mathcal P_a^\ell) \setminus \{v_a^\ell\}) \neq \emptyset, \\
        & \sigma_i(v_b^\ell) \in V(\mathcal K^\ell) \implies \sigma_i^{-1}(\{\kappa_a^\ell, \kappa_b^\ell\}) \cap (V(\mathcal P_b^\ell) \setminus \{v_b^\ell\}) \neq \emptyset.
    \end{align*}
\end{enumerate}
\end{lemma}

\begin{figure}[ht]
  \centering
  \subfloat[Lemma \ref{lem:two_statements}(1) on the configuration $\sigma_i$: if $\sigma_i^{-1}(V(\mathcal K^\ell)) \subset V(\mathcal P_a^\ell) \cup V(\mathcal P_b^\ell)$ and $\sigma_i(\{v_a^\ell, v_b^\ell\}) \subset V(\mathcal K^\ell)$, then either $\kappa_a^\ell$ or $\kappa_b^\ell$ lies upon $V(\mathcal P_a^\ell) \setminus \{v_a^\ell\}$ or $V(\mathcal P_b^\ell) \setminus \{v_b^\ell\}$.]{\includegraphics[width=0.495\textwidth]{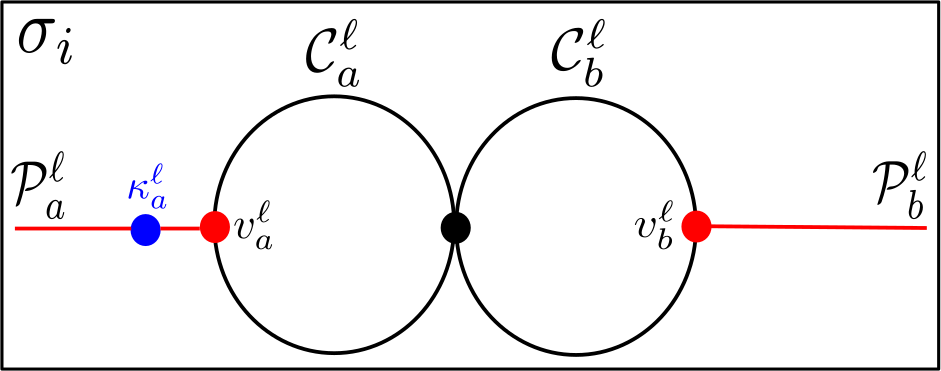}\label{fig:two_statements_1}}
  \hfill
  \subfloat[First implication of Lemma \ref{lem:two_statements}(2) on the configuration $\sigma_i$: if there exists some $\mu \in V(\mathcal K^\ell)$ with $\sigma_i^{-1}(\mu) \notin V(\mathcal P_a^\ell) \cup V(\mathcal P_b^\ell)$, and $\sigma_i(v_a^\ell) \in V(\mathcal K^\ell)$, then either $\kappa_a^\ell$ or $\kappa_b^\ell$ lies upon $V(\mathcal P_a^\ell) \setminus \{v_a^\ell\}$.]{\includegraphics[width=0.495\textwidth]{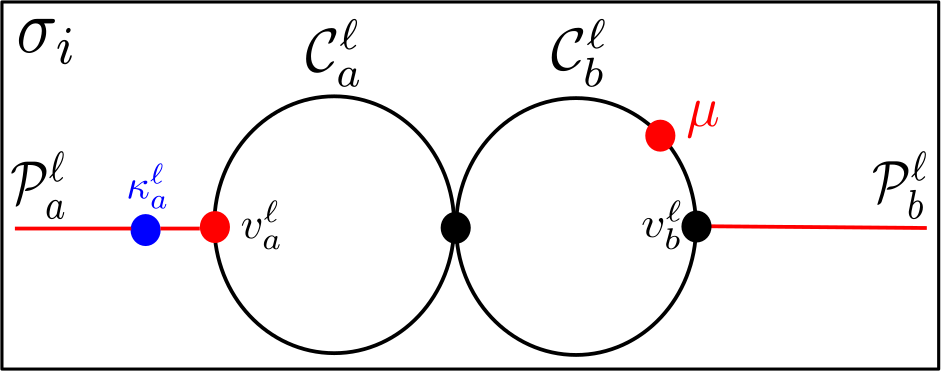}\label{fig:two_statements_2}}
  \caption{Illustrations for both parts of Lemma \ref{lem:two_statements} for some $\sigma_i \in \Sigma$. Subgraphs/vertices colored in red correspond to $\sigma_i$-preimages of $V(\mathcal K^\ell)$, while $\sigma_i$-preimages of elements in $\{\kappa_a^\ell, \kappa_b^\ell\}$ are colored in blue. For Figure \ref{fig:two_statements_2}, note that by appealing to Proposition \ref{prop:layer_independence}(4) and comparing cardinalities, we can deduce that at most two vertices of $\sigma_i(V(\mathcal P_a^\ell) \cup V(\mathcal P_b^\ell))$ can fail to lie in $V(\mathcal K^\ell)$.}
  \label{fig:two_statements}
\end{figure}

\begin{proof}[Proof of Lemma \ref{lem:two_statements}]
Fix $\Sigma = \{\sigma_i\}_{i=0}^{\lambda-1}$ to be a swap sequence satisfying the assumptions of Lemma \ref{lem:two_statements}. We prove the two statements of Lemma \ref{lem:two_statements} hold for all $\ell \in [L]$ inductively for $0 \leq i \leq \lambda-1$. They can be checked to hold for all $\ell \in [L]$ when $i=0$, so assume they are true for some $0 \leq i < \lambda-1$. We prove that $\sigma_{i+1}$ satisfies both statements for all $\ell \in [L]$. In what follows, assume we refer (unless stated otherwise) to some fixed, arbitrary $\ell \in [L]$. We break into cases based on whether or not $\sigma_i^{-1}(\mu) \subset V(\mathcal P_a^\ell) \cup V(\mathcal P_b^\ell)$. 

\medskip

\paragraph{\textbf{Case 1: $\sigma_i^{-1}(V(\mathcal K^\ell)) \subset V(\mathcal P_a^\ell) \cup V(\mathcal P_b^\ell)$.}} We will further break into subcases based on whether or not $\sigma_i(\{v_a^\ell, v_b^\ell\}) \subset V(\mathcal K^\ell)$.

\smallskip

\paragraph{\textbf{Subcase 1.1: $\sigma_i(\{v_a^\ell, v_b^\ell\}) \subset V(\mathcal K^\ell)$.}} By the induction hypothesis, we have that
\begin{align} \label{eq:two_statements_lemma_eq_1}
    |\sigma_i^{-1}(\{\kappa_a^\ell, \kappa_b^\ell\}) \cap ((V(\mathcal P_a^\ell) \setminus \{v_a^\ell\}) \cup (V(\mathcal P_b^\ell) \setminus \{v_b^\ell\}))| = 1.
\end{align}
If $\sigma_i(v_a^\ell) = \sigma_{i+1}(v_a^\ell)$ and $\sigma_i(v_b^\ell) = \sigma_{i+1}(v_b^\ell)$, then $\sigma_{i+1}$ satisfies Lemma \ref{lem:two_statements}(1) since
\begin{align*}
    & \sigma_{i+1}^{-1}(V(\mathcal K^\ell)) \subset V(\mathcal P_a^\ell) \cup V(\mathcal P_b^\ell),
    & |\sigma_{i+1}^{-1}(\{\kappa_a^\ell, \kappa_b^\ell\}) \cap ((V(\mathcal P_a^\ell) \setminus \{v_a^\ell\}) \cup (V(\mathcal P_b^\ell) \setminus \{v_b^\ell\}))| = 1,
\end{align*}
and satisfies Lemma \ref{lem:two_statements}(2) trivially.\footnote{Generally, in what follows, we do not comment on the ``other statement" in Lemma \ref{lem:two_statements} holding trivially, and only check the statement which applies, depending on whether $\sigma_{i+1}^{-1}(V(\mathcal K^\ell)) \subset V(\mathcal P_a^\ell) \cup V(\mathcal P_b^\ell)$ or not in the given context.} So consider the setting where either $\sigma_i(v_a^\ell) \neq \sigma_{i+1}(v_a^\ell)$ or $\sigma_i(v_b^\ell) \neq \sigma_{i+1}(v_b^\ell)$: say $\sigma_i(v_a^\ell) \neq \sigma_{i+1}(v_a^\ell)$ (the setting $\sigma_i(v_b^\ell) \neq \sigma_{i+1}(v_b^\ell)$ is analogous). If 
\begin{align*}
    \sigma_{i+1}^{-1}(\sigma_i(v_a^\ell)) \in N_{X_L}(v_a^\ell) \cap V(\mathcal P_a^\ell),
\end{align*}
then $\sigma_{i+1}$ satisfies Lemma \ref{lem:two_statements}(1). Indeed, since $\kappa_a^\ell, \kappa_b^\ell \notin V(\mathcal K^\ell)$, the hypothesis $\sigma_{i+1}(\{v_a^\ell, v_b^\ell\}) \subset V(\mathcal K^\ell)$ implies that $\sigma_i^{-1}(\kappa_a^\ell) = \sigma_{i+1}^{-1}(\kappa_a^\ell)$ and $\sigma_i^{-1}(\kappa_b^\ell) = \sigma_{i+1}^{-1}(\kappa_b^\ell)$. If 
\begin{align*}
    \sigma_{i+1}^{-1}(\sigma_i(v_a^\ell)) \notin V(\mathcal P_a^\ell),
\end{align*}
then since $\sigma_i(v_a^\ell) \in V(\mathcal K^\ell)$, we have that
\begin{align*}
    \sigma_{i+1}^{-1}(V(\mathcal K^\ell)) \not\subset V(\mathcal P_a^\ell) \cup V(\mathcal P_b^\ell).
\end{align*}
From studying the neighborhoods of vertices in $V(\mathcal K^\ell)$ to produce possibilities for $\sigma_{i+1}(v_a^\ell)$, Propositions \ref{prop:rule_of_two} and \ref{prop:layer_independence}(2,3)\footnote{Indeed, $\sigma_{i+1}$ would violate Proposition \ref{prop:layer_independence}(2) if $\sigma_{i+1}(v_a^\ell) \in \{\kappa_a^{\ell+2}, \kappa_b^{\ell+2}\}$ and Proposition \ref{prop:layer_independence}(3) if $\sigma_{i+1}(v_a^\ell) \in V(\mathcal K^{\ell+1}) \setminus \{\kappa_a^{\ell+2}, \kappa_b^{\ell+2}\}$. Henceforth, we do not explicitly make such further distinctions when appealing to multiple properties from Proposition \ref{prop:layer_independence} together.} imply 
\begin{align*}
    \sigma_{i+1}(v_a^\ell) \in V(\mathcal K^\ell) \cup \{\kappa_a^\ell, \kappa_b^\ell\}
\end{align*}
(consider the possible vertices in $N_{Y_L}(\sigma_i(v_a^\ell))$), from which $\sigma_i^{-1}(V(\mathcal K^\ell)) \subset V(\mathcal P_a^\ell) \cup V(\mathcal P_b^\ell)$ implies 
\begin{align*}
    \sigma_{i+1}(v_a^\ell) \in \{\kappa_a^\ell, \kappa_b^\ell\}.
\end{align*}
This yields $\ell = 1$. Indeed, if $\ell \geq 2$, then $\sigma_i$ violates Proposition \ref{prop:layer_independence}(4) on layer $\ell-1$, since with \eqref{eq:two_statements_lemma_eq_1},
\begin{align*}
    \sigma_i^{-1}(\{\kappa_a^\ell, \kappa_b^\ell\}) \subset \sigma_i^{-1}(V(\mathcal K^{\ell-1})) \setminus (V(\mathcal P_a^{\ell-1}) \cup V(\mathcal P_b^{\ell-1})),
\end{align*}
so that $\sigma_{i+1}(v_a^1) = \kappa_a^1$ by Proposition \ref{prop:layer_independence}(1) on $\sigma_{i+1}$. This result, with Proposition \ref{prop:layer_independence}(1) (on $\sigma_i$) and \eqref{eq:two_statements_lemma_eq_1}, yields 
\begin{align*}
    \sigma_i^{-1}(\kappa_b^1) = \sigma_{i+1}^{-1}(\kappa_b^1) \in V(\mathcal P_b^1) \setminus \{v_b^1\},
\end{align*}
so $\sigma_{i+1}$ satisfies Lemma \ref{lem:two_statements}(2). 

\smallskip

\paragraph{\textbf{Subcase 1.2: $\sigma_i(\{v_a^\ell, v_b^\ell\}) \not\subset V(\mathcal K^\ell)$.}} Since $|V(\mathcal P_a^\ell) \cup V(\mathcal P_b^\ell)| = 31$ and $|V(\mathcal K^\ell)| = 30$, and recalling our initial assumption $\sigma_i^{-1}(V(\mathcal K^\ell)) \subset V(\mathcal P_a^\ell) \cup V(\mathcal P_b^\ell)$, we have that
\begin{align} \label{eq:two_statements_lemma_eq_2}
    & |\sigma_i(\{v_a^\ell, v_b^\ell\}) \cap V(\mathcal K^\ell)| = 1,
    & \sigma_i((V(\mathcal P_a^\ell) \setminus \{v_a^\ell\}) \cup (V(\mathcal P_b^\ell) \setminus \{v_b^\ell\})) \subset V(\mathcal K^\ell).
\end{align}
Say $\sigma_i(v_a^\ell) \in V(\mathcal K^\ell)$; the setting $\sigma_i(v_b^\ell) \in V(\mathcal K^\ell)$ is argued analogously.  By \eqref{eq:two_statements_lemma_eq_2}, $\sigma_i(v_b^\ell) \notin V(\mathcal K^\ell)$. If $\sigma_{i+1}^{-1}(\sigma_i(v_a^\ell)) \notin V(\mathcal P_a^\ell)$, then $\sigma_i^{-1}(V(\mathcal K^\ell)) \subset V(\mathcal P_a^\ell) \cup V(\mathcal P_b^\ell)$ and $\sigma_i(v_b^\ell) = \sigma_{i+1}(v_b^\ell) \notin V(\mathcal K^\ell)$, implying
\begin{align*}
    \sigma_{i+1}^{-1}(V(\mathcal K^\ell)) \not\subset V(\mathcal P_a^\ell) \cup V(\mathcal P_b^\ell) \text{ and } \sigma_{i+1}(v_a^\ell) \notin V(\mathcal K^\ell),
\end{align*}
so $\sigma_{i+1}$ satisfies Lemma \ref{lem:two_statements}(2). If $\sigma_{i+1}^{-1}(\sigma_i(v_a^\ell)) \in V(\mathcal P_a^\ell)$, then $\sigma_i^{-1}(V(\mathcal K^\ell)) \subset V(\mathcal P_a^\ell) \cup V(\mathcal P_b^\ell)$ and $\sigma_i(v_b^\ell) \notin V(\mathcal K^\ell)$ yield 
\begin{align} \label{eq:two_statements_lemma_eq_3}
    \sigma_{i+1}^{-1}(V(\mathcal K^\ell)) \subset V(\mathcal P_a^\ell) \cup V(\mathcal P_b^\ell).
\end{align}
Studying the neighborhoods of vertices in $V(\mathcal K^\ell)$ and recalling that $\sigma_i(v_b^\ell) \notin V(\mathcal K^\ell)$ yields that the only way we can have that $\sigma_{i+1}(v_b^\ell) \in V(\mathcal K^\ell)$ (exactly when $\sigma_{i+1}$ does not trivially satisfy Lemma \ref{lem:two_statements}(1)) without $\sigma_{i+1}$ violating Proposition \ref{prop:rule_of_two} or \ref{prop:layer_independence}(2,3) is if 
\begin{align*}
    \sigma_i(v_b^\ell) \in \{\kappa_a^\ell, \kappa_b^\ell\} \text{ and } \sigma_{i+1}^{-1}(\sigma_i(v_b^\ell)) \in N_{X_L}(v_b^\ell) \cap V(\mathcal P_b^\ell).
\end{align*}
These results, along with \eqref{eq:two_statements_lemma_eq_3}, $|V(\mathcal P_a^\ell) \cup V(\mathcal P_b^\ell)| = 31$, $|V(\mathcal K^\ell)| = 30$, and $\kappa_a^\ell, \kappa_b^\ell \notin V(\mathcal K^\ell)$, imply
\begin{align*}
    |\sigma_{i+1}^{-1}(\{\kappa_a^\ell, \kappa_b^\ell\}) \cap ((V(\mathcal P_a^\ell) \setminus \{v_a^\ell\}) \cup (V(\mathcal P_b^\ell) \setminus \{v_b^\ell\}))| = 1,
\end{align*}
so $\sigma_{i+1}$ satisfies Lemma \ref{lem:two_statements}(1). 

\medskip

\paragraph{\textbf{Case 2: $\sigma_i^{-1}(V(\mathcal K^\ell)) \not\subset V(\mathcal P_a^\ell) \cup V(\mathcal P_b^\ell)$.}} By Proposition \ref{prop:layer_independence}(4) (on $\sigma_i$), there exists a unique $\mu \in V(\mathcal K^\ell)$ such that $\sigma_i^{-1}(\mu) \notin V(\mathcal P_a^\ell) \cup V(\mathcal P_b^\ell)$, so $|V(\mathcal P_a^\ell) \cup V(\mathcal P_b^\ell)| = 31$ and $|V(\mathcal K^\ell)| = 30$ yield
\begin{align} \label{eq:two_statements_lemma_eq_4}
    |(V(\mathcal P_a^\ell) \cup V(\mathcal P_b^\ell)) \setminus \sigma_i^{-1}(V(\mathcal K^\ell))| = 2.
\end{align}
We break into subcases based on the subset of $\{\sigma_i(v_a^\ell), \sigma_i(v_b^\ell)\}$ that is in $V(\mathcal K^\ell)$. 

\smallskip

\paragraph{\textbf{Subcase 2.1: $\sigma_i(\{v_a^\ell, v_b^\ell\}) \subset V(\mathcal K^\ell)$.}} By the induction hypothesis,\footnote{The resulting observations are enough to deduce that $\ell=1$, but this is not necessary for the proceeding argument.}
\begin{align*}
    & \sigma_i^{-1}(\{\kappa_a^\ell, \kappa_b^\ell\}) \cap (V(\mathcal P_a^\ell) \setminus \{v_a^\ell\}) \neq \emptyset,
    & \sigma_i^{-1}(\{\kappa_a^\ell, \kappa_b^\ell\}) \cap (V(\mathcal P_b^\ell) \setminus \{v_b^\ell\}) \neq \emptyset.
\end{align*}
If $\sigma_i(v_a^\ell) = \sigma_{i+1}(v_a^\ell)$ and $\sigma_i(v_b^\ell) = \sigma_{i+1}(v_b^\ell)$, then $\sigma_{i+1}$ satisfies Lemma \ref{lem:two_statements}(2). If $\sigma_i(v_a^\ell) \neq \sigma_{i+1}(v_a^\ell)$ (the setting $\sigma_i(v_b^\ell) \neq \sigma_{i+1}(v_b^\ell)$ is argued analogously), we must have that (exactly) one of
\begin{align*}
    & \sigma_{i+1}^{-1}(\sigma_i(v_a^\ell)) \in N_{X_L}(v_a^\ell) \cap V(\mathcal P_a^\ell),
    & \sigma_{i+1}^{-1}(\sigma_i(v_a^\ell)) \in N_{X_L}(v_a^\ell) \setminus V(\mathcal P_a^\ell) \text{  and  } \sigma_{i+1}(v_a^\ell) = \mu
\end{align*}
must hold, since $\sigma_{i+1}$ would otherwise violate Proposition \ref{prop:layer_independence}(4), due to
\begin{align*}
    \{\mu, \sigma_i(v_a^\ell)\} \subset \sigma_{i+1}^{-1}(V(\mathcal K^\ell)) \setminus (V(\mathcal P_a^\ell) \cup V(\mathcal P_b^\ell)).
\end{align*}
As before, $\sigma_{i+1}$ satisfies Lemma \ref{lem:two_statements}(2) in both situations. 

\smallskip

\paragraph{\textbf{Subcase 2.2: $\{\sigma_i(v_a^\ell),\sigma_i(v_b^\ell)\} \cap V(\mathcal K^\ell) = \{\sigma_i(v_a^\ell)\}$.}} The setting $\{\sigma_i(v_a^\ell),\sigma_i(v_b^\ell)\} \cap V(\mathcal K^\ell) = \{\sigma_i(v_b^\ell)\}$ is argued analogously. The induction hypothesis yields 
\begin{align*}
    \sigma_i^{-1}(\{\kappa_a^\ell, \kappa_b^\ell\}) \cap (V(\mathcal P_a^\ell) \setminus \{v_a^\ell\}) \neq \emptyset.
\end{align*}
From \eqref{eq:two_statements_lemma_eq_4}, we deduce that 
\begin{align} \label{eq:two_statements_lemma_eq_5}
    |\sigma_i^{-1}(\{\kappa_a^\ell, \kappa_b^\ell\}) \cap (V(\mathcal P_a^\ell) \setminus \{v_a^\ell\})| = 1
\end{align}
and also that
\begin{align*}
    (V(\mathcal P_a^\ell) \cup V(\mathcal P_b^\ell)) \setminus \sigma_i^{-1}(V(\mathcal K^\ell)) = \{v_b^\ell\} \cup (\sigma_i^{-1}(\{\kappa_a^\ell, \kappa_b^\ell\}) \cap (V(\mathcal P_a^\ell) \setminus \{v_a^\ell\})).
\end{align*}
We can argue as in Subcase 2.1 to deduce that $\sigma_{i+1}$ satisfies Lemma \ref{lem:two_statements}(2) if $\sigma_i(v_a^\ell) = \sigma_{i+1}(v_a^\ell)$ and $\sigma_i(v_b^\ell) = \sigma_{i+1}(v_b^\ell)$, or if $\sigma_i(v_a^\ell) \neq \sigma_{i+1}(v_a^\ell)$. If $\sigma_i(v_b^\ell) \neq \sigma_{i+1}(v_b^\ell)$, 
\begin{align*}
    \sigma_i^{-1}(\mu) = \sigma_{i+1}^{-1}(\mu) \notin V(\mathcal P_a^\ell) \cup V(\mathcal P_b^\ell),
\end{align*}
so $\sigma_{i+1}$ satisfies Lemma \ref{lem:two_statements}(2) if $\sigma_{i+1}(v_b^\ell) \notin V(\mathcal K^\ell)$. Thus, assume $\sigma_{i+1}(v_b^\ell) \in V(\mathcal K^\ell)$. Studying the neighborhoods of vertices in $V(\mathcal K^\ell)$ yields that $\sigma_i(v_b^\ell) \in \{\kappa_a^\ell, \kappa_b^\ell\}$, as $\sigma_i$ would violate Proposition \ref{prop:rule_of_two} or \ref{prop:layer_independence}(2,3) otherwise. If 
\begin{align*}
    \sigma_{i+1}^{-1}(\sigma_i(v_b^\ell)) \in N_{X_L}(v_b^\ell) \cap V(\mathcal P_b^\ell),
\end{align*}
$\sigma_{i+1}$ satisfies Lemma \ref{lem:two_statements}(2). If 
\begin{align*}
    \sigma_{i+1}^{-1}(\sigma_i(v_b^\ell)) \in N_{X_L}(v_b^\ell) \setminus V(\mathcal P_b^\ell),
\end{align*}
it must be that $\sigma_{i+1}(v_b^\ell) = \mu$ (recall that $\mu \in V(\mathcal K^\ell)$ is the unique such vertex for which $\sigma_i^{-1}(\mu) \notin V(\mathcal P_a^\ell) \cup V(\mathcal P_b^\ell)$). By \eqref{eq:two_statements_lemma_eq_5}, alongside $\sigma_i(v_b^\ell) \in \{\kappa_a^\ell, \kappa_b^\ell\}$ and $\sigma_{i+1}^{-1}(\sigma_i(v_b^\ell)) \in N_{X_L}(v_b^\ell) \cap V(\mathcal P_b^\ell)$, $\sigma_{i+1}$ satisfies Lemma \ref{lem:two_statements}(1).

\smallskip

\paragraph{\textbf{Subcase 2.3: $\sigma_i(\{v_a^\ell, v_b^\ell\}) \cap V(\mathcal K^\ell) = \emptyset$.}} From \eqref{eq:two_statements_lemma_eq_4}, we have that
\begin{align*}
    \sigma_i^{-1}(V(\mathcal K^\ell) \setminus \{\mu\}) = (V(\mathcal P_a^\ell) \cup V(\mathcal P_b^\ell)) \setminus \{v_a^\ell, v_b^\ell\}
\end{align*}
since the LHS is a subset of the RHS and their cardinalities are equal. If $\sigma_{i+1}^{-1}(\mu) \in \{v_a^\ell, v_b^\ell\}$, then
\begin{align*}
    \sigma_{i+1}^{-1}(V(\mathcal K^\ell)) \subset V(\mathcal P_a^\ell) \cup V(\mathcal P_b^\ell) \text{ and } \sigma_{i+1}(\{v_a^\ell, v_b^\ell\}) \not\subset V(\mathcal K^\ell),
\end{align*}
so $\sigma_{i+1}$ satisfies Lemma \ref{lem:two_statements}(1). Now assume $\sigma_{i+1}^{-1}(\mu) \notin \{v_a^\ell, v_b^\ell\}$, from which it easily follows that
\begin{align*}
    \sigma_{i+1}^{-1}(V(\mathcal K^\ell)) \not\subset V(\mathcal P_a^\ell) \cup V(\mathcal P_b^\ell) \text{ and } |\{\sigma_{i+1}(v_a^\ell), \sigma_{i+1}(v_b^\ell)\} \cap V(\mathcal K^\ell)| \leq 1.
\end{align*}
Of course, $\sigma_{i+1}$ satisfies Lemma \ref{lem:two_statements}(2) if 
\begin{align*}
    \{\sigma_{i+1}(v_a^\ell), \sigma_{i+1}(v_b^\ell)\} \cap V(\mathcal K^\ell) = \emptyset.
\end{align*}
If $\sigma_{i+1}(v_a^\ell) \in V(\mathcal K^\ell)$ (the setting $\sigma_{i+1}(v_b^\ell) \in V(\mathcal K^\ell)$ is argued analogously), then by studying the neighborhoods of vertices in $V(\mathcal K^\ell)$, it must be that $\sigma_i(v_a^\ell) \in \{\kappa_a^\ell, \kappa_b^\ell\}$, since $\sigma_i$ would otherwise violate Proposition \ref{prop:rule_of_two} or \ref{prop:layer_independence}(2,3). Furthermore, 
\begin{align*}
    \sigma_{i+1}^{-1}(\sigma_i(v_a^\ell)) \in N_{X_L}(v_a^\ell) \cap V(\mathcal P_a^\ell),
\end{align*}
since if $\sigma_{i+1}^{-1}(\sigma_i(v_a^\ell)) \in N_{X_L}(v_a^\ell) \setminus V(\mathcal P_a^\ell)$, we would have
\begin{align*}
    \{\sigma_i^{-1}(\mu), \sigma_i^{-1}(\sigma_{i+1}(v_a^\ell))\} \subset \sigma_i^{-1}(V(\mathcal K^\ell)) \setminus (V(\mathcal P_a^\ell) \cup V(\mathcal P_b^\ell)),
\end{align*}
implying $\sigma_i$ violates Proposition \ref{prop:layer_independence}(4); $\mu \neq \sigma_{i+1}(v_a^\ell)$ since $\sigma_{i+1}^{-1}(\mu) \notin \{v_a^\ell, v_b^\ell\}$. It follows quickly that $\sigma_{i+1}$ satisfies Lemma \ref{lem:two_statements}(2). This completes the induction.
\end{proof}

We are now ready to prove Proposition \ref{prop:layer_independence}.

\begin{proof}[Proof of Proposition \ref{prop:layer_independence}]
Assume (towards a contradiction) that the proposition is false, so there exists a swap sequence $\Sigma = \{\sigma_i\}_{i=0}^\lambda$ with $\sigma_0 = \sigma_s$ of minimal length $\lambda$ containing a vertex that violates Proposition \ref{prop:layer_independence}. It is apparent from the preceding comment that $\lambda \geq 1$. We also observe that all terms $\sigma_i \in \Sigma$ satisfy Proposition \ref{prop:rule_of_two}, and that $\sigma_\lambda$ must violate at least one of the four properties of Proposition \ref{prop:layer_independence}. We break into cases based on the property that the configuration $\sigma_\lambda$ violates, and reach a contradiction in every case to deduce that none of these four properties can be broken by $\sigma_\lambda$. This will produce the desired contradiction on our initial assumption.

\subsubsection*{\textbf{Case 1: $\sigma_\lambda^{-1}(\kappa_a^1) \notin V(X_a^1)$ or $\sigma_\lambda^{-1}(\kappa_b^1) \notin V(X_b^1)$.}}

Assume that this statement holds. We only study the setting in which $\sigma_\lambda^{-1}(\kappa_a^1) \notin V(X_a^1)$; raising a contradiction when $\sigma_\lambda^{-1}(\kappa_b^1) \notin V(X_b^1)$ is analogous. To reach $\sigma_\lambda$ from $\sigma_{\lambda-1}$, we must have $\sigma_{\lambda-1}^{-1}(\kappa_a^1) \in \{v^1, v_a^{1,2}\}$ (in particular, we must have $\lambda \geq 2$, since $\sigma_s^{-1}(\kappa_a^1) \notin \{v^1, v_a^{1,2}\}$). We break into subcases based on the value of $\sigma_{\lambda-1}^{-1}(\kappa_a^1)$.

\smallskip

\paragraph{\textbf{Subcase 1.1: $\sigma_{\lambda-1}^{-1}(\kappa_a^1) = v^1$.}} Here, $\sigma_\lambda^{-1}(\kappa_a^1) \in N_{X_L}(v^1) \cap V(\mathcal C_b^1)$. Recall that 
\begin{align*}
    N_{Y_L}(\kappa_a^1) = (V(\mathcal S_a^1) \setminus \{\kappa_a^1\}) \cup V(\mathcal K^1).
\end{align*}
Since $\sigma_{\lambda-1}$ satisfies Proposition \ref{prop:rule_of_two} (on $\mathcal C_a^1$), the vertex $\sigma_{\lambda-1}(\sigma_\lambda^{-1}(\kappa_a^1))$ that $\kappa_a^1$ swaps with to reach $\sigma_\lambda$ from $\sigma_{\lambda-1}$ lies in $V(\mathcal K^1)$. Since $\sigma_{\lambda-1}$ satisfies Proposition \ref{prop:rule_of_two} (on $\mathcal C_b^1$), it must be that
\begin{align*}
    \{\kappa_a^1, \sigma_{\lambda-1}(\sigma_\lambda^{-1}(\kappa_a^1))\} = \sigma_{\lambda-1}(V(\mathcal C_b^1)) \setminus (V(\mathcal S_b^1) \setminus \{\kappa_b^1\}).
\end{align*}
Combining this with $\sigma_{\lambda-1}^{-1}(\kappa_b^1) \in V(X_b^1)$ (due to $\sigma_{\lambda-1}$ satisfying Proposition \ref{prop:layer_independence}(1)), we deduce that
\begin{align*}
    \sigma_{\lambda-1}^{-1}(\kappa_b^1) \in V(\mathcal P_b^1) \setminus \{v_b^1\}.
\end{align*}
However, by applying Propositions \ref{prop:rule_of_two} and \ref{prop:layer_independence}(1-3) to $\sigma_{\lambda-1}$, and recalling our assumption that $\sigma_{\lambda-1}^{-1}(\kappa_a^1) = v^1$, we deduce that
\begin{align*}
    \sigma_{\lambda-1}(\mathcal P_b^1) \setminus \{\kappa_b^1\} \subset (V(\mathcal S_b^1) \setminus \{\kappa_b^1\}) \cup V(\mathcal K^1) = N_{Y_L}(\kappa_b^1),
\end{align*}
so from $\sigma_{\lambda-1}$, we can swap $\kappa_b^1$ along $V(\mathcal P_b^1)$ onto $v_b^1$, yielding a configuration $\tau \in V(\mathscr{C})$ satisfying
\begin{align*}
    |\tau(V(\mathcal C_b^1)) \setminus (V(\mathcal S_b^1) \setminus \{\kappa_b^1\})| \geq 3,
\end{align*}
contradicting Proposition \ref{prop:rule_of_two}. See Figure \ref{fig:layer_independence_1.1} for an illustration. In particular, this argument (with the analogue for the setting where $\sigma_\lambda^{-1}(\kappa_b^1) \notin V(X_b^1)$) concludes the study of the first three cases for $L=1$.

\begin{figure}[ht]
\begin{center}

\includegraphics[width=0.8\textwidth]{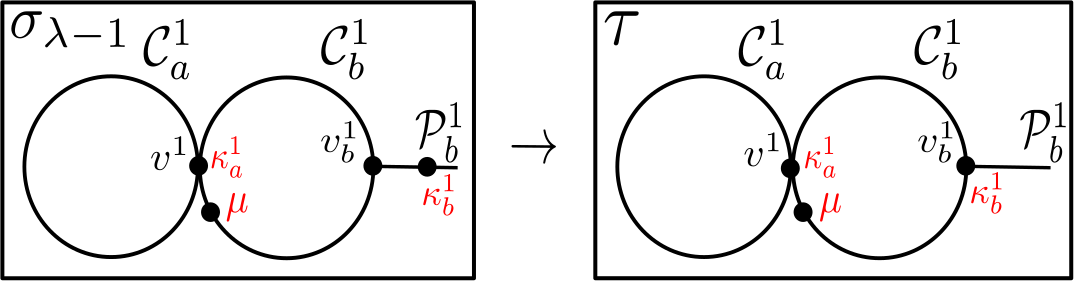} 

\caption{Configurations in $\Sigma$ used to raise a contradiction for Subcase 1.1, where we let $\mu = \sigma_{\lambda-1}(\sigma_\lambda^{-1}(\kappa_a^1))$. From $\sigma_{\lambda-1}$, swapping $\kappa_b^1$ left onto $v_b^1$ yields a configuration $\tau$ which violates Proposition \ref{prop:rule_of_two} on $\mathcal C_b^1$.}
\label{fig:layer_independence_1.1}
\end{center}
\end{figure}
\paragraph{\textbf{Subcase 1.2: $\sigma_{\lambda-1}^{-1}(\kappa_a^1) = v_a^{1,2}$.}} This subcase only applies for $L \geq 2$. Observing that we must have 
\begin{align*}
    \sigma_\lambda^{-1}(\kappa_a^1) \in N_{X_L}(v_a^{1,2}) \cap V(\mathcal C_a^2),
\end{align*}
studying $N_{Y_L}(\kappa_a^1)$ yields $\sigma_\lambda(v_a^{1,2}) \in \{\kappa_a^2, \kappa_b^2\}$, since 
\begin{align*}
    \sigma_\lambda(v_a^{1,2}) \in V(\mathcal S_a^1) \setminus \{\kappa_a^1\} \text{  and 
 }\sigma_\lambda(v_a^{1,2}) \in V(\mathcal K^1) \setminus \{\kappa_a^2, \kappa_b^2\}
\end{align*}
imply $\sigma_{\lambda-1}$ violates Proposition \ref{prop:rule_of_two} and Proposition \ref{prop:layer_independence}(3), respectively. Since $\sigma_j$ satisfies Propositions \ref{prop:rule_of_two} and \ref{prop:layer_independence}(2,3) for all $0 \leq j \leq \lambda-1$, a case check on the types of vertices in $V(Y_L)$ and considering which of them can be in $\sigma_j(V(X^1) \setminus \{v_a^{1,2}, v_b^{1,2}\})$ implies
\begin{align} \label{eq:layer_ind_case_1_eq_1}
    \sigma_j(V(X^1) \setminus \{v_a^{1,2}, v_b^{1,2}\}) \subset \sigma_s(V(X^1)) \text{ for all } 0 \leq j \leq \lambda-1.
\end{align}
The observations $|\sigma_{\lambda-1}(V(X^1))| = |\sigma_s(V(X^1))|$ and $\sigma_s^{-1}(\{\kappa_a^2, \kappa_b^2\}) \subset V(X^1)$ together imply that, since $\kappa_a^1$ swaps with either $\kappa_a^2$ or $\kappa_b^2$ into $N_{X_L}(v_a^{1,2}) \cap V(\mathcal C_a^2)$ to reach $\sigma_\lambda$ from $\sigma_{\lambda-1}$,
\begin{align} \label{eq:layer_ind_case_1_eq_2}
    \sigma_{\lambda-1}(V(X^1)) \setminus \sigma_s(V(X^1)) \neq \emptyset,
\end{align}
while \eqref{eq:layer_ind_case_1_eq_1} applied to $j = \lambda-1$ and $\sigma_{\lambda-1}(v_a^{1,2}) = \kappa_a^1 \in \sigma_s(V(X^1))$ together imply that
\begin{align} \label{eq:layer_ind_case_1_eq_3}
    |\sigma_{\lambda-1}(V(X^1)) \setminus \sigma_s(V(X^1))| = |\sigma_s(V(X^1)) \setminus \sigma_{\lambda-1}(V(X^1))|\leq 1.
\end{align}
If it were true that $\sigma_{\lambda-1}(v_b^{1,2}) \in \sigma_s(V(X^1))$, recalling that $\sigma_{\lambda-1}(v_a^{1,2}) = \kappa_a^1 \in \sigma_s(V(X^1))$, we get
\begin{align*}
    \sigma_{\lambda-1}(\{v_a^{1,2}, v_b^{1,2}\}) \subset \sigma_s(V(X^1)) \implies \sigma_{\lambda-1}(V(X^1) \setminus \{v_a^{1,2}, v_b^{1,2}\}) \not\subset \sigma_s(V(X^1)),
\end{align*}
with the implication due to \eqref{eq:layer_ind_case_1_eq_2}, contradicting \eqref{eq:layer_ind_case_1_eq_1} on $j = \lambda-1$. Therefore, 
\begin{align*}
    \sigma_{\lambda-1}(v_b^{1,2}) \notin \sigma_s(V(X^1)).
\end{align*}
This result, alongside a case check on the possible values of $\sigma_{\lambda-1}(v_b^{1,2})$ (applying Propositions \ref{prop:rule_of_two} and \ref{prop:layer_independence}(2,3) to $\sigma_{\lambda-1}$), gives
\begin{align} \label{eq:layer_ind_case_1_eq_4}
    \sigma_{\lambda-1}(v_b^{1,2}) \in (V(\mathcal S_b^2) \setminus \{\kappa_b^2\}) \cup V(\mathcal K^2).
\end{align}
Let $\sigma_\xi$ be the final term of $\Sigma$ before $\sigma_{\lambda-1}$ satisfying 
\begin{align*}
    \sigma_\xi(v_b^{1,2}) \neq \sigma_{\lambda-1}(v_b^{1,2});
\end{align*}
$\xi < \lambda-1$ is well-defined since $\sigma_{\lambda-1}(v_b^{1,2}) \notin \sigma_s(V(X^1))$. To reach $\sigma_{\xi+1}$ from $\sigma_\xi$, we swap $\sigma_{\lambda-1}(v_b^{1,2})$ with $\sigma_\xi(v_b^{1,2})$, where 
\begin{align*}
    \sigma_\xi^{-1}(\sigma_{\lambda-1}(v_b^{1,2})) \in N_{X_L}(v_b^{1,2}) \cap V(\mathcal C_b^2).
\end{align*}
Indeed, see \eqref{eq:layer_ind_case_1_eq_4}; if we had that
\begin{align*}
    \sigma_\xi^{-1}(\sigma_{\lambda-1}(v_b^{1,2})) \in N_{X_L}(v_b^{1,2}) \cap V(\mathcal C_b^1),
\end{align*}
$\sigma_\xi$ would violate Proposition \ref{prop:rule_of_two} on $\mathcal C_b^2$ if $\sigma_{\lambda-1}(v_b^{1,2}) \in V(\mathcal S_b^2) \setminus \{\kappa_b^2\}$ and Proposition \ref{prop:layer_independence}(2,3) if $\sigma_{\lambda-1}(v_b^{1,2}) \in V(\mathcal K^2)$. By the definition of $\xi$, $\sigma_j(v_b^{1,2})$ remains unchanged for $\xi+1 \leq j \leq \lambda-1$. Furthermore, from \eqref{eq:layer_ind_case_1_eq_4}, we observe that
\begin{align} \label{eq:layer_ind_case_1_eq_5}
    \sigma_\xi(v_b^{1,2}) \in \{\kappa_a^2, \kappa_b^2\} \cup V(\mathcal K^2), 
\end{align}
since the statements
\begin{align*}
    \sigma_\xi(v_b^{1,2}) \in (V(\mathcal S_a^3) \setminus \{\kappa_a^3\}) \cup (V(\mathcal S_b^3) \setminus \{\kappa_b^3\}) \text {  and  }\sigma_\xi(v_b^{1,2}) \in V(\mathcal K^3)
\end{align*}
would result in $\sigma_\xi$ violating Propositions \ref{prop:rule_of_two} and \ref{prop:layer_independence}(2,3), respectively. If $\sigma_\xi(v_b^{1,2}) \in \{\kappa_a^2, \kappa_b^2\}$, then 
\begin{align*}
    \sigma_\lambda^{-1}(\sigma_\xi(v_b^{1,2})) \in V(X^1);
\end{align*}
this is immediate if $\sigma_\xi(v_b^{1,2}) = \sigma_\lambda(v_a^{1,2})$ (recall that $\sigma_\lambda(v_a^{1,2}) \in \{\kappa_a^2, \kappa_b^2\}$), and if $\sigma_\xi(v_b^{1,2}) \neq \sigma_\lambda(v_a^{1,2})$, the assumption $\sigma_{\lambda-1}^{-1}(\sigma_\xi(v_b^{1,2})) = \sigma_\lambda^{-1}(\sigma_\xi(v_b^{1,2})) \notin V(X^1)$ (we swap $\sigma_\lambda(v_a^{1,2})$ and $\kappa_a^1$ to reach $\sigma_\lambda$ from $\sigma_{\lambda-1}$), alongside $\sigma_{\lambda-1}^{-1}(\sigma_\lambda(v_a^{1,2})) \in N_{X_L}(v_a^{1,2}) \cap V(\mathcal C_a^2)$, would contradict \eqref{eq:layer_ind_case_1_eq_3}, since we would have
\begin{align*}
     \{\sigma_\xi(v_b^{1,2}), \sigma_\lambda(v_a^{1,2})\} \subseteq \sigma_s(V(X^1)) \setminus \sigma_{\lambda-1}(V(X^1)).
\end{align*}
Thus, $\sigma_\xi(v_b^{1,2})$ traverses a path from $v_b^{1,2}$ to $v_a^{1,2}$, not involving $v_b^{1,2}$ past $\sigma_\xi$, as we go from $\sigma_\xi$ to $\sigma_\lambda$. Certainly, this traversal swaps $\sigma_\xi(v_b^{1,2})$ along both $V(\mathcal C_a^2)$ and $V(\mathcal C_b^2)$. Suppose $\sigma_\xi(v_b^{1,2}) = \kappa_a^2$. Due to \eqref{eq:layer_ind_case_1_eq_4}, $\sigma_\xi(v_b^{1,2}) = \kappa_a^2$ must have swapped with $\sigma_{\lambda-1}(v_b^{1,2}) \in V(\mathcal K^2)$ to reach $\sigma_{\xi+1}$ from $\sigma_\xi$. Let $\zeta > \xi+1$ be the earliest such index satisfying 
\begin{align*}
    \sigma_\zeta(\sigma_{\xi+1}^{-1}(\kappa_a^2)) \neq \kappa_a^2;
\end{align*}
$\zeta$ is well-defined since $\sigma_\xi(v_b^{1,2}) = \kappa_a^2$ swaps along both $V(\mathcal C_a^2)$ and $V(\mathcal C_b^2)$ to reach $\sigma_\lambda$. The vertex $\sigma_\zeta(\sigma_{\xi+1}^{-1}(\kappa_a^2))$ must have swapped with $\kappa_a^2$ to reach $\sigma_\zeta$ from $\sigma_{\zeta-1}$. Since $\sigma_\zeta(v_b^{1,2}) = \sigma_{\lambda-1}(v_b^{1,2}) \in V(\mathcal K^2)$, $\sigma_\zeta(\sigma_{\xi+1}^{-1}(\kappa_a^2)) \in N_{Y_L}(\kappa_a^2)$, and $\kappa_a^2$ are all not in $V(\mathcal S_b^2) \setminus \{\kappa_b^2\}$, we have
\begin{align*}
    |\sigma_\zeta(V(\mathcal C_b^2)) \setminus (V(\mathcal S_b^2) \setminus \{\kappa_b^2\})| \geq 3,
\end{align*}
contradicting Proposition \ref{prop:rule_of_two}. See Figure \ref{fig:layer_independence_1.2a} for an illustration. Thus, $\sigma_\xi(v_b^{1,2}) = \kappa_b^2$. Let $\zeta > \xi+1$ be the earliest such index satisfying 
\begin{align*}
    \sigma_\zeta^{-1}(\kappa_b^2) \in V(\mathcal C_a^2) \setminus N_{X_L}[\bd(\mathcal C_a^2)];
\end{align*}
$\zeta$ is well-defined since $\kappa_b^2$ goes from $v_b^{1,2}$ to $v_a^{1,2}$ to reach $\sigma_\lambda$. Here, $\kappa_b^2$ must have swapped with $\kappa_a^2$ to reach $\sigma_\zeta$ from $\sigma_{\zeta-1}$: as in the preceding case, $\kappa_b^2$ would be ``stuck" otherwise, due to $\sigma_\zeta$ satisfying Proposition \ref{prop:rule_of_two} (on $\mathcal C_a^2$). But then $\sigma_\zeta$ would violate Proposition \ref{prop:layer_independence}(4) on $\ell=1$, namely since $\{\kappa_a^2, \kappa_b^2\} \subset V(\mathcal K^1)$, which implies
\begin{align*}
    |\sigma_\zeta^{-1}(V(\mathcal K^1)) \setminus (V(\mathcal P_a^1) \cup V(\mathcal P_b^1))| \geq 2.
\end{align*}
See Figure \ref{fig:layer_independence_1.2b} for an illustration. So, by \eqref{eq:layer_ind_case_1_eq_5}, we must have $\sigma_\xi(v_b^{1,2}) \in V(\mathcal K^2)$. Since we swap $\sigma_{\lambda-1}(v_b^{1,2})$ with $\sigma_\xi(v_b^{1,2})$ to reach $\sigma_{\xi+1}$ from $\sigma_\xi$, it follows from \eqref{eq:layer_ind_case_1_eq_4} that $\sigma_{\lambda-1}(v_b^{1,2}) \in V(\mathcal K^2)$, since there is no element of $V(\mathcal K^2)$ (in particular, $\sigma_\xi(v_b^{1,2})$) that can swap with an element of $V(\mathcal S_b^2) \setminus \{\kappa_b^2\}$. But then $\sigma_\xi$ violates Proposition \ref{prop:layer_independence}(4) (on $\ell = 2$), which is our final contradiction in this case. We conclude that Proposition \ref{prop:layer_independence}(1) cannot have been the property violated by $\sigma_\lambda$.

\begin{figure}[ht]
  \centering
  \subfloat[Assuming $\sigma_\xi(v_b^{1,2}) = \kappa_a^2$, with $\mu' = \sigma_\zeta(\sigma_{\xi+1}^{-1}(\kappa_a^2))$. Here, $\sigma_\zeta$ violates Proposition \ref{prop:rule_of_two} on $\mathcal C_b^2$ due to $\kappa_a^2$, $\mu$, and $\mu'$, none of which are in $V(\mathcal S_b^2) \setminus \{\kappa_b^2\}$.]{\includegraphics[width=0.45\textwidth]{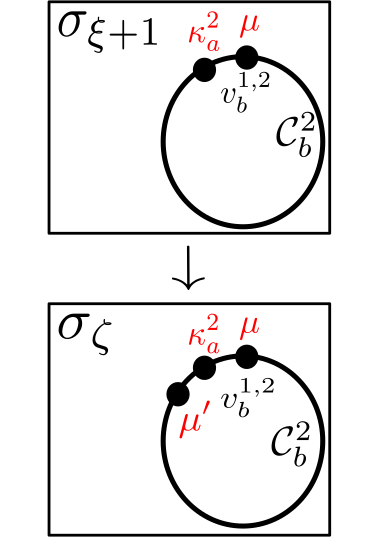}\label{fig:layer_independence_1.2a}}
  \hfil
  \subfloat[Assuming $\sigma_\xi(v_b^{1,2}) = \kappa_b^2$. Here, $\sigma_\zeta$ violates Proposition \ref{prop:layer_independence}(4) on $\ell=1$ due to $\kappa_a^2$, $\kappa_b^2$.]{\includegraphics[width=0.44\textwidth]{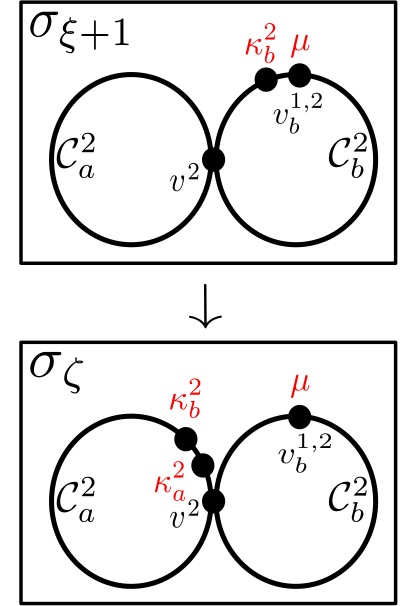}\label{fig:layer_independence_1.2b}}
  \caption{Configurations in $\Sigma$ used to raise a contradiction for Subcase 1.2 when we assume that $\sigma_\xi(v_b^{1,2}) \in \{\kappa_a^2, \kappa_b^2\}$. We let $\mu = \sigma_{\lambda-1}(v_b^{1,2})$.}
  \label{fig:layer_independence_1.2}
\end{figure}

\subsubsection*{\textbf{Case 2: For some $\ell \geq 2$, we have $\sigma_\lambda^{-1}(\kappa_a^\ell) \notin V(X^{\ell-1}) \cup V(X^\ell)$ or $\sigma_\lambda^{-1}(\kappa_b^\ell) \notin V(X^{\ell-1}) \cup V(X^\ell)$.}}

This case is relevant only for $L \geq 2$. Assume this statement holds for some $2 \leq \ell \leq L$. We only study the setting in which $\sigma_\lambda^{-1}(\kappa_a^\ell) \notin V(X^{\ell-1}) \cup V(X^\ell)$. Raising a contradiction when $\sigma_\lambda^{-1}(\kappa_b^\ell) \notin V(X^{\ell-1}) \cup V(X^\ell)$ is entirely analogous. Notice that 
\begin{align*}
    \sigma_{\lambda-1}^{-1}(\kappa_a^\ell) \in \{v_a^{\ell-2, \ell-1}, v_b^{\ell-2, \ell-1}, v_a^{\ell, \ell+1}, v_b^{\ell, \ell+1}\}
\end{align*}
(precisely, the RHS above is the subset of these vertices defined for $\ell$). To reach $\sigma_\lambda$ from $\sigma_{\lambda-1}$, the vertex $\sigma_{\lambda-1}(\sigma_\lambda^{-1}(\kappa_a^\ell))$ that $\kappa_a^\ell$ swaps with satisfies
\begin{align} \label{eq:layer_ind_case_2_eq_1}
    \sigma_{\lambda-1}(\sigma_\lambda^{-1}(\kappa_a^\ell)) \in \{\kappa_a^{\ell-1}, \kappa_b^{\ell-1}\} \cup V(\mathcal K^\ell) \cup (V(\mathcal S_a^\ell) \setminus \{\kappa_a^\ell\}),
\end{align}
as $\sigma_{\lambda-1}(\sigma_\lambda^{-1}(\kappa_a^\ell)) \in V(\mathcal K_b^{\ell-1})$ would cause $\sigma_{\lambda-1}$ to violate Proposition \ref{prop:layer_independence}(4), since we then get
\begin{align*}
     \{\kappa_a^\ell, \sigma_{\lambda-1}(\sigma_\lambda^{-1}(\kappa_a^\ell))\} \subset \sigma_{\lambda-1}^{-1}(V(\mathcal K^{\ell-1})) \setminus (V(\mathcal P_a^{\ell-1}) \cup V(\mathcal P_b^{\ell-1}))
\end{align*}
which implies that
\begin{align*}
    |\sigma_{\lambda-1}^{-1}(V(\mathcal K^{\ell-1})) \setminus (V(\mathcal P_a^{\ell-1}) \cup V(\mathcal P_b^{\ell-1}))| \geq 2.
\end{align*}
We break into subcases based on the value of $\sigma_{\lambda-1}^{-1}(\kappa_a^\ell)$.

\smallskip

\paragraph{\textbf{Subcase 2.1: $\sigma_{\lambda-1}^{-1}(\kappa_a^\ell) \in \{v_a^{\ell-2, \ell-1}, v_b^{\ell-2, \ell-1}\}$.}} This subcase applies for $\ell \geq 3$. The vertex which $\kappa_a^\ell$ swaps onto satisfies
\begin{align*}
    \sigma_\lambda^{-1}(\kappa_a^\ell) \in (N_{X_L}(v_a^{\ell-2, \ell-1}) \cup N_{X_L}(v_b^{\ell-2, \ell-1})) \cap V(X^{\ell-2}).
\end{align*}
From \eqref{eq:layer_ind_case_2_eq_1}, we deduce that the vertex $\kappa_a^\ell$ swaps with to reach $\sigma_\lambda$ from $\sigma_{\lambda-1}$ satisfies
\begin{align*}
    \sigma_{\lambda-1}(\sigma_\lambda^{-1}(\kappa_a^\ell)) \in \{\kappa_a^{\ell-1}, \kappa_b^{\ell-1}\},
\end{align*}
since the statements
\begin{align*}
    \sigma_{\lambda-1}(\sigma_\lambda^{-1}(\kappa_a^\ell)) \in V(\mathcal S_a^\ell) \setminus \{\kappa_a^\ell\} \text{  and  } \sigma_{\lambda-1}(\sigma_\lambda^{-1}(\kappa_a^\ell)) \in V(\mathcal K^\ell)
\end{align*}
respectively imply that $\sigma_{\lambda-1}$ violates Proposition \ref{prop:rule_of_two} and Proposition \ref{prop:layer_independence}(2,3). Proceeding backwards in $\Sigma$, $\sigma_{\lambda-2} \neq \sigma_\lambda$ ($\sigma_{\lambda-1}^{-1}(\kappa_a^\ell) \neq \sigma_s^{-1}(\kappa_a^\ell)$ implies $\sigma_{\lambda-1} \neq \sigma_s$, so $\sigma_{\lambda-2}$ is well-defined, and $\lambda$ is minimal). Now, if we had that
\begin{align*}
    \sigma_{\lambda-2}^{-1}\left(\sigma_{\lambda-1}(\sigma_\lambda^{-1}(\kappa_a^\ell))\right) = \sigma_\lambda^{-1}(\kappa_a^\ell) \text{ and } \sigma_{\lambda-2}^{-1}(\kappa_a^\ell) = \sigma_{\lambda-1}^{-1}(\kappa_a^\ell),
\end{align*}
then swapping $\sigma_{\lambda-1}(\sigma_\lambda^{-1}(\kappa_a^\ell))$ with $\kappa_a^\ell$ directly from $\sigma_{\lambda-2}$ would contradict $\lambda$ being minimal. Thus, we have
\begin{align*}
    \sigma_{\lambda-2}^{-1}(\kappa_a^\ell) \in N_{X_L}(\sigma_{\lambda-1}^{-1}(\kappa_a^\ell)) \cap V(X^{\ell-1}),
\end{align*}
since $\sigma_{\lambda-2}$ satisfies Proposition \ref{prop:rule_of_two} and neither $\sigma_{\lambda-1}(\sigma_\lambda^{-1}(\kappa_a^\ell))$ nor $\kappa_a^\ell$ can swap with vertices in
\begin{align*}
    (V(\mathcal S_a^{\ell-2}) \setminus \{\kappa_a^{\ell-2}\}) \cup (V(\mathcal S_b^{\ell-2}) \setminus \{\kappa_b^{\ell-2}\}).
\end{align*} 
But any vertex in $N_{Y_L}(\kappa_a^\ell)$ with which $\kappa_a^\ell$ can swap to reach $\sigma_{\lambda-2}$ from $\sigma_{\lambda-1}$ raises a contradiction: a vertex of $\{\kappa_a^{\ell-1}, \kappa_b^{\ell-1}\} \cup V(\mathcal K^{\ell-1})$ implies $\sigma_{\lambda-1}$ violates Proposition \ref{prop:layer_independence}(4) (respectively, on layers $\ell-2$ and $\ell-1$, due to $\sigma_{\lambda-1}(\sigma_\lambda^{-1}(\kappa_a^\ell))$ and $\kappa_a^\ell$), a vertex of $V(\mathcal K^\ell)$ implies $\sigma_{\lambda-1}$ violates Proposition \ref{prop:layer_independence}(2,3), and a vertex in $V(\mathcal S_a^\ell) \setminus \{\kappa_a^\ell\}$ implies $\sigma_{\lambda-1}$ violates Proposition \ref{prop:rule_of_two}. See Figure \ref{fig:layer_independence_2.1} for an illustration.

\begin{figure}[ht]
    \centering
    \includegraphics[width=0.5\textwidth]{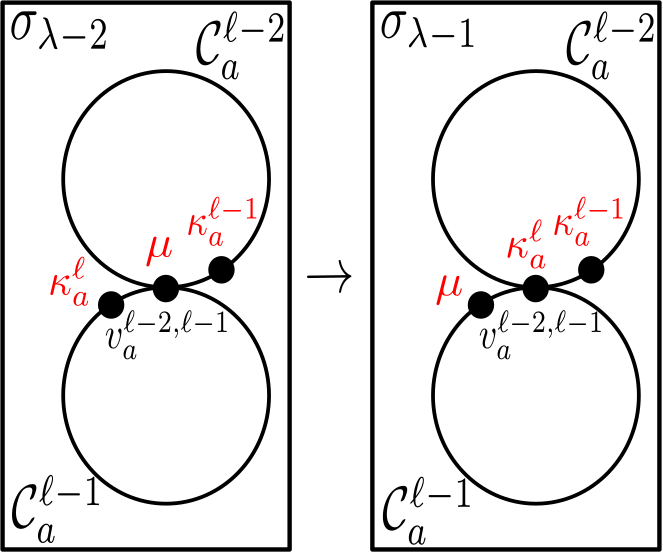}
    \caption{Configurations in $\Sigma$ used to raise a contradiction for Subcase 2.1, illustrated for $\sigma_{\lambda-1}^{-1}(\kappa_a^\ell) = v_a^{\ell-2, \ell-1}$ and $\sigma_{\lambda-1}(\sigma_\lambda^{-1}(\kappa_a^\ell)) = \kappa_a^{\ell-1}$. Here, $\kappa_a^\ell$ must swap with a vertex $\mu \in N_{Y_L}(\kappa_a^\ell)$ to reach $\sigma_{\lambda-2}$ from $\sigma_{\lambda-1}$, for which all possibilities of $\mu$ raise a contradiction.}
    \label{fig:layer_independence_2.1}
\end{figure}

\smallskip

\paragraph{\textbf{Subcase 2.2: $\sigma_{\lambda-1}^{-1}(\kappa_a^\ell) \in \{v_a^{\ell, \ell+1}, v_b^{\ell, \ell+1}\}$.}} This subcase applies for $2 \leq \ell < L$. The vertex $\kappa_a^\ell$ swaps onto satisfies
\begin{align*}
    \sigma_\lambda^{-1}(\kappa_a^\ell) \in (N_{X_L}(v_a^{\ell, \ell+1}) \cup N_{X_L}(v_b^{\ell, \ell+1})) \cap V(X^{\ell+1}).
\end{align*}
From \eqref{eq:layer_ind_case_2_eq_1}, we deduce that the vertex $\kappa_a^\ell$ swaps with to reach $\sigma_\lambda$ from $\sigma_{\lambda-1}$ satisfies
\begin{align} \label{eq:layer_ind_case_2_eq_2}
    \sigma_{\lambda-1}(\sigma_\lambda^{-1}(\kappa_a^\ell)) \in \{\kappa_a^{\ell+1}, \kappa_b^{\ell+1}\},
\end{align}
since the statements
\begin{align*}
    & \sigma_{\lambda-1}(\sigma_\lambda^{-1}(\kappa_a^\ell)) \in V(\mathcal S_a^\ell) \setminus \{\kappa_a^\ell\},
    & \sigma_{\lambda-1}(\sigma_\lambda^{-1}(\kappa_a^\ell)) \in \{\kappa_a^{\ell-1}, \kappa_b^{\ell-1}\} \cup (V(\mathcal K^\ell) \setminus \{\kappa_a^{\ell+1}, \kappa_b^{\ell+1}\})
\end{align*} 
imply that $\sigma_{\lambda-1}$ violates Proposition \ref{prop:rule_of_two} and Proposition \ref{prop:layer_independence}(1-3), respectively.
Let $\sigma_\xi$, with $\xi < \lambda-1$, be the last term in $\Sigma$ before $\sigma_{\lambda-1}$ satisfying 
\begin{align*}
    \sigma_\xi^{-1}\left(\sigma_{\lambda-1}(\sigma_\lambda^{-1}(\kappa_a^\ell))\right) \in V(X^\ell);
\end{align*}
$\xi$ is well-defined since (see \eqref{eq:layer_ind_case_2_eq_2}) $\sigma_s^{-1}(\{\kappa_a^{\ell+1}, \kappa_b^{\ell+1}\}) \subset V(X^\ell)$, which also implies $\sigma_s \neq \sigma_{\lambda-1}$ since 
\begin{align*}
    \sigma_{\lambda-1}^{-1}\left(\sigma_{\lambda-1}(\sigma_\lambda^{-1}(\kappa_a^\ell))\right) = \sigma_\lambda^{-1}(\kappa_a^\ell) \notin V(X^\ell).
\end{align*}
By the definition of $\xi$,
\begin{align} \label{eq:layer_ind_case_2_eq_3}
    \sigma_j^{-1}\left(\sigma_{\lambda-1}(\sigma_\lambda^{-1}(\kappa_a^\ell)) \right) \notin V(X^\ell) \text{ for } \xi+1 \leq j \leq \lambda-1.
\end{align}
Since $\sigma_{\lambda-1}(\sigma_\lambda^{-1}(\kappa_a^\ell))$ traverses a path to $\sigma_{\lambda-1}^{-1}(\kappa_a^\ell)$ as we go from $\sigma_{\xi+1}$ to $\sigma_\lambda$, not involving $V(X^\ell)$ until $\sigma_\lambda$, we further deduce that 
\begin{align*}
    \sigma_\xi^{-1}\left(\sigma_{\lambda-1}(\sigma_\lambda^{-1}(\kappa_a^\ell))\right) = \sigma_{\lambda-1}^{-1}(\kappa_a^\ell),
\end{align*}
since $\sigma_{\lambda-1}(\sigma_\lambda^{-1}(\kappa_a^\ell))$ cannot traverse a path from $\{v_a^{\ell, \ell+1}, v_b^{\ell, \ell+1}\} \setminus \{\sigma_{\lambda-1}^{-1}(\kappa_a^\ell)\}$ to $\sigma_{\lambda-1}^{-1}(\kappa_a^\ell)$ as we swap from $\sigma_{\xi+1}$ to $\sigma_\lambda$ without violating Proposition \ref{prop:rule_of_two} or \ref{prop:layer_independence}(2,4).\footnote{This can be proved using arguments essentially identical to those in Subcase 1.2.} Thus, from $\sigma_\xi$ to $\sigma_{\xi+1}$, $\sigma_{\lambda-1}(\sigma_\lambda^{-1}(\kappa_a^\ell))$ swaps into 
\begin{align*}
    N_{X_L}(\sigma_{\lambda-1}^{-1}(\kappa_a^\ell)) \cap V(X^{\ell+1})
\end{align*}
from $\sigma_{\lambda-1}^{-1}(\kappa_a^\ell)$, and since $\sigma_{\xi+1}$ satisfies Proposition \ref{prop:layer_independence}(2,3), a case check on $N_{Y_L}\left(\sigma_{\lambda-1}(\sigma_\lambda^{-1}(\kappa_a^\ell))\right)$ (see \eqref{eq:layer_ind_case_2_eq_2}) yields
\begin{align*}
    \sigma_{\xi+1}(\sigma_{\lambda-1}^{-1}(\kappa_a^\ell)) \in (V(\mathcal S_a^{\ell+1}) \setminus \{\kappa_a^{\ell+1}\}) \cup (V(\mathcal S_b^{\ell+1}) \setminus \{\kappa_b^{\ell+1}\}) \cup V(\mathcal K^{\ell+1}).
\end{align*}
If it were true that
\begin{align} \label{eq:layer_ind_case_2_eq_4}
    \sigma_{\xi+1}(\sigma_{\lambda-1}^{-1}(\kappa_a^\ell)) \in (V(\mathcal S_a^{\ell+1}) \setminus \{\kappa_a^{\ell+1}\}) \cup (V(\mathcal S_b^{\ell+1}) \setminus \{\kappa_b^{\ell+1}\}),
\end{align}
then it must be that $\sigma_{\lambda-1}(\sigma_\lambda^{-1}(\kappa_a^\ell))$ is the corresponding knob vertex. So from \eqref{eq:layer_ind_case_2_eq_3}, we deduce that $\sigma_j(\sigma_{\lambda-1}^{-1}(\kappa_a^\ell))$ would be fixed for $\xi+1 \leq j \leq \lambda-1$. Taking $j = \xi+1$ and $j = \lambda-1$ would imply 
\begin{align*}
    \sigma_{\xi+1}(\sigma_{\lambda-1}^{-1}(\kappa_a^\ell)) = \kappa_a^\ell,
\end{align*}
contradicting \eqref{eq:layer_ind_case_2_eq_4}. Therefore, $\sigma_{\xi+1}(\sigma_{\lambda-1}^{-1}(\kappa_a^\ell)) \in V(\mathcal K^{\ell+1})$. We now inductively establish that 
\begin{align} \label{eq:layer_ind_case_2_eq_5}
    \sigma_j(\sigma_{\lambda-1}^{-1}(\kappa_a^\ell)) \in V(\mathcal K^{\ell+1}) \text{ for } \xi+1 \leq j \leq \lambda-1
\end{align}
by showing that it is fixed for all such $j$. Assume for some $j$ satisfying $\xi+1 \leq j < \lambda-1$ that $\sigma_j$ satisfies this claim. Then to reach $\sigma_{j+1}$ from $\sigma_j$, $\sigma_j(\sigma_{\lambda-1}^{-1}(\kappa_a^\ell))$ cannot swap with either $\{\kappa_a^{\ell+1}, \kappa_b^{\ell+1}\}$ (see \eqref{eq:layer_ind_case_2_eq_2}; $\sigma_{j+1}$ would violate either \eqref{eq:layer_ind_case_2_eq_3} or Proposition \ref{prop:layer_independence}(4) for layer $\ell$, depending on whether it swaps with $\sigma_{\lambda-1}(\sigma_\lambda^{-1}(\kappa_a^\ell))$ or not, respectively), another vertex in $V(\mathcal K^{\ell+1})$ ($\sigma_{j+1}$ would violate Proposition \ref{prop:layer_independence}(4) for layer $\ell+1$), or a vertex in 
\begin{align*}
    (V(\mathcal S_a^{\ell+2}) \setminus \{\kappa_a^{\ell+2}\}) \cup (V(\mathcal S_b^{\ell+2}) \setminus \{\kappa_b^{\ell+2}\}) \cup V(\mathcal K^{\ell+2})
\end{align*}
(for the setting $\sigma_j(\sigma_{\lambda-1}^{-1}(\kappa_a^\ell)) \in \{\kappa_a^{\ell+2}, \kappa_b^{\ell+2}\}$, if it applies; $\sigma_{j+1}$ would violate Proposition \ref{prop:rule_of_two} or Proposition \ref{prop:layer_independence}(2,3)). This completes the induction. Now, \eqref{eq:layer_ind_case_2_eq_5} on $j = \lambda-1$ raises a contradiction, since $\kappa_a^\ell \notin V(\mathcal K^{\ell+1})$. See Figure \ref{fig:layer_independence_2.2} for an illustration. This is our final contradiction in this case. We conclude that Proposition \ref{prop:layer_independence}(2) cannot have been the property violated by $\sigma_\lambda$.

\begin{figure}[ht]
  {\includegraphics[width=0.8\textwidth]{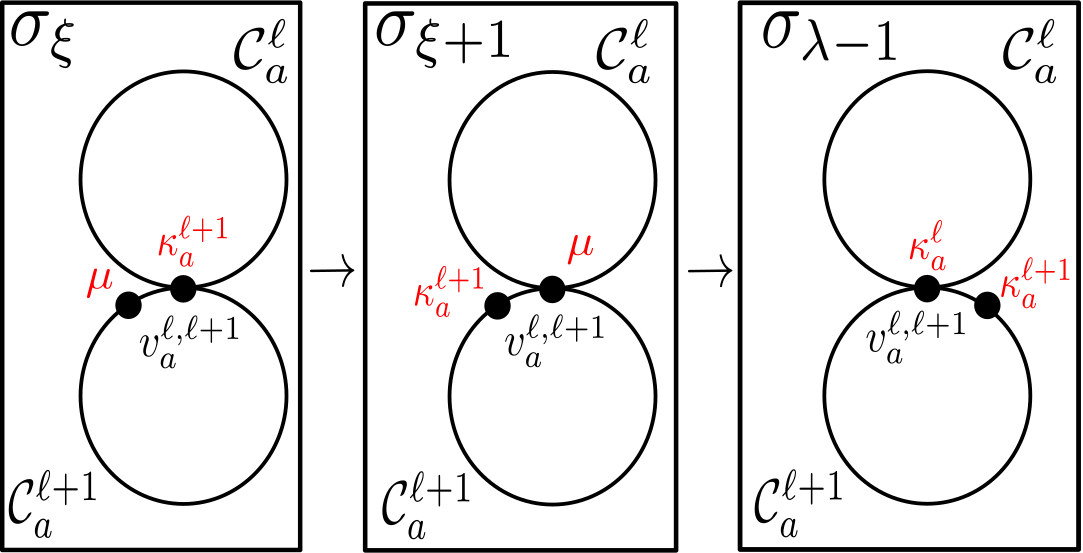}}
  \caption{Configurations in $\Sigma$ used to raise a contradiction for Subcase 2.2, illustrated for $\sigma_{\lambda-1}^{-1}(\kappa_a^\ell) = v_a^{\ell, \ell+1}$ and $\sigma_{\lambda-1}(\sigma_\lambda^{-1}(\kappa_a^\ell)) = \kappa_a^{\ell+1}$. We let $\mu = \sigma_{\xi+1}(v_a^{\ell,\ell+1})$. For $\xi+1 \leq j \leq \lambda-1$, we have $\sigma_j(v_a^{\ell, \ell+1}) \in V(\mathcal K^{\ell+1})$, contradicting $\sigma_{\lambda-1}^{-1}(\kappa_a^\ell) = v_a^{\ell, \ell+1}$.}
  \label{fig:layer_independence_2.2}
\end{figure}

\subsubsection*{\textbf{Case 3: There exists $\ell \in [L-1]$ and $\mu \in V(\mathcal K^\ell) \setminus \{\kappa_a^{\ell+1}, \kappa_b^{\ell+1}\}$ such that $\sigma_\lambda^{-1}(\mu) \notin V(X^\ell)$, or there exists $\mu \in V(\mathcal K^L)$ such that $\sigma_\lambda^{-1}(\mu) \notin V(X^L)$.}}

This case is relevant only for $L \geq 2$. The proceeding argument raises a contradiction both when assuming the existence of $\ell \in [L-1]$ for which there exists $\mu \in V(\mathcal K^\ell) \setminus \{\kappa_a^{\ell+1}, \kappa_b^{\ell+1}\}$ such that $\sigma_\lambda^{-1}(\mu) \notin V(X^\ell)$, and also when assuming the existence of $\mu \in V(\mathcal K^L)$ such that $\sigma_\lambda^{-1}(\mu) \notin V(X^L)$, taking $\ell = L$.

Observe that (where, more precisely, the RHS is the subset that is well-defined for $\ell$)
\begin{align*}
    \sigma_{\lambda-1}^{-1}(\mu) \in \{v_a^{\ell-1, \ell}, v_b^{\ell-1, \ell}, v_a^{\ell, \ell+1}, v_b^{\ell, \ell+1}\},
\end{align*}
and also that the vertex $\sigma_{\lambda-1}(\sigma_\lambda^{-1}(\mu))$ that $\mu$ swaps with to reach $\sigma_\lambda$ from $\sigma_{\lambda-1}$ satisfies
\begin{align} \label{eq:layer_ind_case_3_eq_1}
    \sigma_{\lambda-1}(\sigma_\lambda^{-1}(\mu)) \in \{\kappa_a^{\ell}, \kappa_b^{\ell}\},
\end{align}
since $N_{Y_L}(\mu) \subset V(\mathcal K^\ell) \cup \{\kappa_a^{\ell}, \kappa_b^{\ell}\}$, and $\sigma_{\lambda-1}(\sigma_\lambda^{-1}(\mu)) \in V(\mathcal K^\ell)$ would imply $\sigma_{\lambda-1}$ violates Proposition \ref{prop:layer_independence}(4) on layer $\ell$. If $\sigma_{\lambda-1}^{-1}(\mu) \in \{v_a^{\ell, \ell+1}, v_b^{\ell, \ell+1}\}$ (valid for $\ell < L$), then we would have that
\begin{align*}
    \sigma_\lambda^{-1}(\mu) \in (N_{X_L}(v_a^{\ell, \ell+1}) \cup N_{X_L}(v_b^{\ell, \ell+1})) \cap V(X^{\ell+1}),
\end{align*}
from which \eqref{eq:layer_ind_case_3_eq_1} implies that $\sigma_{\lambda-1}$ violates Proposition \ref{prop:layer_independence}(1) if $\ell = 1$ and Proposition \ref{prop:layer_independence}(2) if $\ell \geq 2$. Thus, it must be that $\ell \geq 2$ and $\sigma_{\lambda-1}^{-1}(\mu) \in \{v_a^{\ell-1, \ell}, v_b^{\ell-1, \ell}\}$, so that
\begin{align} \label{eq:layer_ind_case_3_eq_2}
    \sigma_\lambda^{-1}(\mu) \in (N_{X_L}(v_a^{\ell-1, \ell}) \cup N_{X_L}(v_b^{\ell-1, \ell})) \cap V(X^{\ell-1}).
\end{align}
Proceeding backwards in $\Sigma$, $\sigma_{\lambda-2} \neq \sigma_\lambda$ ($\sigma_{\lambda-1}^{-1}(\mu) \neq \sigma_s^{-1}(\mu)$ implies that $\sigma_{\lambda-1} \neq \sigma_s$, so $\sigma_{\lambda-2}$ is well-defined, and $\lambda$ is minimal). If neither $\mu$ nor $\sigma_{\lambda-1}(\sigma_\lambda^{-1}(\mu))$ were swapped to reach $\sigma_{\lambda-2}$ from $\sigma_{\lambda-1}$, swapping them directly from $\sigma_{\lambda-2}$ would contradict $\lambda$ being minimal. Furthermore, from \eqref{eq:layer_ind_case_3_eq_1} and \eqref{eq:layer_ind_case_3_eq_2}, $\mu$ swaps onto 
\begin{align*}
    N_{X_L}(\sigma_{\lambda-1}^{-1}(\mu)) \cap V(X^\ell)
\end{align*}
to reach $\sigma_{\lambda-2}$ from $\sigma_{\lambda-1}$, since $\sigma_{\lambda-1}$ satisfies Proposition \ref{prop:rule_of_two} and neither $\mu$ nor $\sigma_{\lambda-1}(\sigma_\lambda^{-1}(\kappa_a^\ell))$ can swap with vertices in the set
\begin{align*}
    (V(\mathcal S_a^{\ell-1}) \setminus \{\kappa_a^{\ell-1}\}) \cup (V(\mathcal S_b^{\ell-1}) \setminus \{\kappa_b^{\ell-1}\}).
\end{align*}
But the vertex $\sigma_{\lambda-2}(\sigma_{\lambda-1}^{-1}(\mu))$ that $\mu$ swaps with to reach $\sigma_{\lambda-2}$ from $\sigma_{\lambda-1}$ implies that $\sigma_{\lambda-2}$ violates Proposition \ref{prop:layer_independence}(4) on layer $\ell-1$ if 
\begin{align*}
    \sigma_{\lambda-2}(\sigma_{\lambda-1}^{-1}(\mu)) \in \{\kappa_a^\ell, \kappa_b^\ell\}
\end{align*}
due to $\sigma_{\lambda-1}(\sigma_\lambda^{-1}(\mu))$ and $\sigma_{\lambda-2}(\sigma_{\lambda-1}^{-1}(\mu))$ (see \eqref{eq:layer_ind_case_3_eq_1}) and on layer $\ell$ if 
\begin{align*}
    \sigma_{\lambda-2}(\sigma_{\lambda-1}^{-1}(\mu)) \in V(\mathcal K^\ell)
\end{align*}
due to $\mu$ and $\sigma_{\lambda-2}(\sigma_{\lambda-1}^{-1}(\mu))$. See Figure \ref{fig:layer_independence_3} for an illustration. We conclude that Proposition \ref{prop:layer_independence}(3) cannot have been the property violated by $\sigma_\lambda$.

\begin{figure}[ht]
    \centering
    \includegraphics[width=0.5\textwidth]{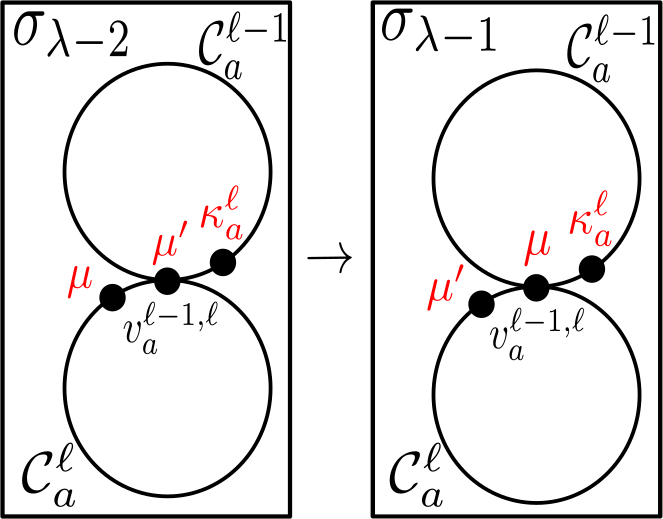}
    \caption{Configurations in $\Sigma$ used to raise a contradiction in Case 3, illustrated for $\sigma_{\lambda-1}(\sigma_\lambda^{-1}(\mu)) = \kappa_a^\ell$ and $\sigma_{\lambda-1}^{-1}(\mu) = v_a^{\ell-1,\ell}$. We let $\mu' = \sigma_{\lambda-2}(\sigma_{\lambda-1}^{-1}(\mu))$. All possibilities of $\mu'$ will cause $\sigma_{\lambda-2}$ to violate Proposition \ref{prop:layer_independence}(4).}
    \label{fig:layer_independence_3}
\end{figure}

\subsubsection*{\textbf{Case 4: There exists $\ell \in [L]$ such that $|\sigma_\lambda^{-1}(V(\mathcal K^\ell)) \setminus (V(\mathcal P_a^\ell) \cup V(\mathcal P_b^\ell))| \geq 2$.}}

Assume that this statement holds for some $\ell \in [L]$. We must have that
\begin{align} \label{eq:layer_ind_case_4_eq}
    |\sigma_{\lambda-1}^{-1}(V(\mathcal K^\ell)) \setminus (V(\mathcal P_a^\ell) \cup V(\mathcal P_b^\ell))| = 1 \text{ and } |\sigma_\lambda^{-1}(V(\mathcal K^\ell)) \setminus (V(\mathcal P_a^\ell) \cup V(\mathcal P_b^\ell))| = 2,
\end{align}
since $|\sigma_\lambda^{-1}(V(\mathcal K^\ell)) \setminus (V(\mathcal P_a^\ell) \cup V(\mathcal P_b^\ell))| \geq 2$, $\lambda$ is minimal, and for any index $1 \leq i \leq \lambda$, we have
\begin{align*}
    |\sigma_i^{-1}(V(\mathcal K^\ell)) \setminus (V(\mathcal P_a^\ell) \cup V(\mathcal P_b^\ell))| - |\sigma_{i-1}^{-1}(V(\mathcal K^\ell)) \setminus (V(\mathcal P_a^\ell) \cup V(\mathcal P_b^\ell))| \leq 1.
\end{align*}
By \eqref{eq:layer_ind_case_4_eq} and Proposition \ref{prop:layer_independence}(4), there is a unique $\mu \in V(\mathcal K^\ell)$ such that 
\begin{align*}
    \sigma_{\lambda-1}^{-1}(\mu) \notin V(\mathcal P_a^\ell) \cup V(\mathcal P_b^\ell).
\end{align*}
Furthermore, there exists $\mu' \in V(\mathcal K^\ell)$ such that (exactly) one of the two following statements hold: 
\begin{align*}
    & \sigma_{\lambda-1}(v_a^\ell) = \mu', \ \sigma_\lambda^{-1}(\mu') \in N_{X_L}(v_a^\ell) \setminus V(\mathcal P_a^\ell);
    & \sigma_{\lambda-1}(v_b^\ell) = \mu', \ \sigma_\lambda^{-1}(\mu') \in N_{X_L}(v_b^\ell) \setminus V(\mathcal P_b^\ell).
\end{align*}
Studying the neighborhoods of vertices in $V(\mathcal K^\ell)$ yields 
\begin{align*}
    \sigma_\lambda(\sigma_{\lambda-1}^{-1}(\mu')) \in \{\kappa_a^\ell, \kappa_b^\ell\};
\end{align*}
it can easily be checked that $\sigma_{\lambda-1}$ would violate one of \eqref{eq:layer_ind_case_4_eq}, Proposition \ref{prop:rule_of_two}, or Proposition \ref{prop:layer_independence}(2,3) otherwise. We will assume (the other three cases are analogous)
\begin{align*}
    \sigma_{\lambda-1}(v_a^\ell) = \mu', \ \sigma_\lambda(\sigma_{\lambda-1}^{-1}(\mu')) = \kappa_a^\ell.
\end{align*}
It follows from Lemma \ref{lem:two_statements}(2) that $\sigma_{\lambda-1}^{-1}(\kappa_b^\ell) \in V(\mathcal P_a^\ell) \setminus \{v_a^\ell\}$. But if $\ell=1$, $\sigma_{\lambda-1}$ violates Proposition \ref{prop:layer_independence}(1) since 
\begin{align*}
    \sigma_{\lambda-1}^{-1}(\{\kappa_a^1, \kappa_b^1\}) \subset V(X_a^1) \setminus \{v^1\};
\end{align*}
if $\ell \geq 2$, $\sigma_{\lambda-1}$ violates Proposition \ref{prop:layer_independence}(4) on layer $\ell-1$ since
\begin{align*}
    \sigma_{\lambda-1}^{-1}(\{\kappa_a^\ell, \kappa_b^\ell\}) \subset \sigma_{\lambda-1}^{-1}(V(\mathcal K^{\ell-1})) \setminus (V(\mathcal P_a^{\ell-1}) \cup V(\mathcal P_b^{\ell-1})).
\end{align*}
See Figure \ref{fig:layer_independence_4} for an illustration.

\begin{figure}[ht]
    \centering
    \includegraphics[width=0.8\textwidth]{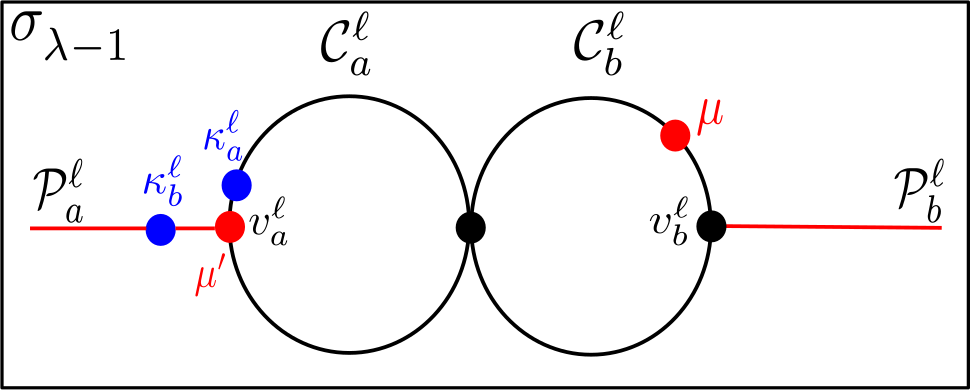}
    \caption{Raising a contradiction for Case 4, illustrated under the assumptions $\sigma_{\lambda-1}(v_a^\ell) = \mu'$, $\sigma_\lambda^{-1}(\mu') \in N_{X_L}(v_a^\ell) \setminus V(\mathcal P_a^\ell)$, and $\sigma_\lambda(\sigma_{\lambda-1}^{-1}(\mu')) = \kappa_a^\ell$. Since there exists $\mu \in V(\mathcal K^\ell)$ for which $\sigma_{\lambda-1}^{-1}(\mu) \notin V(\mathcal P_a^\ell) \cup V(\mathcal P_b^\ell)$ and $\sigma_{\lambda-1}(v_a^\ell) \in V(\mathcal K^\ell)$, Lemma \ref{lem:two_statements}(2) yields $\sigma_{\lambda-1}^{-1}(\kappa_b^\ell) \in V(\mathcal P_a^\ell) \setminus \{v_a^\ell\}$. This implies that $\sigma_{\lambda-1}$ violates Proposition \ref{prop:layer_independence}, regardless of what the value of $\ell$ is.}
    \label{fig:layer_independence_4}
\end{figure}

\medskip

We conclude that Proposition \ref{prop:layer_independence}(4) cannot have been the property violated by $\sigma_\lambda$. Together with the conclusions of the other three cases, we conclude that $\sigma_\lambda$ satisfies all properties of Proposition \ref{prop:layer_independence}, which contradicts $\sigma_\lambda$ failing to satisfy at least one of the properties, completing the proof.
\end{proof}

We can understand Propositions \ref{prop:rule_of_two} and \ref{prop:layer_independence} as separating elements of $V(Y_L)$ so that for any configuration $\sigma \in V(\mathscr{C})$, specific vertices of $Y_L$ can lie only upon specific subgraphs of $X_L$. In particular, for any $\ell \in [L]$, it follows from these two results that 
\begin{align*}
    \sigma(V(\mathcal P_a^\ell) \cup V(\mathcal P_b^\ell) \setminus \{v_a^\ell, v_b^\ell\}) \subseteq V(\mathcal K^\ell) \cup \{\kappa_a^\ell, \kappa_b^\ell\}.
\end{align*}
Proposition \ref{prop:layer_independence} and Lemma \ref{lem:two_statements} together now yield the following result.
\begin{proposition} \label{prop:path_images}
For any $\sigma \in V(\mathscr{C})$, the following two statements hold.
\begin{enumerate}
    \item If $\sigma^{-1}(V(\mathcal K^\ell)) \subset V(\mathcal P_a^\ell) \cup V(\mathcal P_b^\ell)$, then
    \begin{align*}
        \sigma(\{v_a^\ell, v_b^\ell\}) \subset V(\mathcal K^\ell) \implies |\sigma^{-1}(\{\kappa_a^\ell, \kappa_b^\ell\}) \cap ((V(\mathcal P_a^\ell) \setminus \{v_a^\ell\}) \cup (V(\mathcal P_b^\ell) \setminus \{v_b^\ell\}))| = 1.
    \end{align*}
    \item If $\sigma^{-1}(V(\mathcal K^\ell)) \not\subset V(\mathcal P_a^\ell) \cup V(\mathcal P_b^\ell)$, then 
    \begin{align*}
        & \sigma(v_a^\ell) \in V(\mathcal K^\ell) \implies \sigma^{-1}(\{\kappa_a^\ell, \kappa_b^\ell\}) \cap (V(\mathcal P_a^\ell) \setminus \{v_a^\ell\}) \neq \emptyset, \\
        & \sigma(v_b^\ell) \in V(\mathcal K^\ell) \implies \sigma^{-1}(\{\kappa_a^\ell, \kappa_b^\ell\}) \cap (V(\mathcal P_b^\ell) \setminus \{v_b^\ell\}) \neq \emptyset.
    \end{align*}
    \end{enumerate}
\end{proposition}

We now prove a third invariant of any configuration in $\mathscr{C}$. Toward this, we begin by introducing the following notion of ordering for elements of $V(\mathcal K^\ell)$ in the same partite set, which is illustrated in Figure \ref{fig:left}.

\begin{definition} \label{defn:left}
For $\sigma \in V(\mathscr{C})$, $\ell \in [L]$, and $\mu_1, \mu_2 \in V(\mathcal K^\ell)$ in the same partite set, say that $\mu_1$ is \textit{\textcolor{red}{left}} of $\mu_2$ on $\sigma$ if (exactly) one of the following holds:
\begin{enumerate}
    \item $\sigma^{-1}(\{\mu_1, \mu_2\}) \subset V(\mathcal P_a^\ell)$ and $d(\sigma^{-1}(\mu_2), v_a^\ell) < d(\sigma^{-1}(\mu_1), v_a^\ell)$,
    \item $\sigma^{-1}(\{\mu_1, \mu_2\}) \subset V(\mathcal P_b^\ell)$ and $d(\sigma^{-1}(\mu_1), v_b^\ell) < d(\sigma^{-1}(\mu_2), v_b^\ell)$,
    \item $\sigma^{-1}(\mu_1) \in V(\mathcal P_a^\ell)$ and $\sigma^{-1}(\mu_2) \in V(\mathcal P_b^\ell)$.
\end{enumerate}
\end{definition}

\begin{figure}[ht]
  \centering
  \subfloat[Definition \ref{defn:left}(1).]{\includegraphics[width=0.32\textwidth]{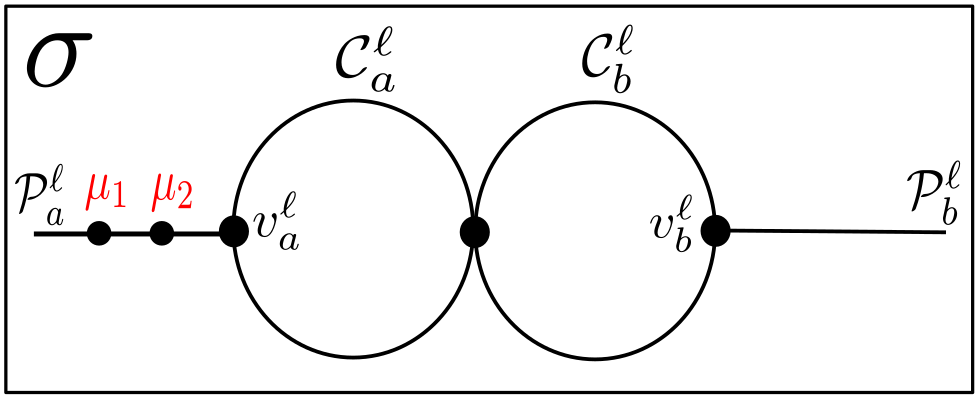}\label{fig:left_1}}
  \hfill
  \subfloat[Definition \ref{defn:left}(2).]{\includegraphics[width=0.32\textwidth]{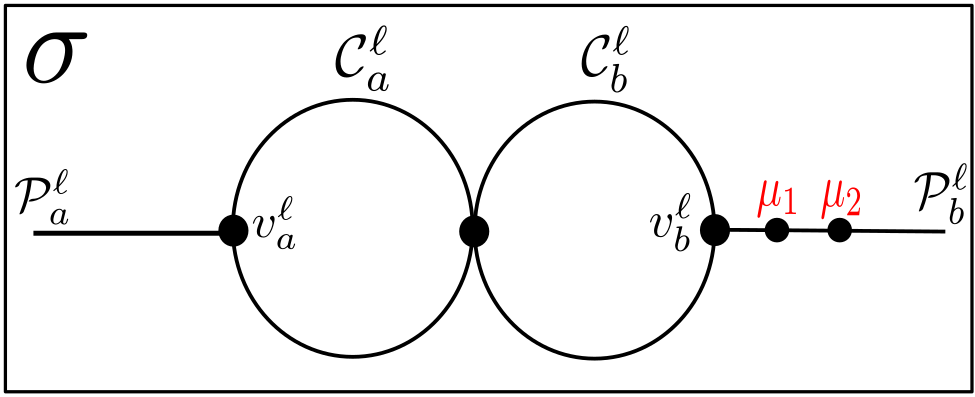}\label{fig:left_2}}
  \hfill
  \subfloat[Definition \ref{defn:left}(3).]{\includegraphics[width=0.32\textwidth]{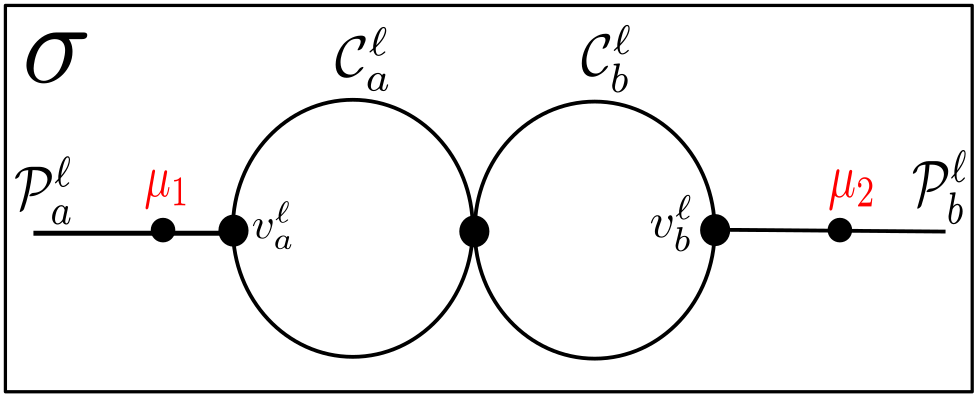}\label{fig:left_3}}
  \caption{An illustration of Definition \ref{defn:left}.}
  \label{fig:left}
\end{figure}

Since $\sigma_s^{-1}(V(\mathcal K^\ell)) \subset V(\mathcal P_a^\ell) \cup V(\mathcal P_b^\ell)$, it follows from Definition \ref{defn:left} that for any $\mu_1, \mu_2 \in V(\mathcal K^\ell)$ in the same partite set, either $\mu_1$ is left of $\mu_2$ on $\sigma_s$ or $\mu_2$ is left of $\mu_1$ on $\sigma_s$. The following proposition asserts that the left relation established by $\sigma_s$ cannot change for other $\sigma \in V(\mathscr{C})$.

\begin{proposition} \label{prop:order_invariance}
Take $\ell \in [L]$ and $\mu_1, \mu_2 \in V(\mathcal K^\ell)$ in the same partite set, with $\mu_1$ left of $\mu_2$ on $\sigma_s$. If $\sigma \in V(\mathscr{C})$ is such that $\sigma^{-1}(\{\mu_1, \mu_2\}) \subset V(\mathcal P_a^\ell) \cup V(\mathcal P_b^\ell)$, then $\mu_1$ is left of $\mu_2$ in $\sigma$.
\end{proposition}

\begin{proof}
Let $\Sigma = \{\sigma_i\}_{i=0}^\lambda$ with $\sigma_0 = \sigma_s$ and $\sigma_\lambda = \sigma$ be a swap sequence in $\FS(X_L, Y_L)$ starting from $\sigma_s$ and ending at $\sigma$, where $\sigma^{-1}(\{\mu_1, \mu_2\}) \subset V(\mathcal P_a^\ell) \cup V(\mathcal P_b^\ell)$. By Proposition \ref{prop:layer_independence}(4), any $\sigma_i \in \Sigma$ satisfies 
\begin{align*}
    |\sigma_i^{-1}(V(\mathcal K^\ell)) \setminus (V(\mathcal P_a^\ell) \cup V(\mathcal P_b^\ell))| \leq 1,
\end{align*}
so that in particular,
\begin{align*}
    |\sigma_i^{-1}(\{\mu_1, \mu_2\}) \setminus (V(\mathcal P_a^\ell) \cup V(\mathcal P_b^\ell))| \leq 1.
\end{align*}
Consider the subsequence $\Sigma' = \{\sigma_{i_j}\}_{j=0}^{\lambda'} \subseteq \Sigma$, $\lambda' \leq \lambda$ with $i_0 = 0$ and then consisting of all configurations $\sigma_i \in \Sigma$ for which 
\begin{align*}
    |\sigma_{i-1}^{-1}(\{\mu_1, \mu_2\}) \setminus (V(\mathcal P_a^\ell) \cup V(\mathcal P_b^\ell))| = 1 \text{ and } |\sigma_i^{-1}(\{\mu_1, \mu_2\}) \setminus (V(\mathcal P_a^\ell) \cup V(\mathcal P_b^\ell))| = 0.
\end{align*}
If $\mu_1$ is left of $\mu_2$ on $\sigma_{i_{\lambda'}}$, then $\mu_1$ is left of $\mu_2$ on $\sigma_k$ for all $k \geq \lambda'$. Indeed, the construction of $\Sigma'$ and $\sigma^{-1}(\{\mu_1, \mu_2\}) \subset V(\mathcal P_a^\ell) \cup V(\mathcal P_b^\ell)$ imply that $\sigma_k^{-1}(\mu_1)$ and $\sigma_k^{-1}(\mu_2)$ remain upon the same path subgraphs in $X_L$ for all such $k$, and this claim now follows if $\mu_1$ is left of $\mu_2$ on $\sigma_{i_{\lambda'}}$ due to Definition \ref{defn:left}(3) and from $\{\mu_1, \mu_2\} \notin E(Y_L)$ otherwise. Since $\lambda \geq \lambda'$, it suffices to show that $\mu_1$ is left of $\mu_2$ on $\sigma_{i_{\lambda'}}$, toward which we can induct on $j$ to show that $\mu_1$ is left of $\mu_2$ on $\sigma_{i_j}$ for all $0 \leq j \leq \lambda'$. The statement holds for $j=0$ by assumption, so assume $\mu_1$ is left of $\mu_2$ on $\sigma_{i_j}$ for some $0 \leq j < \lambda'$. Take the unique vertex $\mu \in \{\mu_1, \mu_2\}$ such that 
\begin{align*}
    \sigma_{i_{j+1}-1}^{-1}(\mu) \notin V(\mathcal P_a^\ell) \cup V(\mathcal P_b^\ell).
\end{align*}
It is now straightforward to inductively argue, relying on the definition of $\Sigma'$, Proposition \ref{prop:layer_independence}(4), and the fact that $\{\mu_1, \mu_2\} \notin E(Y_L)$, that the other vertex in $\{\mu_1, \mu_2\}$ (i.e., not $\mu$) must remain upon the same path subgraph in $X_L$ over all configurations $\sigma_k$ for $i_j \leq k \leq i_{j+1}$. With this observation, it quickly follows, by breaking into cases based on which statement of Definition \ref{defn:left} yields $\mu_1$ left of $\mu_2$ on $\sigma_{i_j}$ and relying on the fact that $\{\mu_1, \mu_2\} \notin E(Y_L)$, that $\mu_1$ is left of $\mu_2$ on $\sigma_{i_{j+1}}$.
\end{proof}

We are now ready to prove the main result (in conjunction with Proposition \ref{prop:order_invariance}) we will need in order to derive a lower bound on the diameter of $\mathscr{C}$.

\begin{proposition} \label{prop:knob_extract}
For any configuration $\sigma \in V(\mathscr{C})$ and $\ell \in [L-1]$, 
\begin{enumerate}
    \item $\sigma^{-1}(\kappa_a^{\ell+1}) \notin V(\mathcal P_a^\ell) \cup V(\mathcal P_b^\ell) \implies V(\mathcal K_b^\ell) \subset \sigma(V(\mathcal P_a^\ell))$,
    \item $\sigma^{-1}(\kappa_b^{\ell+1}) \notin V(\mathcal P_a^\ell) \cup V(\mathcal P_b^\ell) \implies V(\mathcal K_a^\ell) \subset \sigma(V(\mathcal P_a^\ell))$.
\end{enumerate}
\end{proposition}

\begin{proof}
We will take $\sigma \in V(\mathscr{C})$ and $\ell \in [L-1]$ such that $\sigma^{-1}(\kappa_a^{\ell+1}) \notin V(\mathcal P_a^\ell) \cup V(\mathcal P_b^\ell)$; proving the latter implication when assuming $\sigma^{-1}(\kappa_b^{\ell+1}) \notin V(\mathcal P_a^\ell) \cup V(\mathcal P_b^\ell)$ can be done analogously. By Proposition \ref{prop:layer_independence}(4) and the assumption on $\sigma^{-1}(\kappa_a^{\ell+1})$,
\begin{align*}
    |\sigma^{-1}(V(\mathcal K^\ell)) \setminus (V(\mathcal P_a^\ell) \cup V(\mathcal P_b^\ell))| = 1,
\end{align*}
from which it follows that $V(\mathcal K_b^\ell) \subset \sigma(V(\mathcal P_a^\ell) \cup V(\mathcal P_b^\ell))$. To prove that $V(\mathcal K_b^\ell) \subset \sigma(V(\mathcal P_a^\ell))$, let $\Sigma = \{\sigma_i\}_{i=0}^\lambda$ be a swap sequence from $\sigma_0 = \sigma_s$ to $\sigma_\lambda = \sigma$: note that $\lambda \geq 1$, since 
\begin{align*}
    \sigma_s^{-1}(\kappa_a^{\ell+1}) \in V(\mathcal P_a^\ell) \cup V(\mathcal P_b^\ell).
\end{align*}
Consider the largest $\xi < \lambda$ for which
\begin{align*}
    \sigma_\xi^{-1}(\kappa_a^{\ell+1}) \in V(\mathcal P_a^\ell) \cup V(\mathcal P_b^\ell),
\end{align*}
noting that $\xi < \lambda$ is well-defined, since $\sigma_s^{-1}(\kappa_a^{\ell+1}) \in V(\mathcal P_a^\ell) \cup V(\mathcal P_b^\ell)$. It must be that 
\begin{align} \label{eq:knob_extract_eq_1}
    & \sigma_\xi^{-1}(\kappa_a^{\ell+1}) \in \{v_a^\ell, v_b^\ell\},
    & \sigma_{\xi+1}^{-1}(\kappa_a^{\ell+1}) \notin V(\mathcal P_a^\ell) \cup V(\mathcal P_b^\ell).
\end{align}
Since $N_{Y_L}(\kappa_a^{\ell+1}) \cap V(\mathcal K_a^\ell) = \emptyset$, we deduce that 
\begin{align} \label{eq:knob_extract_eq_2}
    \sigma_\xi^{-1}(V(\mathcal K_a^\ell) \setminus \{\kappa_a^{\ell+1}\}) \subset V(\mathcal P_a^\ell) \cup V(\mathcal P_b^\ell),
\end{align}
as otherwise, we would have that
\begin{align*}
    |\sigma_{\xi+1}^{-1}(V(\mathcal K_a^\ell)) \setminus (V(\mathcal P_a^\ell) \cup V(\mathcal P_b^\ell))| \geq 2,
\end{align*}
contradicting Proposition \ref{prop:layer_independence}(4). From \eqref{eq:knob_extract_eq_1}, \eqref{eq:knob_extract_eq_2}, and Proposition \ref{prop:order_invariance}, we further observe that
\begin{align} \label{eq:knob_extract_eq_3}
    \sigma_\xi^{-1}(V(\mathcal K_a^\ell) \setminus \{\kappa_a^{\ell+1}\}) \subset V(\mathcal P_b^\ell),
\end{align}
as for any $\mu \in V(\mathcal K_a^\ell)$, $\kappa_a^{\ell+1}$ is left of $\mu$ on $\sigma_\xi$ since $\kappa_a^{\ell+1}$ is left of $\mu$ on $\sigma_s$. By the definition of $\xi$, 
\begin{align*}
    \sigma_k^{-1}(\kappa_a^{\ell+1}) \notin V(\mathcal P_a^\ell) \cup V(\mathcal P_b^\ell) \text{ for } k > \xi,
\end{align*}
so by Proposition \ref{prop:layer_independence}(4),
\begin{align} \label{eq:knob_extract_eq_4}
    \{\kappa_a^{\ell+1}\} = \sigma_k^{-1}(V(\mathcal K^\ell)) \setminus (V(\mathcal P_a^\ell) \cup V(\mathcal P_b^\ell)) \text{ for } k > \xi.
\end{align}
If $V(\mathcal K_b^\ell) \subset \sigma_\xi(V(\mathcal P_a^\ell))$, \eqref{eq:knob_extract_eq_4} can be used to inductively prove that
\begin{align*}
    V(\mathcal K_b^\ell) \subset \sigma_k(V(\mathcal P_a^\ell)) \text{ for } k > \xi,
\end{align*}
with the induction basis following from \eqref{eq:knob_extract_eq_1}. In particular, $V(\mathcal K_b^\ell) \subset \sigma(V(\mathcal P_a^\ell))$, which is the desired statement. Thus, we now proceed under the assumption $|V(\mathcal K_b^\ell) \setminus \sigma_\xi(V(\mathcal P_a^\ell))| \geq 1$. Further assume (towards a contradiction) that there exists
\begin{align} \label{eq:knob_extract_eq_5}
    \mu \in \left(V(\mathcal K_b^\ell) \setminus \sigma_\xi(V(\mathcal P_a^\ell))\right) \cap \sigma_\xi(V(\mathcal P_b^\ell)).
\end{align}
Then from \eqref{eq:knob_extract_eq_3}, \eqref{eq:knob_extract_eq_5}, and the fact that the LHS and RHS have equal cardinality,
\begin{align} \label{eq:knob_extract_eq_6}
    (V(\mathcal K_a^\ell) \setminus \{\kappa_a^{\ell+1}\}) \cup \{\mu\} = \sigma_\xi(V(\mathcal P_b^\ell)).
\end{align}
See \eqref{eq:knob_extract_eq_1}; \eqref{eq:knob_extract_eq_6} immediately raises a contradiction if $\sigma_\xi^{-1}(\kappa_a^{\ell+1}) = v_b^\ell$, and if $\sigma_\xi^{-1}(\kappa_a^{\ell+1}) = v_a^\ell$, \eqref{eq:knob_extract_eq_1} and Proposition \ref{prop:path_images}(2) (the hypotheses necessary for the implication follow from \eqref{eq:knob_extract_eq_1} and \eqref{eq:knob_extract_eq_6}) imply that
\begin{align*}
    \sigma_\xi(V(\mathcal P_b^\ell)) = \sigma_{\xi+1}(V(\mathcal P_b^\ell)) \text{ and } \sigma_{\xi+1}^{-1}(\{\kappa_a^\ell, \kappa_b^\ell\}) \cap (V(\mathcal P_b^\ell) \setminus \{v_b^\ell\}) \neq \emptyset,
\end{align*}
respectively, raising a contradiction on \eqref{eq:knob_extract_eq_6}. Therefore,
\begin{align} \label{eq:knob_extract_eq_7}
    \left(V(\mathcal K_b^\ell) \setminus \sigma_\xi(V(\mathcal P_a^\ell))\right) \cap \sigma_\xi(V(\mathcal P_b^\ell)) = \emptyset.
\end{align}
If it were true that $|V(\mathcal K_b^\ell) \setminus \sigma_\xi(V(\mathcal P_a^\ell))| \geq 2$, Proposition \ref{prop:layer_independence}(4) would imply
\begin{align*}
    |V(\mathcal K_b^\ell) \setminus (V(\mathcal P_a^\ell) \cup V(\mathcal P_b^\ell))| \leq 1,
\end{align*}
so there would exist $\mu \in V(\mathcal K_b^\ell) \setminus \sigma_\xi(V(\mathcal P_a^\ell))$ such that $\sigma_\xi^{-1}(\mu) \in V(\mathcal P_a^\ell) \cup V(\mathcal P_b^\ell)$, contradicting \eqref{eq:knob_extract_eq_7}. Thus, 
\begin{align*}
    |V(\mathcal K_b^\ell) \setminus \sigma_\xi(V(\mathcal P_a^\ell))| = 1.
\end{align*}
Letting $\mu$ denote the unique element in this set, it must be that $\mu \notin \sigma_\xi(V(\mathcal P_b^\ell))$ by \eqref{eq:knob_extract_eq_7}, so that
\begin{align*}
    \sigma_\xi^{-1}(\mu) \in \sigma_\xi^{-1}(V(\mathcal K^\ell)) \setminus (V(\mathcal P_a^\ell) \cup V(\mathcal P_b^\ell)).
\end{align*}
It thus follows from Proposition \ref{prop:layer_independence}(4) that $\sigma_{\xi+1}(\sigma_\xi^{-1}(\kappa_a^{\ell+1})) = \mu$, so that $V(\mathcal K_b^\ell) \subset \sigma_{\xi+1}(V(\mathcal P_a^\ell))$. Now $V(\mathcal K_b^\ell) \subset \sigma(V(\mathcal P_a^\ell))$ can be established by arguing as when we assumed $V(\mathcal K_b^\ell) \subset \sigma_\xi(V(\mathcal P_a^\ell))$.
\end{proof}

\subsection{Extractions} \label{subsec:extractions}

Equipped with the results of Subsection \ref{subsec:configurations_in_c}, we are now ready to establish that there are two configurations in $\mathscr{C}$, one of them being $\sigma_s$, with distance $e^{\Omega(n)}$. The idea is to construct a series of swaps, layer-by-layer. For $\ell \in [L-1]$, each iteration on layer $\ell+1$ will require a certain number of iterations in layer $\ell$. We formalize this in Definition \ref{defn:extraction} by a notion that we refer to as $\ell$-extractions, which we illustrate in Figure \ref{fig:extraction}.

\begin{definition} \label{defn:extraction}
For $\ell \in [L]$ and $\sigma, \tau \in V(\mathscr{C})$, we say that $\tau$ is an \textit{\textcolor{red}{$\ell$-extraction of $\sigma$}} if either:
\begin{enumerate}
    \item $V(\mathcal K_a^\ell) \subset \sigma(V(\mathcal P_a^\ell))$ and $V(\mathcal K_a^\ell) \cap \tau(V(\mathcal P_a^\ell)) = \emptyset$, $V(\mathcal K_b^\ell) \subset \tau(V(\mathcal P_a^\ell))$,
    \item $V(\mathcal K_b^\ell) \subset \sigma(V(\mathcal P_a^\ell))$ and $V(\mathcal K_b^\ell) \cap \tau(V(\mathcal P_a^\ell)) = \emptyset$, $V(\mathcal K_a^\ell) \subset \tau(V(\mathcal P_a^\ell))$.
\end{enumerate}
\end{definition}

In other words, if $\tau$ is an $\ell$-extraction of $\sigma$, one of the two partite sets of $V(\mathcal K^\ell)$ is a subset of $\sigma(V(\mathcal P_a^\ell))$. Then $\tau$ ``extracts" this partite set out of $\mathcal P_a^\ell$ and replaces it with the other partite set of $V(\mathcal K^\ell)$, which is then a subset $\tau(V(\mathcal P_a^\ell))$.

\begin{figure}[ht]
  \centering
  \subfloat[Definition \ref{defn:extraction}(1).]{\includegraphics[width=0.49\textwidth]{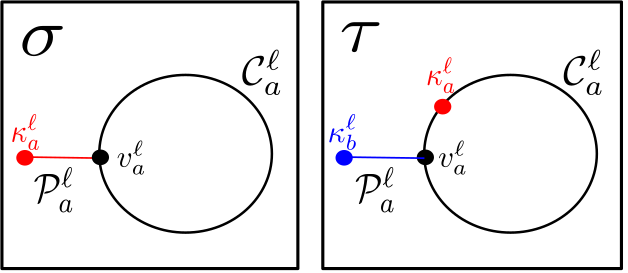}\label{fig:extract_a}}
  \hfill
  \subfloat[Definition \ref{defn:extraction}(2).]{\includegraphics[width=0.49\textwidth]{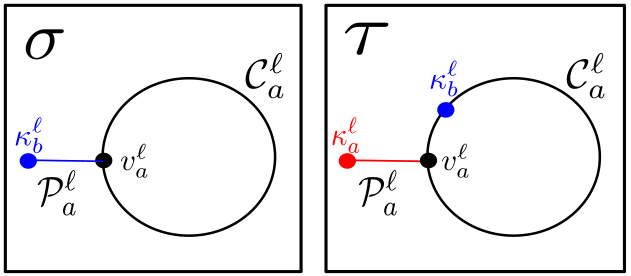}\label{fig:extract_b}}
  \caption{An illustration of Definition \ref{defn:extraction}. Red subgraphs/vertices corresponding to preimages of $V(\mathcal K_a^\ell)$, while blue subgraphs/vertices correspond to preimages of $V(\mathcal K_b^\ell)$. By Proposition \ref{prop:order_invariance}, the relative ordering of the vertices in a partite set of $V(\mathcal K^\ell)$ is the same as in $\sigma_s$, so the appropriate knob vertex always lies upon the leftmost vertex in $\mathcal P_a^\ell$.}
  \label{fig:extraction}
\end{figure}

For use in the proof of Proposition \ref{prop:L_knob}, we also introduce the following definition, corresponding to knob vertices in $Y_L$ rotating about their corresponding cycle subgraphs in $X_L$. Recall from Subsection \ref{subsec:construction} that for all $\ell \in [L]$, 
\begin{align*}
    & |V(\mathcal S_a^\ell)| = |V(\mathcal S_b^\ell)| = 15,
    & |V(\mathcal C_a^\ell)| = |V(\mathcal C_b^\ell)| = 16.
\end{align*}

\begin{definition} \label{defn:knob_rotation}
    For $\ell \in [L]$, $\mu_a \in N_{Y_L}(\kappa_a^\ell)$ and a positive integer $\lambda$ such that $\lambda \equiv 0 \pmod{16}$, a \textit{\textcolor{red}{$\kappa_a^\ell$-rotation with $\mu_a$}} is a swap sequence $\{\sigma_i\}_{i=0}^\lambda$ for which $\sigma_i(V(\mathcal C_a^\ell)) = \{\mu_a\} \cup V(\mathcal S_a^\ell)$ for all $0 \leq i \leq \lambda$ and there exists an enumeration $V(\mathcal C_a^\ell) = \{v_0, v_1, \dots, v_{15}\}$ such that $\{v_{i-1}, v_i\} \in E(\mathcal C_a^\ell)$ for all $i \in [15]$ and $\sigma_j(v_i) = \kappa_a^\ell$ whenever $i \equiv j \pmod{16}$. Similarly, for $\ell \in [L]$, $\mu_b \in N_{Y_L}(\kappa_b^\ell)$ and a positive integer $\lambda$ such that $\lambda \equiv 0 \pmod{16}$, a \textit{\textcolor{red}{$\kappa_b^\ell$-rotation with $\mu_b$}} is a swap sequence $\{\sigma_i\}_{i=0}^\lambda$ for which $\sigma_i(V(\mathcal C_b^\ell)) = \{\mu_b\} \cup V(\mathcal S_b^\ell)$ for all $0 \leq i \leq \lambda$ and there exists an enumeration $V(\mathcal C_b^\ell) = \{v_0, v_1, \dots, v_{15}\}$ such that $\{v_{i-1}, v_i\} \in E(\mathcal C_b^\ell)$ for all $i \in [15]$ and $\sigma_j(v_i) = \kappa_b^\ell$ whenever $i \equiv j \pmod{16}$. 
\end{definition}

Note that Definition \ref{defn:knob_rotation}, which is illustrated in Figure \ref{fig:knob_rotation}, corresponds to a cyclic rotation of all elements in $\sigma_0(V(\mathcal C_a^\ell)) \setminus \{\kappa_a^\ell\}$ about the knob vertex $\kappa_a^\ell$, which is fixed in the same position since $\sigma_0(v_0) = \sigma_\lambda(v_0) = \kappa_a^\ell$. The direction and length $\lambda$ of this rotation depend on the enumeration of the vertices in the relevant cycle and the value $\lambda/16$, respectively.

\begin{figure}[ht]
    \centering
    \includegraphics[width=\textwidth]{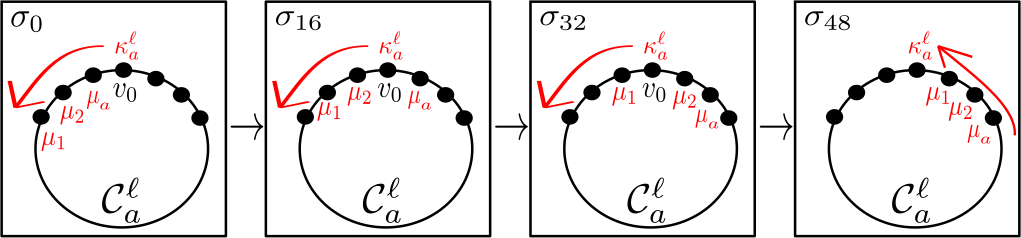}
    \caption{An illustration of a $\kappa_a^\ell$-rotation with $\mu_a$, where $\lambda=3\cdot 16 = 48$ and $\kappa_a^\ell$ rotates counterclockwise around $\mathcal C_a^\ell$. Here, $\mu_1, \mu_2 \in V(\mathcal S_a^\ell) \setminus \{\kappa_a^\ell\}$. As $\kappa_a^\ell$ rotates over $\mathcal C_a^\ell$, it cyclically rotates all elements of $(V(\mathcal S_a^\ell) \setminus \{\kappa_a^\ell\}) \cup \{\mu_a\}$ about it. In this case, every such element moves three vertices clockwise along $V(\mathcal C_a^\ell) \setminus \{v_0\}$.}
    \label{fig:knob_rotation}
\end{figure}

The final configuration $\sigma_f$ in $\mathscr{C}$ for which we will argue that $d(\sigma_s, \sigma_f) = \Omega(n^{L-1})$ is going to be an $L$-extraction of $\sigma_s$ (of the kind from Definition \ref{defn:extraction}(1)). We begin by showing that for any $\ell \in [L]$, $\ell$-extractions of $\sigma_s$ exist in $\mathscr{C}$. This will follow as an immediate corollary of Proposition \ref{prop:L_knob} by taking $\eta = 1$ for this value of $\ell$, since $\sigma_\lambda$ is then an $\ell$-extraction of $\sigma_s$.

\begin{proposition} \label{prop:L_knob}
For any positive integer $\eta$ and $\ell \in [L]$, there exists a swap sequence $\{\sigma_i\}_{i=0}^\lambda$, $\sigma_0 = \sigma_s$ with a subsequence $\{\sigma_{i_j}\}_{j=0}^\eta$, $i_0 = 0$, $i_\eta = \lambda$ such that
\begin{enumerate}
    \item for every $j \in [\eta]$, $\sigma_{i_j}$ is an $\ell$-extraction of $\sigma_{i_{j-1}}$;
    \item for every $j \in [\eta]$ and $\mu \in V(\mathcal K^L)$, there exists a $\kappa_a^L$-rotation with $\mu$ and $\kappa_b^L$-rotation with $\mu$ that is a contiguous subsequence of $\{\sigma_i\}_{i=i_{j-1}}^{i_j}$.
\end{enumerate}
\end{proposition}

\begin{proof}
We deviate from our usual practice in Subsections \ref{subsec:configurations_in_c} and \ref{subsec:extractions} of assuming that everything proceeds under the context of some fixed $L \geq 1$, and establish Proposition \ref{prop:L_knob} via induction on $L$. Specifically, we will show by induction on $L \geq 1$ that for any fixed $L \geq 1$, Proposition \ref{prop:L_knob} holds for the graphs $X_L$ and $Y_L$. During the induction step, in another deviation from our usual practice, we shall be more explicit about the pairs of graphs and the starting configurations that we reference for sake of clarity. 

We begin with the induction basis, $L=1$. Here, $\ell = 1$ is the only value of $\ell$ for which Proposition \ref{prop:L_knob} applies. Consider the following sequence of swaps from $\sigma_s$: Figure \ref{fig:L=1} illustrates the first three steps of this procedure.
\begin{enumerate}
    \item Perform a $\kappa_b^1$-rotation with $\sigma_s(v_b^1)$ to move $\sigma_s(v_b^1)$ to $v^1$.
    \item Perform a $\kappa_a^1$-rotation with $\sigma_s(v_b^1)$ to move $\sigma_s(v_b^1)$ to $v_a^1$.
    \item Swap $\sigma_s(v_b^1)$ as far left through $V(\mathcal P_a^1)$ as possible, yielding a vertex $\mu \in V(\mathcal K_a^1)$ upon $v_a^1$.
    \item Perform a $\kappa_a^1$-rotation with $\mu$ to move $\mu$ to $v^1$.
    \item Perform a $\kappa_b^1$-rotation with $\mu$ to move $\mu$ to $v_b^1$.
    \item Swap $\mu$ as far right through $V(\mathcal P_b^1)$ as possible, producing a vertex in $V(\mathcal K_b^1)$ upon $v_b^1$.
\end{enumerate}
It is straightforward to conclude that repeating this algorithm $15$ times (since $|V(\mathcal K_a^\ell)| = |V(\mathcal K_b^\ell)| = 15$) from $\sigma_s$ (adapted to the mapping upon $v_b^1$, then the mapping upon $v_a^1$, for subsequent iterations) yields a $1$-extraction $\sigma_{i_1}$ of $\sigma_s$, namely of the kind in Definition \ref{defn:extraction}(1), since we have
\begin{align*}
    & \sigma_{i_1}(V(\mathcal P_b^1)) = \sigma_s(V(\mathcal P_a^1) \setminus \{v_a^1\}) = V(\mathcal K_a^1),
    & \sigma_{i_1}(V(\mathcal P_a^1) \setminus \{v_a^1\}) = \sigma_s(V(\mathcal P_b^1)) = V(\mathcal K_b^1),
\end{align*}
and that for every $\mu \in V(\mathcal K^1)$, there exists a $\kappa_a^1$-rotation with $\mu$ and $\kappa_b^1$-rotation with $\mu$ that is a contiguous subsequence of the resulting swap sequence. It is similarly straightforward to see that we can repeat this algorithm to interchange the positions of $V(\mathcal K_a^1)$ and $V(\mathcal K_b^1)$ arbitrarily many times (i.e., for any positive integer $\eta$), with a $\kappa_a^1$-rotation with $\mu$ and $\kappa_b^1$-rotation with $\mu$ for every $\mu \in V(\mathcal K^1)$ executed as a contiguous subsequence of every such interchange. On even iterations of this interchange, we simply switch the roles of $V(\mathcal K_a^1)$ and $V(\mathcal K_b^1)$ in the above algorithm, resulting in $1$-extractions of the kind in Definition \ref{defn:extraction}(2).

\begin{figure}[ht]
    \centering
    \includegraphics[width=\textwidth, height=6cm]{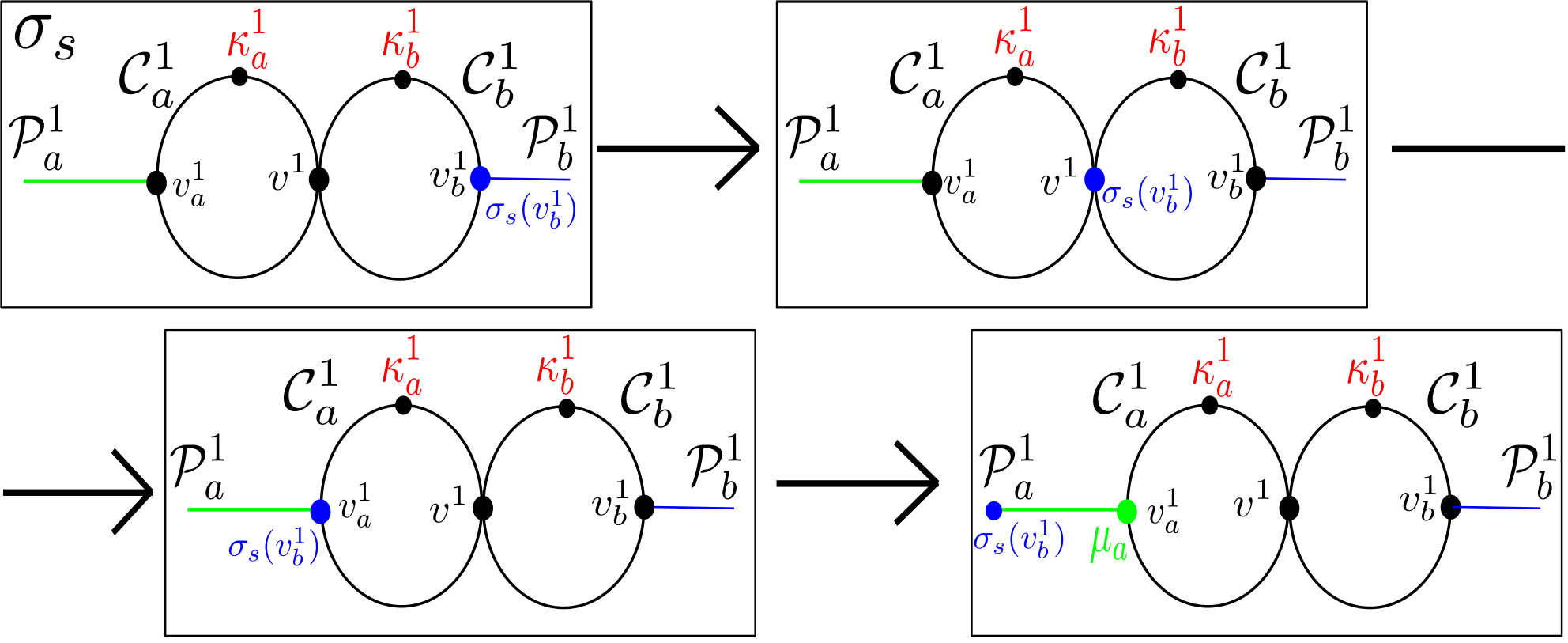}
    \caption{An illustration of the first half of the sequence of swaps discussed for the induction basis, $L=1$. Preimages of $V(\mathcal K_a^1)$ are colored green, and preimages of $V(\mathcal K_b^1)$ are colored blue. This segment of the sequence of swaps involves a $\kappa_b^1$-rotation with $\sigma_s(v_b^1)$, a $\kappa_a^1$-rotation with $\sigma_s(v_b^1)$, and a sequence of swaps moving $\sigma_s(v_b^1)$ left through $V(\mathcal P_a^\ell)$. A vertex $\mu_a \in V(\mathcal K_a^1)$ now lies upon $v_a^\ell$; we can similarly move $\mu_a$ to the right. Continuing until every vertex of $V(\mathcal K^1)$ is moved in an analogous subroutine yields a $1$-extraction $\sigma_{i_1}$ of $\sigma_s$. We can interchange $\mathcal K_a^1$ and $\mathcal K_b^1$ in this way arbitrarily many times.}
    \label{fig:L=1}
\end{figure}
Now assume Proposition \ref{prop:L_knob} holds for some fixed $L \geq 1$ (i.e., for this fixed $L \geq 1$, Proposition \ref{prop:L_knob} holds for the graphs $X_L$ and $Y_L$). By the induction hypothesis applied on $\eta = 31$ and $\ell \in [L]$, we can extract a swap sequence $\{\sigma_i\}_{i=0}^\lambda$ in $V(\mathscr{C}(X_L, Y_L))$ with $\sigma_0 = \sigma_s(X_L, Y_L)$ and with a subsequence $\{\sigma_{i_j}\}_{j=0}^{31}$ satisfying Proposition \ref{prop:L_knob}. Now consider $X_{L+1}$ and $Y_{L+1}$, which has corresponding starting configuration $\sigma_s(X_{L+1}, Y_{L+1})$ in the connected component $\mathscr{C}(X_{L+1}, Y_{L+1})$, which we denote $\sigma_s$ and $\mathscr{C}$, respectively. In an abuse of notation, for the rest of the present proof we let $X_L$ denote the first $L$ layers of $X_{L+1}$ and $Y_L = Y_{L+1} |_{\sigma_s(V(X_L))}$. These subgraphs are isomorphic to the graphs $X_L$ and $Y_L$ as they were originally defined during their construction in Subsection \ref{subsec:construction}, and under these isomorphisms, $\sigma_s$ restricted to $X_L$ can be understood to be the same as $\sigma_s(X_L, Y_L)$ as defined in Subsection \ref{subsec:construction}. Furthermore, the swap sequence $\{\sigma_i\}_{i=0}^\lambda$ can be understood as being in $\mathscr{C}$, with $\sigma_0 = \sigma_s$, if we set
\begin{align*}
    \sigma_i(v) = \sigma_s(v) \text{ for all } v \in V(X_{L+1}) \setminus V(X_L), \ i = 0, \dots, \lambda.
\end{align*}
As such, it follows from the induction hypothesis that Proposition \ref{prop:L_knob} holds for $(X_{L+1}, Y_{L+1})$ if we take $\ell \in [L]$, and all that remains is to confirm that Proposition \ref{prop:L_knob} holds for $(X_{L+1}, Y_{L+1})$ for $\ell = L+1$. In the proceeding argument, we assume that the swap sequence $\{\sigma_i\}_{i=0}^\lambda$ we extracted above using the induction hypothesis was for $\ell = L$. In a similar vein, given some $\sigma \in V(\mathscr{C})$ with $\sigma(X_L) = V(Y_L)$ and $\sigma_i$ from the swap sequence $\{\sigma_i\}_{i=0}^\lambda$, define the \textit{\textcolor{red}{extension}} of $\sigma_i$ with respect to $\sigma$ to be the configuration\footnote{Whenever we construct such extensions in the forthcoming argument, it will be clear that they lie in $V(\mathscr{C})$.} $\tau \in V(\FS(X_{L+1}, Y_{L+1}))$ with
\begin{itemize}
    \item $\tau(v) = \sigma(v)$ for all $v \in V(X_{L+1}) \setminus V(X_L)$;
    \item $\tau(v) = \sigma_i(v)$ for all $v \in V(X_L)$.
\end{itemize}
We will say that we \textit{\textcolor{red}{extend}} $\sigma_i$ with respect to $\sigma$, and will generally apply this notion en masse to subsequences of $\{\sigma_i\}_{i=0}^\lambda$ with respect to a single configuration of $V(\mathscr{C})$. 

\medskip

We will now construct a swap sequence $\{\sigma'_i\}_{i=0}^{\lambda'}$ in $V(\mathscr{C})$, with $\sigma'_0 = \sigma_s$, satisfying Proposition \ref{prop:L_knob} for $\eta=1$. This is illustrated in Figure \ref{fig:general_L}. From the swap sequence $\{\sigma_i\}_{i=0}^\lambda$, with subsequence $\{\sigma_{i_j}\}_{j=0}^{31}$ as discussed before, consider $\{\sigma_i\}_{i=i_0}^{i_1}$, which, by the induction hypothesis, has a contiguous subsequence $\{\sigma_i\}_{i=j_1}^{k_1}$ that is a $\kappa_b^L$-rotation with $\kappa_b^{L+1}$. Let $t_1$ be such that $j_1 \leq t_1 \leq k_1$ and $\sigma_{t_1}(\kappa_b^{L+1}) = v_b^{L, L+1}$: the observation that such a $t_1$ exists follows quickly from the restrictions of Definition \ref{defn:knob_rotation}. We construct a swap sequence $\mathscr{S}_1$ in $\mathscr{C}$ by merging $\{\sigma_i^{1,1}\}_{i=0}^{t_1-i_0}, \{\sigma_i^{1,2}\}_{i=0}^{z_1}$, and $\{\sigma_i^{1,3}\}_{i=0}^{i_1-t_1}$, which we now define.
\begin{enumerate}
    \item Extend $\{\sigma_i\}_{i=i_0}^{t_1}$ with respect to $\sigma_s$, yielding $\{\sigma_i^{1,1}\}_{i=0}^{t_1-i_0}$.
    
    \item Let $\{\sigma_i^{1,2}\}_{i=0}^{z_1}$, with $\sigma_0^{1,2} = \sigma_{t_1-i_0}^{1,1}$, be a $\kappa_b^{L+1}$-rotation with $\sigma_{t_1-i_0}^{1,1}(v_b^{L+1})$, with length such that $\sigma_{t_1-i_0}^{1,1}(v_b^{L+1})$ is moved to $v^{L+1}$.
    
    \item Extend $\{\sigma_i\}_{i=t_1}^{i_1}$ with respect to $\sigma_{t_1-i_0}^{1,1}$, yielding $\{\sigma_i^{1,3}\}_{i=0}^{i_1-t_1}$.
\end{enumerate}
Now take the subsequence $\{\sigma_i\}_{i=i_1}^{i_2}$ of $\{\sigma_i\}_{i=0}^\lambda$, which has contiguous subsequence $\{\sigma_i\}_{i=j_2}^{k_2}$ that is a $\kappa_a^L$-rotation with $\kappa_a^{L+1}$, and $t_2$ such that $j_2 \leq t_2 \leq k_2$ and $\sigma_{t_2}(\kappa_a^{L+1}) = v_a^{L, L+1}$. Construct $\mathscr{S}_2$ by merging $\{\sigma_i^{2,1}\}_{i=0}^{t_2-i_1}, \{\sigma_i^{2,2}\}_{i=0}^{t_2-i_1}$, and $\{\sigma_i^{2,3}\}_{i=0}^{t_2-i_1}$, which we now define.
\begin{enumerate}
    \item Extend $\{\sigma_i\}_{i=i_1}^{t_2}$ with respect to $\sigma^{1,3}_{i_1-t_1}$, yielding $\{\sigma_i^{2,1}\}_{i=0}^{t_2-i_1}$.
    
    \item Let $\{\sigma_i^{2,2}\}_{i=0}^{z_2}$, with $\sigma_0^{2,2} = \sigma_{t_2-i_1}^{2,1}$, be the result of performing a $\kappa_a^{L+1}$-rotation with $\sigma_s(v_b^{L+1})$ to move $\sigma_s(v_b^{L+1})$ to $v_a^{L+1}$, then swapping $\sigma_s(v_b^{L+1})$ as far left as possible across $V(\mathcal P_a^{L+1})$, then performing a $\kappa_a^{L+1}$-rotation to swap the resulting vertex $\mu$ upon $v_a^{L+1}$ onto $v^{L+1}$.
    
    \item Extend $\{\sigma_i\}_{i=t_2}^{i_2}$ with respect to $\sigma_{t_2-i_1}^{2,1}$, yielding $\{\sigma_i^{2,3}\}_{i=0}^{i_2-t_2}$.
\end{enumerate}
It is now straightforward to see how to similarly construct the sequences $\mathscr{S}_1, \dots, \mathscr{S}_{31}$, and why we took $\eta = 31$ when appealing to the induction hypothesis: each sequence corresponding to a different vertex in $V(\mathcal K^{L+1})$ lying on $v^{L+1}$, and following $\mathscr{S}_1$, we alternate the path that we ``push" this vertex through. The only modification arises when we construct $\mathscr{S}_{31}$: during the $\kappa_b^L$-rotation with $\kappa_b^{L+1}$ within $\{\sigma_i\}_{i=i_{\eta-1}}^{i_\eta}$, simply include a $\kappa_b^{L+1}$-rotation which moves the vertex upon $v^{L+1}$ onto $v_b^{L+1}$. Merging $\mathscr{S}_1, \dots, \mathscr{S}_{31}$ yields a sequence $\{\sigma'_i\}_{i=0}^{\lambda'}$, such that $\sigma'_{\lambda'}$ is an $(L+1)$-extraction of $\sigma'_0 = \sigma_s$, since 
\begin{align*}
    & \sigma'_{\lambda'}(V(\mathcal P_b^{L+1})) = \sigma_s(V(\mathcal P_a^{L+1}) \setminus \{v_a^{L+1}\}) = V(\mathcal K_a^{L+1}), \\
    & \sigma'_{\lambda'}(V(\mathcal P_a^{L+1}) \setminus \{v_a^{L+1}\}) = \sigma_s(V(\mathcal P_b^{L+1})) = V(\mathcal K_b^{L+1}).
\end{align*}
It is evident by tracing the above construction that for every $\mu \in V(\mathcal K^{L+1})$, there exists a $\kappa_a^{L+1}$-rotation with $\mu$ and $\kappa_b^{L+1}$-rotation with $\mu$ that is a contiguous subsequence of $\{\sigma'_i\}_{i=0}^{\lambda'}$. Thus, the swap sequence $\{\sigma'_i\}_{i=0}^{\lambda'}$ establishes that Proposition \ref{prop:L_knob} holds for $L+1$ on $\eta = 1$. 

\begin{figure}[ht]
    \centering
    \includegraphics[width=\textwidth]{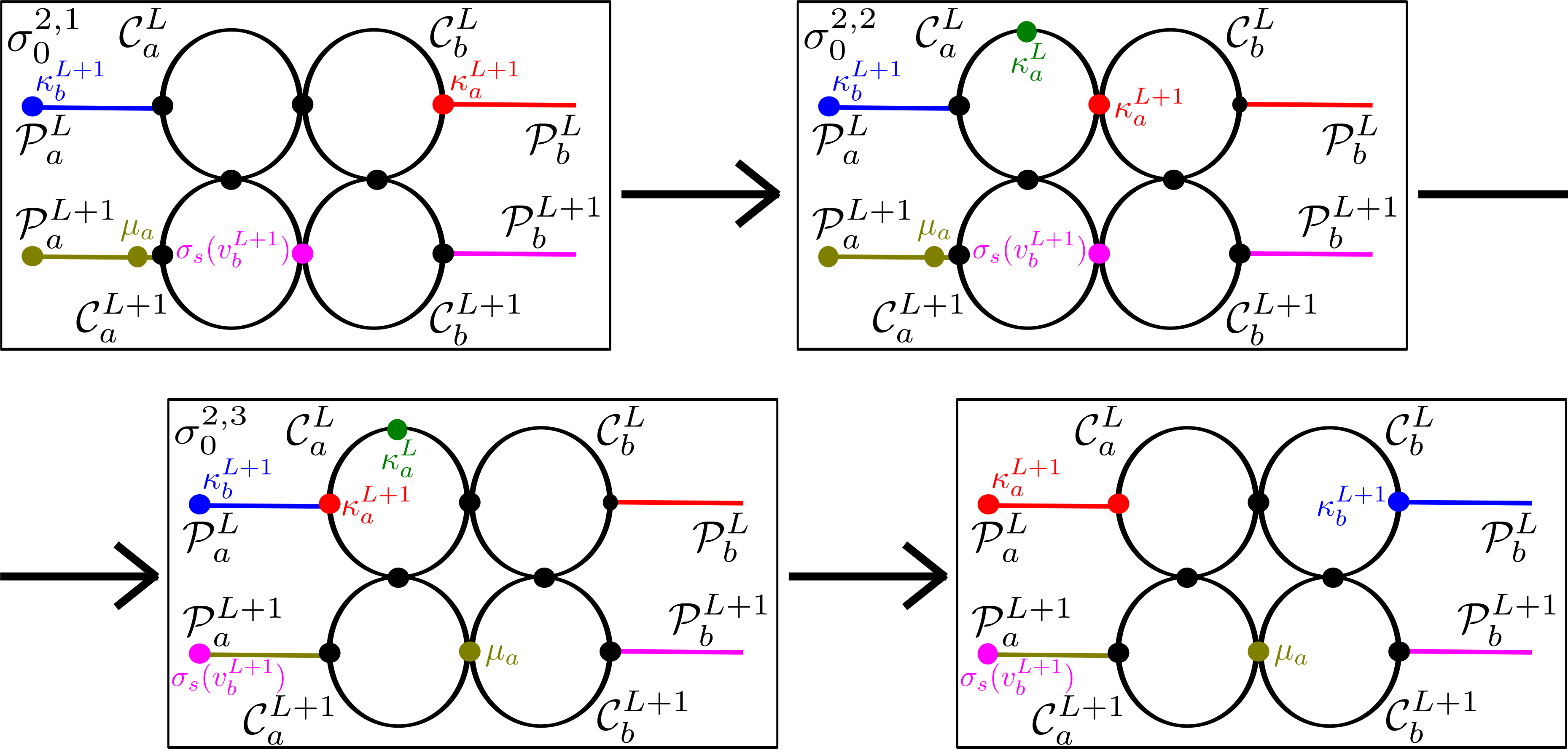}
    \caption{An illustration of the sequences of swaps defined during the construction of $\{\sigma_i'\}_{i=0}^{\lambda'}$ in the induction step, on $L+1$ layers. Subgraphs/vertices corresponding to preimages of $V(\mathcal K_a^L)$, $V(\mathcal K_b^L)$, $V(\mathcal K_a^{L+1})$, and $V(\mathcal K_b^{L+1})$ are red, blue, gold, and pink, respectively. We specifically depict the construction of the sequence $\mathscr{S}_2$, constructed from the subsequence $\{\sigma_i\}_{i=i_1}^{i_2}$ of the original sequence $\{\sigma_i\}_{i=0}^\lambda$. Initially, for $\sigma_0^{2,1}$, we have $\sigma_s(v_b^{L+1})$ upon $v^{L+1}$. At $\sigma_0^{2,2}$, we extend a $\kappa_a^L$-rotation with $\kappa_a^{L+1}$ in $\{\sigma_i\}_{i=i_1}^{i_2}$ (which is guaranteed to exist by the induction hypothesis) so that it includes the following sequence of swaps: a $\kappa_a^{L+1}$-rotation with $\sigma_s(v_b^{L+1})$, swapping $\sigma_s(v_b^{L+1})$ left into $\mathcal P_a^{L+1}$, and a $\kappa_a^{L+1}$-rotation with the resulting $\mu_a$ on $v_a^{L+1}$. This will result in the configuration $\sigma_0^{2,3}$. From $\sigma_0^{2,3}$ to the final configuration in $\mathscr{S}_2$, we execute the rest of $\{\sigma_i\}_{i=i_1}^{i_2}$, leading to an $L$-extraction of $\sigma_0^{2,1}$. Exhausting the original sequence $\{\sigma_i\}_{i=0}^\lambda$ by proceeding like this will yield an $(L+1)$-extraction of $\sigma_s$, and the resulting swap sequence satisfies Proposition \ref{prop:L_knob}(2).}
    \label{fig:general_L}
\end{figure}

For general $\eta \geq 1$, we can invoke the induction hypothesis, applied to $31\eta$, to extract a swap sequence $\{\sigma_i\}_{i=0}^\lambda$ in $V(\mathscr{C}(X_L, Y_L))$ with subsequence $\{\sigma_{i_j}\}_{j=0}^{31\eta}$. Then we can proceed as in the $\eta = 1$ case for every contiguous subsequence $\{\sigma_k\}_{k=31(i-1)}^{31i}$ in $\{\sigma_i\}_{i=0}^\lambda$, for $i \in [\eta]$, to construct a swap sequence $\{\sigma'_i\}_{i=0}^{\lambda'} \subset V(\mathscr{C})$ which establishes Proposition \ref{prop:L_knob} for this value of $\eta$.
\end{proof}

In the proof of Proposition \ref{prop:L_knob}, during the induction basis we reached a $1$-extraction of $\sigma_s$ by performing $\Omega(n)$ iterations of an algorithm which executed $\Omega(n^2)$ swaps.\footnote{Note that we inducted on $L$ in the proof of Proposition \ref{prop:L_knob}, so the sequence of swaps we found for smaller values of $L$ would be executed on subgraphs of $(X_L, Y_L)$ for larger values of $L$. However, it is easy to verify, by tracing the construction in Subsection \ref{subsec:construction}, that the asymptotic statements here hold regardless of the fixed value of $L$ that we choose.} Then in the inductive step, we reached an $(\ell+1)$-extraction of $\sigma_s$ by taking $\Omega(n)$ $\ell$-extractions of $\sigma_s$ and stringing them together by appending some other swap sequences. Altogether, it follows that we found a sequence of $\Omega(n^{L+2})$ swaps to reach an $L$-extraction of $\sigma_s$ --- if this were tight, taking $L$ to be as large as desired would be enough to answer Question \ref{ques:poly_bdd} in the negative. Motivated by these ideas, we prove Proposition \ref{prop:next_layer_ext}, which will lend itself to a lower bound on $d(\sigma_s,\sigma_f)$.

\begin{proposition} \label{prop:next_layer_ext}
Fix integers $L \geq 2$ and $\ell \in [L-1]$, and take $\sigma, \tau \in V(\mathscr{C})$ such that $\tau$ is an $(\ell+1)$-extraction of $\sigma$. Any swap sequence $\{\sigma_i\}_{i=0}^\lambda$ with $\sigma_0 = \sigma$ and $\sigma_\lambda = \tau$ must have a subsequence $\{\sigma_{i_j}\}_{j=0}^{25}$ such that, for $j \in [25]$, there exists a configuration $\Tilde{\sigma} \in \{\sigma_i\}_{i=i_{j-1}}^{i_j}$ that is an $\ell$-extraction of $\sigma_{i_{j-1}}$.
\end{proposition}

\begin{proof}
Assume $\tau$ is an $(\ell+1)$-extraction of $\sigma$ of the kind of Definition \ref{defn:extraction}(1). Proposition \ref{prop:next_layer_ext} can be proved in the setting where $\tau$ is an $(\ell+1)$-extraction of $\sigma$ of the kind of Definition \ref{defn:extraction}(2) entirely analogously, where we switch the roles of several expressions corresponding to the ``left and right sides" of the subgraphs $X^{\ell+1}$ and $\mathcal K^{\ell+1}$ in this case.

We will say that $X_a^{\ell+1}$ is the \textit{\textcolor{red}{initial subgraph}} of any vertex $\mu \in V(\mathcal K_a^{\ell+1}) \setminus \{\kappa_a^{\ell+2}\}$ ($\mu \in V(\mathcal K_a^{L})$ if $\ell = L-1$) with $\sigma^{-1}(\mu) \in V(X_a^{\ell+1})$. By Proposition \ref{prop:layer_independence}(3), $\mu$ leaving the initial subgraph corresponds to an $(X_L, Y_L)$-friendly swap where $\mu$ is upon $v^{\ell+1}$ and swaps onto some vertex in $N_{X_L}(v^{\ell+1}) \cap V(X_b^{\ell+1})$. Similarly, $X_b^{\ell+1}$ is the \textit{\textcolor{red}{initial subgraph}} for $\mu \in V(\mathcal K_b^{\ell+1}) \setminus \{\kappa_b^{\ell+2}\}$ ($\mu \in V(\mathcal K_b^{L})$ if $\ell = L-1$) with $\sigma^{-1}(\mu) \in V(X_b^{\ell+1})$. By Proposition \ref{prop:layer_independence}(3), $\mu$ leaving the initial subgraph corresponds to an $(X_L, Y_L)$-friendly swap where $\mu$ is upon $v^{\ell+1}$ and swaps onto some vertex in $N_{X_L}(v^{\ell+1}) \cap V(X_a^{\ell+1})$.

Let $\Sigma = \{\sigma_i\}_{i=0}^\lambda$ be a swap sequence with $\sigma_0 = \sigma$ and $\sigma_\lambda = \tau$. It is straightforward to show from Proposition \ref{prop:layer_independence}(4) and Definition \ref{defn:extraction}(1) that at least $26$ vertices in $V(\mathcal K^{\ell+1}) \setminus \{\kappa_a^{\ell+2}, \kappa_b^{\ell+2}\}$ ($V(\mathcal K^{L})$ for $\ell = L-1$) switch to the ``opposite" layer $\ell+1$ subgraph in $X_L$ over the course of $\Sigma$. Take any $26$ such vertices $\{\mu_1, \dots, \mu_{26}\}$, indexed in the order that they first leave their initial subgraph during $\Sigma$ (it is clear that at most one such vertex can leave their initial subgraph over a given swap). Construct a subsequence $\{\sigma_{i_j}\}_{j=1}^{26}$ of $\Sigma$ such that, for every $j \in [26]$, $i_j$ is the smallest index for which 
\begin{itemize}
    \item $\sigma_{i_j}^{-1}(\mu_j) = v^{\ell+1}$;
    \item $\sigma_{i_j+1}^{-1}(\mu_j)$ is not a vertex in the initial subgraph of $\mu_j$.
\end{itemize}
Consider any $j \in [26]$ for which $\mu_j \in V(\mathcal K_a^{\ell+1})$. The neighborhood of $\mu_j$ is
\begin{align*}
    N_{Y_L}(\mu_j) = V(\mathcal K_b^{\ell+1}) \cup \{\kappa_a^{\ell+1}, \kappa_b^{\ell+1}\}.
\end{align*}
The vertex $\sigma_{i_j+1}(v^{\ell+1})$ that $\mu_j$ swaps with to reach $\sigma_{i_j+1}$ from $\sigma_{i_j}$ satisfies 
\begin{align*}
    \sigma_{i_j+1}(v^{\ell+1}) \in \{\kappa_a^{\ell+1}, \kappa_b^{\ell+1}\},
\end{align*}
since $\sigma_{i_j}$ would violate Proposition \ref{prop:layer_independence}(4) (on layer $\ell+1$) if we had that $\sigma_{i_j+1}(v^{\ell+1}) \in V(\mathcal K_b^{\ell+1})$. Assume (towards a contradiction) that $\sigma_{i_j+1}(v^{\ell+1}) = \kappa_a^{\ell+1}$, and let $1 \leq \xi \leq i_j$ (the lower bound is since $\sigma_s^{-1}(\mu_j) \neq v^{\ell+1}$) be the smallest such index satisfying 
\begin{align} \label{eq:next_layer_ext_eq_2}
    \sigma_\xi^{-1}(\mu_j) = \sigma_{i_j}^{-1}(\mu_j) = v^{\ell+1} \text{ and } \sigma_\xi^{-1}(\kappa_a^{\ell+1}) = \sigma_{i_j}^{-1}(\kappa_a^{\ell+1}) \in N_{X_L}(v^\ell) \cap V(X_b^{\ell+1}).
\end{align}
Exactly one of the two statements
\begin{itemize}
    \item $\sigma_{\xi-1}^{-1}(\kappa_a^{\ell+1}) \neq \sigma_\xi^{-1}(\kappa_a^{\ell+1})$;
    \item $\sigma_{\xi-1}^{-1}(\mu_j) \neq \sigma_\xi^{-1}(\mu_j)$
\end{itemize}
is true; both being false would contradict $\xi$ being the smallest possible, while both being true would contradict $i_j$ being the smallest possible. But $\sigma_{\xi-1}^{-1}(\kappa_a^{\ell+1}) \neq \sigma_\xi^{-1}(\kappa_a^{\ell+1})$ would imply that $\sigma_{\xi-1}$ violates Proposition \ref{prop:rule_of_two} (on $\mathcal C_b^{\ell+1}$), and $\sigma_{\xi-1}^{-1}(\mu_j) \neq \sigma_\xi^{-1}(\mu_j)$ would imply that $\sigma_{\xi-1}$ violates Proposition \ref{prop:layer_independence}(4) (on layer $\ell$ if it swaps with $\kappa_b^{\ell+1}$, and on layer $\ell+1$ if it swaps with a vertex in $V(\mathcal K_b^{\ell+1})$). So for all $j \in [26]$,
\begin{align} \label{eq:next_layer_ext_eq_3}
    \mu_j \in V(\mathcal K_a^{\ell+1}) \implies \sigma_{i_j+1}(v^{\ell+1}) = \kappa_b^{\ell+1} \text{ and } \mu_j \in V(\mathcal K_b^{\ell+1}) \implies \sigma_{i_j+1}(v^{\ell+1}) = \kappa_a^{\ell+1},
\end{align}
where the latter claim can be deduced from an entirely analogous argument. See Figure \ref{fig:next_layer_ext} for an illustration.

\begin{figure}[ht]
    \centering
    \subfloat[Case where $\mu_j \in V(\mathcal K_a^{\ell+1})$.]{\includegraphics[width=0.49\textwidth]{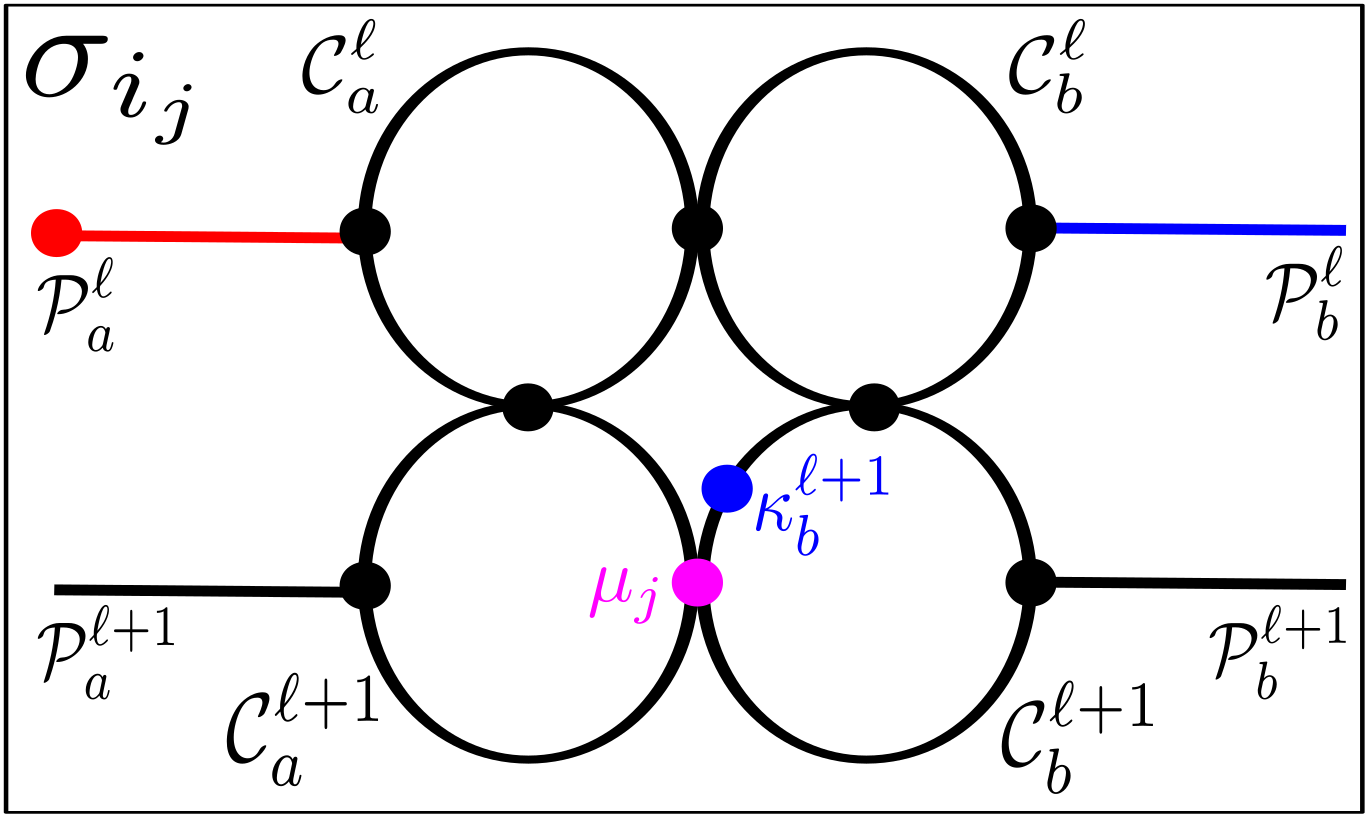}\label{fig:next_layer_ext_a}}
    \hfill
    \subfloat[Case where $\mu_j \in V(\mathcal K_b^{\ell+1})$.]{\includegraphics[width=0.49\textwidth]{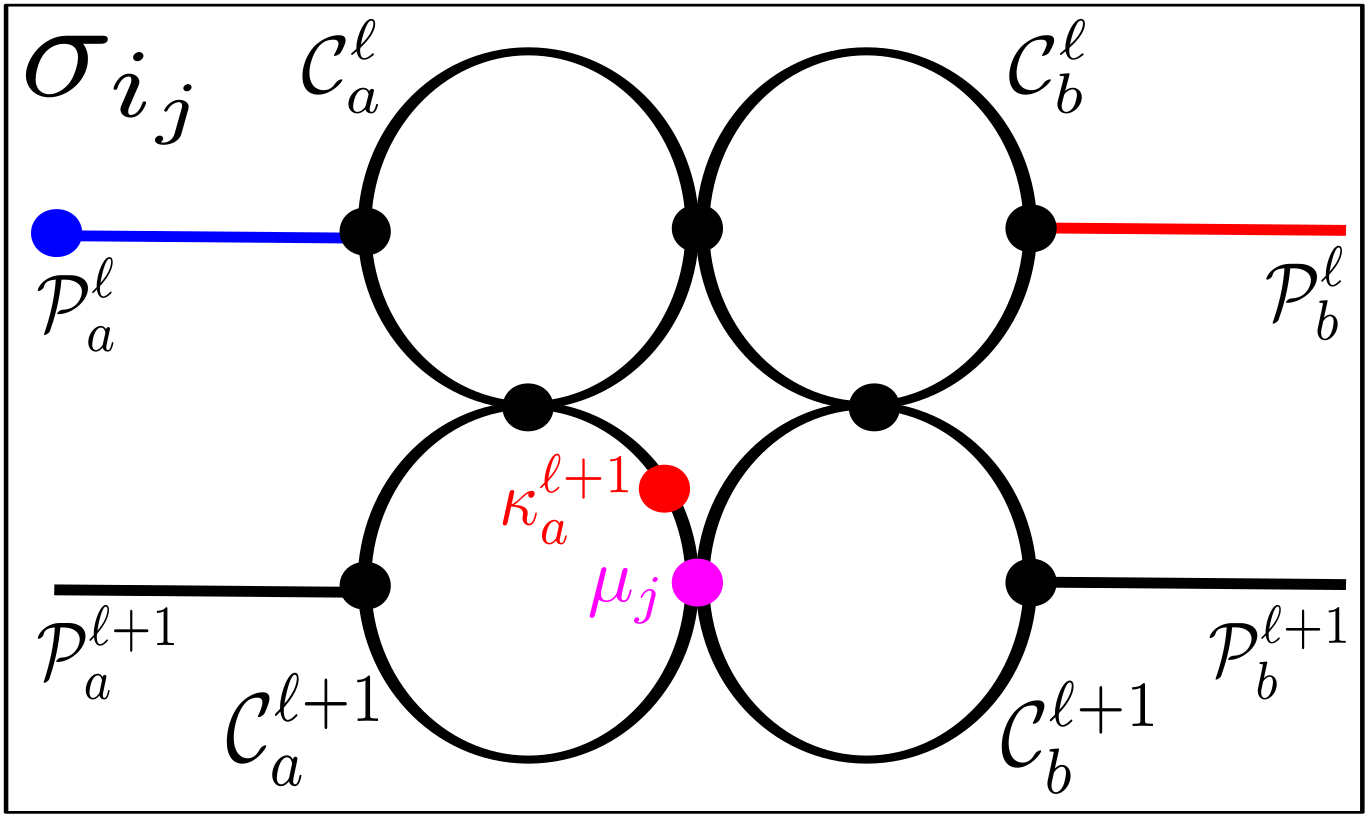}\label{fig:next_layer_ext_b}}
    \caption{The two possibilities for the configuration $\sigma_{i_j}$ for any $j \in [26]$. Subgraphs/vertices corresponding to preimages of $V(\mathcal K_a^\ell)$ and $V(\mathcal K_b^\ell)$ are colored red and blue, respectively. The coloring of $\mathcal P_a^\ell$ in both cases follows from Proposition \ref{prop:knob_extract}.}
    \label{fig:next_layer_ext}
\end{figure}

Now consider $2 \leq j \leq 26$ for which $\mu_j \in V(\mathcal K_a^{\ell+1})$. For such values of $j$ which have that $\mu_j \in V(\mathcal K_b^{\ell+1})$, establishing the existence of a configuration $\Tilde{\sigma} \in \{\sigma_i\}_{i=i_{j-1}}^{i_j}$ that is an $\ell$-extraction of $\sigma_{i_{j-1}}$ can be done entirely analogously. By \eqref{eq:next_layer_ext_eq_3}, $\sigma_{i_j+1}(v^{\ell+1}) = \kappa_b^{\ell+1}$, so certainly 
\begin{align*}
    \sigma_{i_j}^{-1}(\kappa_b^{\ell+1}) \notin V(\mathcal P_a^\ell) \cup V(\mathcal P_b^\ell),
\end{align*}
and since $\kappa_b^{\ell+1} \in V(\mathcal K^\ell)$,
\begin{align*}
    \sigma_{i_j}^{-1}(V(\mathcal K^\ell)) \not\subset V(\mathcal P_a^\ell) \cup V(\mathcal P_b^\ell).
\end{align*}
By Proposition \ref{prop:knob_extract}(2), 
\begin{align} \label{eq:next_layer_ext_eq_4}
    V(\mathcal K_a^\ell) \subset \sigma_{i_j}(V(\mathcal P_a^\ell)).
\end{align}
If it were true that $V(\mathcal K_b^\ell) \cap \sigma_{i_j}(V(\mathcal P_a^\ell)) \neq \emptyset$, there would exist $\eta_1 \in V(\mathcal K_b^\ell)$ satisfying $\eta_1 \in \sigma_{i_j}(V(\mathcal P_a^\ell))$. Combined with \eqref{eq:next_layer_ext_eq_4}, we would have
\begin{align} \label{eq:next_layer_ext_eq_5}
    V(\mathcal K_a^\ell) \cup \{\eta_1\} = \sigma_{i_j}(V(\mathcal P_a^\ell)),
\end{align}
since the LHS is a subset of the RHS and their cardinalities are equal. In particular, it must be that $\sigma_{i_j}(v_a^\ell) \in V(\mathcal K^\ell)$, so Proposition \ref{prop:path_images}(2) would imply that 
\begin{align*}
    \sigma_{i_j}^{-1}(\{\kappa_a^\ell, \kappa_b^\ell\}) \cap (V(\mathcal P_a^\ell) \setminus \{v_a^\ell\}) \neq \emptyset,
\end{align*}
so there exists $\eta_2 \in \{\kappa_a^\ell, \kappa_b^\ell\}$ that is in $\sigma_{i_j}(V(\mathcal P_a^\ell))$, contradicting \eqref{eq:next_layer_ext_eq_5}. So it must be that $\sigma_{i_j}^{-1}(V(\mathcal K_b^\ell)) \cap V(\mathcal P_a^\ell) = \emptyset$, i.e., that
\begin{align} \label{eq:next_layer_ext_eq_6}
    V(\mathcal K_b^\ell) \cap \sigma_{i_j}(V(\mathcal P_a^\ell)) = \emptyset.
\end{align}
If $\mu_{j-1} \in V(\mathcal K_b^{\ell+1})$, then \eqref{eq:next_layer_ext_eq_3} implies $\sigma_{i_{j-1}+1}(v^{\ell+1}) = \kappa_a^{\ell+1}$, and so 
\begin{align*}
    \sigma_{i_{j-1}}^{-1}(\kappa_a^{\ell+1}) \notin V(\mathcal P_a^\ell) \cup V(\mathcal P_b^\ell).
\end{align*}
Proposition \ref{prop:knob_extract}(1) thus implies 
\begin{align*}
    V(\mathcal K_b^\ell) \subset \sigma_{i_{j-1}}(V(\mathcal P_a^\ell)).
\end{align*}
This statement, with \eqref{eq:next_layer_ext_eq_4} and \eqref{eq:next_layer_ext_eq_6}, implies that $\sigma_{i_j}$ is an $\ell$-extraction of $\sigma_{i_{j-1}}$, namely of the Definition \ref{defn:extraction}(2) kind. If $\mu_{j-1} \in V(\mathcal K_a^{\ell+1})$, then \eqref{eq:next_layer_ext_eq_3} implies $\sigma_{i_{j-1}+1}(v^{\ell+1}) = \kappa_b^{\ell+1}$, so 
\begin{align*}
    \sigma_{i_{j-1}}^{-1}(\kappa_b^{\ell+1}) \notin V(\mathcal P_a^\ell) \cup V(\mathcal P_b^\ell).
\end{align*}
Proposition \ref{prop:knob_extract}(2) now implies that 
\begin{align*}
    V(\mathcal K_a^\ell) \subset \sigma_{i_{j-1}}(V(\mathcal P_a^\ell)).
\end{align*}
Since $i_j$ is the smallest possible, $\sigma_{i_{j-1}}^{-1}(\mu_j) \in V(X_a^{\ell+1})$, and it follows that $\sigma_{i_{j-1}}^{-1}(\mu_j) \in V(\mathcal P_a^{\ell+1})$, as $\sigma_{i_{j-1}}$ would otherwise violate Proposition \ref{prop:layer_independence}(4) on layer $\ell$ (due to $\kappa_b^{\ell+1}$ and $\mu_j$). It is straightforward to confirm, appealing to Proposition \ref{prop:rule_of_two} on $\mathcal C_a^{\ell+1}$, that $\mu_j$ moves to $v^{\ell+1}$ during $\{\sigma_i\}_{i=i_{j-1}}^{i_j}$, and swaps with $\kappa_a^{\ell+1}$ upon $V(\mathcal C_a^{\ell+1})$ at some point in this swap sequence.\footnote{This can be proved using ideas and arguments which are essentially identical to those that were carried out in Subcase 1.2 of the proof of Proposition \ref{prop:layer_independence}.} Thus, there exists a configuration $\Tilde{\sigma} \in \{\sigma_i\}_{i=i_{j-1}}^{i_j}$ for which 
\begin{align*}
    \Tilde{\sigma}^{-1}(\{\mu_j, \kappa_a^{\ell+1}\}) \subset V(\mathcal C_a^{\ell+1}) \text{ and } \Tilde{\sigma}^{-1}(\kappa_a^{\ell+1}) \neq v_a^\ell,
\end{align*}
from which it immediately follows that
\begin{align*}
    \Tilde{\sigma}^{-1}(\kappa_a^{\ell+1}) \notin V(\mathcal P_a^\ell) \cup V(\mathcal P_b^\ell),
\end{align*}
and Proposition \ref{prop:knob_extract}(1) implies
\begin{align*}
    V(\mathcal K_b^\ell) \subset \Tilde{\sigma}(V(\mathcal P_a^\ell)).
\end{align*}
This, with \eqref{eq:next_layer_ext_eq_4} and \eqref{eq:next_layer_ext_eq_6}, implies that $\Tilde{\sigma}$ is an $\ell$-extraction of $\sigma_{i_{j-1}}$, namely of the Definition \ref{defn:extraction}(2) kind.

Therefore, taking $\{\sigma_{i_j}\}_{j=1}^{26}$ yields the desired subsequence of $\{\sigma_i\}_{i=0}^\lambda$.
\end{proof}

\subsection{Proof of Theorem \ref{thm:main_diam_result}} \label{subsec:proof_of_lower_bound}

We finally derive the desired lower bound on the diameter of $\mathscr{C}$.

\mainDiamResult*

\begin{proof}
For $L \geq 2$, take $X_L$, $Y_L$ on $58L+2$ vertices (see Subsection \ref{subsec:construction}). For $\ell \in [L-1]$, define 
\begin{align} \label{eq:diam_lb_eq_1}
    \lambda_{(L, n)}(\ell) := \min\left\{d(\sigma, \tau): \sigma, \tau \in V(\mathscr{C}), \tau \text{ is an } \ell\text{-extraction of } \sigma\right\}.
\end{align}
It follows from Proposition \ref{prop:next_layer_ext} that for all $\ell \in [L-1]$,
\begin{align} \label{eq:diam_lb_eq_2}
    \lambda_{(L, n)}(\ell+1) \geq 25\lambda_{(L, n)}(\ell).
\end{align}
Let $\sigma_f \in V(\mathscr{C})$ be such that $\sigma_f$ is an $L$-extraction of $\sigma_s$, which exists by Proposition \ref{prop:L_knob}. By \eqref{eq:diam_lb_eq_1} and \eqref{eq:diam_lb_eq_2},
\begin{align*}
    d(\sigma_s, \sigma_f) \geq \lambda_{(L, n)}(L) \geq 25\lambda_{(L, n)}(L-1) \geq \dots \geq 25^{L-1}\lambda_{(L, n)}(1) \geq 25^{L-1}.
\end{align*}
Now, for $n \geq 60$, fix $L = \lfloor (n-2)/58 \rfloor$ (here, $L \geq 1$), and construct $n$-vertex graphs $\Tilde{X}_n$, $\Tilde{Y}_n$ by adding $n' = n-(58L+2)$ isolated vertices to $X_L$ and $Y_L$, respectively. Let $\mathscr{C}(\Tilde{X}_n$, $\Tilde{Y}_n)$ denote the connected component of $\FS(\Tilde{X}_n$, $\Tilde{Y}_n)$ containing the configuration resulting from placing $V(\Tilde{Y}_n)$ upon $V(\Tilde{X}_n)$ as usual (i.e., under the starting configuration as defined in Subsection \ref{subsec:construction}), and then placing the $n-n'$ isolated vertices in $\Tilde{Y}_n$ upon the $n-n'$ isolated vertices of $\Tilde{X}_n$ in some way. It easily follows from our construction that $58L+2 \leq n \leq 58L+58$, so 
\begin{align*}
    d(\sigma_s, \sigma_f) \geq 25^{L-1} = e^{\Omega(n)}.
\end{align*}
By accounting for the values $2 \leq n \leq 59$, which may weaken the constant implicit in the $\Omega(n)$ term, the desired result now follows immediately.
\end{proof}

To conclude Section \ref{sec:large_diameter}, we mention an especially notable implication of Theorem \ref{thm:main_diam_result} in the study of random walks on friends-and-strangers graphs. In the proceeding discussion, to avoid distracting from the nature of this article, we elect to be terse and do not define many of the objects we consider; we refer the reader to \cite{lovasz1993random} for a thorough treatment of random walks on graphs. We begin by providing the following Definition \ref{defn:fs_markov_chain}. The fact that there is a natural discrete-time Markov chain associated to a friends-and-strangers graph was observed in passing in \cite[Section 7]{alon2022typical}, and an investigation of its mixing properties was separately proposed in \cite{alon2021talk}.

\begin{definition} \label{defn:fs_markov_chain}
    Let $X$ and $Y$ be $n$-vertex graphs. The \textit{\textcolor{red}{friends-and-strangers Markov chain}} of $X$ and $Y$ is the discrete-time Markov chain whose state space is $V(\FS(X,Y))$ and such that at each time step, a pair of friends standing upon adjacent vertices is chosen uniformly at random amongst all such pairs and swap places with probability $1/2$.
\end{definition}

The friends-and-strangers Markov chain of $X$ and $K_n$, which models a lazy random walk on a connected component of $\FS(X, Y)$, is aperiodic by construction (recall from Proposition \ref{prop:basic_properties}(2) that friends-and-strangers graphs are bipartite, which warrants the laziness condition in Definition \ref{defn:fs_markov_chain} if we would like to discuss mixing to stationarity) and irreducible when restricted to a connected component of $\FS(X, Y)$. From a different perspective, Definition \ref{defn:fs_markov_chain} may be interpreted as the generalization of a natural discrete variant of the interchange process\footnote{The interchange process is usually posed as a continuous-time stochastic process by assigning, to the edges of $X$, independent point processes on the positive half-line, and transposing the particles upon the vertices incident to a given edge at the points of its corresponding process. One can certainly adapt Definition \ref{defn:fs_markov_chain} to accommodate for such differences in the presentation of the model.} (sometimes called the random stirring process) where we include the condition that arbitrary pairs of particles may be forbidden from swapping positions. The friends-and-strangers Markov chain has received substantial attention under certain restricted settings (chiefly that in which $Y = K_n$), and classical polynomial upper bounds on the mixing time (in total variation distance) of the underlying Markov chain are known; see \cite{aldous1983random, aldous1986shuffling, diaconis1981generating, diaconis1993comparison, jonasson2012mixing, matthews1988strong, wilson2004mixing}. An immediate corollary of Theorem \ref{thm:main_diam_result}, which might be thought of as its natural stochastic analogue (especially in light of the polynomial upper bounds that we derived in Section \ref{sec:fixed_one_graph}), is the following.

\begin{corollary}
    For all $n \geq 2$, there exist $n$-vertex graphs $X$ and $Y$ for which there exists a connected component of $\FS(X,Y)$ such that the friends-and-strangers Markov chain of $X$ and $Y$, when restricted to this component, has mixing time (in total variation distance) which is $e^{\Omega(n)}$.
\end{corollary}

In other words, for this variant of the interchange process in which we may further forbid certain pairs of particles from swapping places with each other, it is possible to fix restrictions between particles in such a way that the mixing time of the underlying Markov chain is exponential in the size of the graph on which the process occurs. This is in stark contrast to the aforementioned polynomial upper bounds regarding rapidly mixing Markov chains. 

\section{Open Questions and Future Directions} \label{sec:future_directions}

Theorem \ref{thm:main_diam_result} of this paper proves that diameters of connected components of friends-and-strangers graphs may grow exponentially in the size of their input graphs. There are many other interesting questions concerning distance and diameter that remain unresolved by this article.

\subsection{Other Choices of Fixed Graphs}

In Section \ref{sec:fixed_one_graph}, we fixed $X$ to be from a particular class of graphs, and derived bounds on the maximal diameter of a connected component $\FS(X,Y)$. Of course, we could pursue similar inquiries for other choices of $X$. One natural choice would be to take $X = \Star_n$. It is known (see \cite{biniaz2023token}) that the diameter of any component of $\FS(\Star_n, K_n)$ is at most $\frac{3}{2}n + O(1)$, but to our knowledge, there are no known bounds on the maximum diameter for general $Y$. We also remark that it may be possible to extract a bound on the maximum diameter of a component of $\FS(\Star_n, Y)$ for biconnected graphs $Y$ by tracing the arguments in \cite{wilson1974graph}.

\subsection{Improvements} \label{subsec:improvements}

For much of our discussion in Subsection \ref{subsec:cycle_graphs}, we were primarily interested in showing that the maximum diameter of a connected component of $\FS(\Cycle_n,Y)$ was polynomially bounded (in the sense of Question \ref{ques:poly_bdd}), rather than achieving tight asymptotic statements. It would be desirable to improve these results, toward which we pose the following conjectures. We mention that generalizing the lattice-theoretic methods of \cite{propp2021lattice}, which are of a very different flavor than the arguments presented here, might lead to the resolution of Conjecture \ref{conj:double_flip_quadratic}. We also note that if Conjecture \ref{conj:cycle_quadratic} were settled, then tracing the proof of Corollary \ref{cor:double_flip_dist} would immediately lead to an $O(n^3)$ bound on the number of double-flips needed to go between double-flip equivalent acyclic orientations $\alpha$ and $\alpha''$ on an $n$-vertex graph, sharpening Corollary \ref{cor:double_flip_dist}.

\begin{conjecture} \label{conj:cycle_quadratic}
    The maximum diameter of a connected component of $\FS(\Cycle_n, Y)$ is $O(n^2)$.
\end{conjecture}

\begin{conjecture} \label{conj:double_flip_quadratic}
    For an $n$-vertex graph $G$ and two acyclic orientations $\alpha, \alpha'' \in \Acyc(G)$ that are double-flip equivalent, it is possible to go from $\alpha$ to $\alpha''$ in $O(n^2)$ double-flips.
\end{conjecture}

In another direction, Theorem \ref{thm:main_diam_result} states that for all $n \geq 2$, there exist $n$-vertex $X$ and $Y$ such that the maximum diameter of a connected component of $\FS(X,Y)$ is $e^{\Omega(n)}$. It is unclear how close this is to the truth. As a first step, we pose the following problem.
\begin{question} \label{ques:upper_bound}
    For $n$-vertex graphs $X$ and $Y$, does there exist a nontrivial upper bound (in terms of $n$) on the maximum diameter of a component of $\FS(X,Y)$?
\end{question}

We briefly clarify what we mean by a nontrivial upper bound in Question \ref{ques:upper_bound}. Let $\mathcal D(n)$ denote the maximum possible diameter of a connected component of $\FS(X,Y)$ when $X$ and $Y$ are $n$-vertex graphs. By Theorem \ref{thm:main_diam_result} for the lower bound and Stirling's approximation applied to $n!$ for the upper bound, we observe that
\begin{align*}
    e^{\Omega(n)} = \mathcal D(n) \leq e^{\left(1-o(1)\right)n \log n}.
\end{align*}
Thus, our understanding of $\mathcal D(n)$ is tight up to a logarithmic factor in the exponent. Any improvement over this naive upper bound, or confirmation that this upper bound is essentially the truth, would be highly desirable. In particular, we ask the following more precise question. Indeed, given the preceding discussion, Question \ref{ques:O(n)_diam_upper_bound} is the natural next target.
\begin{question} \label{ques:O(n)_diam_upper_bound}
    Is it true that $\mathcal D(n) = e^{O(n)}$?
\end{question}
We propose one final problem in this subsection which we would especially like to see resolved.
\begin{problem}
    Find a shorter (perhaps via non-constructive\footnote{Certainly, a proof of Theorem \ref{thm:main_diam_result} using non-constructive techniques would be a novel contribution. In another direction, recall that the central idea behind Section \ref{sec:large_diameter} was to construct an exponentially increasing recursive sequence of swaps in such a way that executing this sequence of swaps is necessary in order to reach one configuration from another. A constructive proof which either proceeds via a similar paradigm with a construction that is more amenable to analysis or leverages different ideas altogether would also be of interest.} means) proof of Theorem \ref{thm:main_diam_result}.
\end{problem}

\subsection{Connected Friends-and-Strangers Graphs}

The proof of Theorem \ref{thm:main_diam_result} relied heavily on characterizing all vertices of $\FS(X_L, Y_L)$ in the same connected component of $\sigma_s$. It is thus natural to ask Question \ref{ques:poly_bdd} in the setting where $\FS(X, Y)$ is assumed to be connected, which was separately raised by Defant and Kravitz.

\begin{question}[{\cite[Subsection 7.3]{defant2021friends}}] \label{ques:diam_q_conn} 
    Does there exist an absolute constant $C > 0$ such that for all $n$-vertex graphs $X$ and $Y$ with $\FS(X, Y)$ connected, it holds that $\diam(\FS(X, Y))$ is $O(n^C)$?
\end{question}

If Question \ref{ques:diam_q_conn} holds in the negative, then settling it will likely require very different techniques and paradigms than those which were developed in this article. Indeed, the proof of the negative result for Question \ref{ques:poly_bdd} relies heavily on ``rigging" the configurations that lie in a particular connected component of $\FS(X_L, Y_L)$, which allows us to argue that two particular configurations (namely, $\sigma_s$ and $\sigma_f$) are necessarily very far apart. Such a strategy is not applicable if we require $\FS(X, Y)$ to be connected. Additionally, by Proposition \ref{prop:basic_properties}(3), we can assume (without loss of generality) that $X$ is biconnected under this setting, and that either $X$ or $Y$ has no cut vertices. Theorem \ref{thm:cycle_diam_1} already gives a positive result for $\Cycle_n$, the ``simplest" biconnected graph (e.g., the $n$-vertex cycle has the smallest Betti number amongst all $n$-vertex biconnected graphs: see \cite[Theorem 19]{whitney1931non}, which might lend itself to an inductive argument) and for $K_n$, the most ``complicated" ($K_n$ has the largest Betti number amongst all $n$-vertex biconnected graphs). Furthermore, the constructions $X_L$ and $Y_L$ contain cut vertices which hold central roles in the proofs of the intermediate propositions (namely, vertices on the paths $\mathcal P_a^\ell, \mathcal P_b^\ell$ for $X_L$, and the knob vertices $\kappa_a^\ell, \kappa_b^\ell$ in $Y_L$).

In another direction, a negative answer to Question \ref{ques:diam_q_conn} implies the existence of long paths in the connected graph $\FS(X, Y)$. The following result shows that the extreme end of this is not possible.

\begin{proposition}
For $n \geq 4$, $\FS(X, Y)$ is not isomorphic to a tree on $n!$ vertices (e.g., $\Path_{n!}$) or a tree on $n!$ vertices with one edge appended (e.g., $\Cycle_{n!}$).
\end{proposition}

\begin{proof}
The number of edges of $\FS(X, Y)$ is $|E(X)| \cdot |E(Y)| \cdot (n-2)!$, while this is $n! - 1$ and $n!$ for a tree on $n!$ vertices and a tree with one edge appended on $n!$ vertices, respectively. Notice that $|E(X)| \cdot |E(Y)| \cdot (n-2)!$ is divisible by $2$ while $n! - 1$ is not, so $\FS(X,Y)$ cannot be isomorphic to a tree on $n!$ vertices. Assume $\FS(X, Y)$ is isomorphic to a tree with an edge appended to it, so $|E(X)| \cdot |E(Y)| \cdot (n-2)! = n!$, or $|E(X)| \cdot |E(Y)| = n(n-1)$. Then $X$ and $Y$ must both be connected, so that (without loss of generality) $|E(X)| = n$ and $|E(Y)| = n-1$, so $Y$ is a tree. Due to Proposition \ref{prop:basic_properties}(3), $X$ is biconnected, so necessarily $X = \Cycle_n$. But $|E(\overline{Y})| = \binom{n}{2} - (n-1)$, contradicting Theorem \ref{thm:cycle_connected}, which gives $|E(\overline{Y})| \leq n-1$.
\end{proof}

\subsection{Probabilistic Problems} \label{subsec:probabilistic_aspects}

In a different direction, we may study notions of distance in friends-and-strangers graphs when we take $X$ and $Y$ to be random graphs. We propose the following problem; we leave the meaning of ``small diameter" up to interpretation.
\begin{problem} \label{prob:small_diameter_whp}
   Let $X$ and $Y$ be independently-chosen random graphs from $\mathcal G(n,p)$. Find conditions on $p$ (in terms of $n$) which guarantee that every connected component of $\FS(X,Y)$ has small diameter with high probability.
\end{problem}
We also restate a problem of this kind proposed by \cite{alon2022typical}.
\begin{problem}[{\cite[Problem 7.9]{alon2022typical}}] \label{prob:expected_diam}
    Obtain estimates (in terms of $n$ and $p$) for the expectation of the maximum diameter of a connected component in $\FS(X,Y)$ when $X$ and $Y$ are independently-chosen random graphs from $\mathcal G(n,p)$.
\end{problem}
In a manner analogous to how we fixed one of the two graphs $X$ and $Y$ in Section \ref{sec:fixed_one_graph} and studied the resulting variant of Question \ref{ques:poly_bdd} before addressing the more global question, it may be insightful to first fix (without loss of generality) $X$ to be a particular kind of graph and study the variants of Problems \ref{prob:small_diameter_whp} and \ref{prob:expected_diam} which only take $Y$ to be drawn from $\mathcal G(n,p)$. The graphs we studied in Section \ref{sec:fixed_one_graph} (complete graphs, paths, and cycles) may also serve as natural starting points here.

\subsection{Complexity} \label{subsec:complexity}

As the literature on the token swapping problem suggests, computing exact distances between two configurations in $\FS(X,Y)$ and the maximum diameter of a component of $\FS(X,Y)$, under mild assumptions on $X$ and $Y$, seems to be intractable. We might thus study distances and diameters in friends-and-strangers graphs from the perspective of complexity theory. We discuss one possible direction of study along these lines here. We start by introducing a decision problem which encapsulates finding the shortest swap sequence between two configurations.

\begin{definition} \label{defn:distance_problem}
    In an instance of the \textit{\textcolor{red}{distance problem,}} we are given graphs $X$ and $Y$ on $n$ vertices, configurations $\sigma, \tau \in V(\FS(X,Y))$, and a positive integer $K$, and want to know if $d(\sigma, \tau) \leq K$.
\end{definition}

This problem has been studied in many restricted contexts. If we proceed under the assumption that $Y = K_n$, the distance problem is known to be PSPACE-complete \cite{jerrum1985complexity}, APX-hard \cite{miltzow_et_al:LIPIcs:2016:6408}, and $W[1]$-hard when parametrized by the shortest number of swaps \cite{bonnet2018complexity}. Furthermore, it is NP-hard when we impose certain additional restrictions, such as when we take $X$ to be a tree and $Y = K_n$ \cite{aichholzer2021hardness}. It might be fruitful to study the complexity of this problem at the level of generality proposed by Definition \ref{defn:distance_problem}.

A natural follow-up to Definition \ref{defn:distance_problem} is to ask for worst-case distances between two configurations, which corresponds to the maximum diameter of a component of $\FS(X,Y)$.

\begin{definition} \label{defn:diameter_problem}
    In an instance of the \textit{\textcolor{red}{diameter problem,}} we are given graphs $X$ and $Y$ on $n$ vertices and a positive integer $K$, and want to know if the maximum diameter of a component of $\FS(X,Y)$ is at most $K$.
\end{definition}

The literature suggests that the diameter problem has not been studied as thoroughly as the distance problem, even when assuming that one of the two graphs is complete. Towards bridging this gap, we pose the following primitive question and problem.

\begin{question}
    Is the diameter problem in EXPSPACE? If so, is it EXPSPACE-complete? What changes if we fix $Y = K_n$?
\end{question}

\begin{problem}
    Find assumptions on $X$, $Y$, and $K$ which guarantee that the diameter problem (under these assumptions) is in PSPACE.
\end{problem}

Notably, even simpler decision problems than those proposed in Definitions \ref{defn:distance_problem} and \ref{defn:diameter_problem} seem to be poorly understood. For instance, we may consider the following decision problems.

\begin{definition}
    In an instance of the \textit{\textcolor{red}{component problem}}, we are given $n$-vertex graphs $X$ and $Y$ and configurations $\sigma, \tau \in V(\FS(X,Y))$, and want to know if $\sigma$ and $\tau$ lie in the same connected component of $\FS(X,Y)$.
\end{definition}

\begin{definition}
    In an instance of the \textit{\textcolor{red}{connectivity problem}}, we are given $n$-vertex graphs $X$ and $Y$, and want to know if $\FS(X,Y)$ is connected.
\end{definition}

We may think of the component problem and the connectivity problem as, respectively, the simplest instances of the distance problem (asking whether $d(\sigma, \tau)$ is finite) and of the diameter problem (asking whether the diameter of $\FS(X,Y)$, when we do not restrict to connected components, is finite). To our knowledge, an understanding of the complexity of the component problem and the connectivity problem remains open, though the results in \cite{alaniz2023complexity} address these problems when studying friends-and-strangers graphs with multiplicities, as elaborated in \cite{milojevic2023connectivity}.

In a different direction, \cite{akers1989group, vaughan1995algorithm, yamanaka2015swapping} independently found $2$-approximation algorithms for determining the distance between two configurations in $\FS(X, K_n)$ when $X$ is a tree. Recall from the proof of Proposition \ref{prop:path_diam} that for any $\sigma, \tau \in \FS(\Path_n, Y)$ in the same connected component, $d(\sigma, \tau) = \inv(\sigma, \tau)$, and an algorithm which exactly computes the distance between any two configurations in $\FS(\Path_n, Y)$ is one which, starting from $\sigma$, reverses a $\tau$-inversion at every step. These two observations naturally suggest the following problem, which one can also pursue by replacing $\Cycle_n$ with a different fixed graph.

\begin{problem} \label{prob:cycle_algo}
    Find, under the most general assumptions on $Y$ possible, an $O(1)$-approximation algorithm for computing the distance between two configurations in $\FS(\Cycle_n, Y)$.
\end{problem}

\section*{Acknowledgments}

This research was initiated at the University of Minnesota Duluth REU and was supported, in part, by NSF-DMS grant 1949884 and NSA Grant H98230-20-1-0009. I would like to thank Professor Joseph Gallian for organizing the Duluth REU and for giving me a chance to participate in his 2021 program, and I am deeply grateful to Colin Defant and Noah Kravitz for suggesting these problems to me and for many helpful conversations over the course of that summer. In particular, I thank Colin Defant for carefully reading through a draft of this article, providing many productive comments on the manuscript. I also thank Yelena Mandelshtam for reviewing a draft of this work and providing constructive feedback. Finally, I would like to sincerely thank the two anonymous referees for thoroughly reviewing an earlier draft of this work, catching many typos and grammatical mistakes, and providing many detailed comments and suggestions which have vastly improved this article's clarity. Notably, I am immensely grateful to the anonymous referee who proposed a simple way to sharpen the main result of this paper. 

\section*{References}

\printbibliography[heading=none]

\end{document}